\documentclass[11pt]{article}

 \usepackage[utf8]{inputenc}
\usepackage[english]{babel} 
\usepackage{amsfonts,amsmath,amsxtra,amsthm,latexsym,mathabx}
\usepackage{graphicx}
\usepackage{epstopdf}
\usepackage{amssymb} 
\usepackage{bbm}
\newtheorem{thm}{Theorem}[section]
 \newtheorem{cor}[thm]{Corollary}
 \newtheorem{lem}[thm]{Lemma}
 \newtheorem{prop}[thm]{Proposition}

 \theoremstyle{definition}
 \newtheorem{defn}{Definition}[section]
 \theoremstyle{remark}
 \newtheorem{rem}{Remark}[section]

 \newtheorem{ex}[thm]{Example}
 
 \numberwithin{equation}{section}

\DeclareMathOperator{\im}{Im}
\DeclareMathOperator{\re}{Re}
\DeclareMathOperator{\dom}{dom}

\DeclareMathOperator{\Bd}{bd}

\DeclareMathOperator{\Intr}{int}

\DeclareMathOperator{\Dr}{Dr}
\DeclareMathOperator*{\argmin}{arg\,min}
\DeclareMathOperator{\Arg}{Arg}
\DeclareMathOperator{\Ln}{Ln}
\DeclareMathOperator{\GL}{GL}
\DeclareMathOperator{\SL}{SL}

\def\RR{\mathbb R}
\def\CC{\mathbb C}
\def\NN{\mathbb N}
\def\ZZ{\mathbb Z}
\def\TT{\mathbb T}
\def\BB{\mathbb B}
\def\DD{\mathbb D}

\def\al{\alpha}

\def\Th{\Theta}
\def\vphi{\varphi}
\def\la{\lambda}
\def\de{\delta}
\def\vep{\varepsilon}
\def\ep{\epsilon}
\def\ga{\gamma}

\def\om{\omega}
\def\Om{\Omega}
\def\Si{\Sigma}
\def\si{\sigma}

\def\vka{\varkappa}
\def\ka{\kappa}
\def\vth{\vartheta}

\def\pa{\partial}

\def\ii{\mathrm{i}}
\def\ee{\mathrm{e}}
\def\dd{\mathrm{d}}

\def\wh{\widehat}
\def\wt{\widetilde}

\def\modn{\hspace{-8pt}\mod}

\def\loc{\mathrm{loc}}

\def\even{\mathrm{even}}
\def\odd{\mathrm{odd}}
\def\sym{\mathrm{sym}}

\def\cent{\mathrm{centr}}
\def\edge{\mathrm{edge}}
\def\Contr{\mathfrak{C}}
\def\Reach{\mathfrak{R}}
\def\Hp{\mathrm{Hp}}

\def\I{\mathcal{I}}

\def\J{\mathcal{J}}
\def\Ls{L}
\def\Li{\mathfrak{L}}
\def\Ws{W}
\def\A{\mathcal{A}}

\def\F{\mathbb{F}}

\def\E{\mathcal{E}}
\def\fq{q_1}
\def\M{\mathcal{M}}

\def\Sec{\mathrm{Sec}}
\def\p{p}
\def\tp{t^+}
\def\ret{\mathfrak{r}}
\def\Sim{\mathcal{S}}

\def\n{\mathfrak{n}}

\def\H{\mathcal{H}}
\def\Pd{\mathrm{P}}
\def\Jmf{\mathfrak{J}}
\def\CJ{\mathcal{C}}
\def\Pa{\mathfrak{P}}
\def\Lr{L}
\def\Lh{\ell}
\def\md{\mathrm{mod}}

\def\ChiCpl{\chi_{_{\scriptstyle \CC_+}}}
\def\<{\langle}
\def\>{\rangle}
\def\EM{\mbox{\large $-$}}
\def\EP{\mbox{\large $+$}}
\def\Tr{\mathrm{Tr}}
\def\ra{r}

\def\nl{\mathrm{nl}}

\begin{document}
\title{Pareto optimization of resonances \\  and minimum-time control}
\author{}
\date{}
\maketitle

{\center 
{\large 
Illya M. Karabash$^{\, \text{a,b,*}}$, Herbert Koch$^{\, \text{a}}$, Ievgen V. Verbytskyi$^{\, \text{c}}$ 
\\[4mm]
}
}

{\small \noindent
$^{\text{a}}$
Mathematical Institute, Rheinische Friedrich-Wilhelms Universität Bonn,
Endenicher Allee 60, D-53115 Bonn, Germany \\[1mm]
 $^{\text{b}}$ Institute of Applied Mathematics and Mechanics of NAS of Ukraine,
Dobrovolskogo st. 1, Slovyans'k 84100, Ukraine\\[1mm]
$^{\text{c}}$ National Technical University of Ukraine ''Igor Sikorsky Kyiv Polytechnic Institute'',
Department of Industrial Electronics, Faculty of Electronics, Politekhnichna st. 16, block 12,
03056 Kyiv, Ukraine \\[1mm] 
$^{\text{*}}$ Corresponding author: i.m.karabash@gmail.com\\[1mm]
E-mails: i.m.karabash@gmail.com, koch@math.uni-bonn.de,  verbitskiy@bigmir.net
}

\begin{abstract}
The aim of the paper is to reduce one spectral optimization problem, which involves  the minimization of 
the decay rate $|\im \ka |$
of a resonance $\ka$, to a collection of optimal control problems on the Riemann sphere $\wh \CC$. 
This reduction allows us to apply methods of extremal synthesis to the structural optimization of 
layered optical cavities.
We start from a dual problem of minimization of the resonator length and give
several reformulations of this problem that involve Pareto optimization of the modulus $|\ka|$ of a resonance, 
a minimum-time control problem on $\wh \CC$, and associated Hamilton-Jacobi-Bellman equations. 
Various types of controllability properties are studied in connection with the 
existence of optimizers and with the relationship between the Pareto optimal frontiers of minimal decay 
and minimal modulus. We give explicit examples of optimal resonances and describe qualitatively properties 
of the Pareto frontiers near them. A special representation of bang-bang controlled trajectories is combined with 
the analysis of extremals to obtain various bounds on optimal widths of layers.
We propose a new method of computation of optimal symmetric resonators based on minimum-time control and compute with high accuracy several Pareto optimal frontiers and high-Q resonators.
\end{abstract}

{
\small \noindent
MSC-classes: 
49N35, 
35B34, 
49L25, 
78M50, 
49R05, 
93B27 
\\[0.5ex]

\tableofcontents
}

\normalsize

\section{Introduction}
\label{s:introduction}

\subsection{Resonance optimization and motivations for its study}
\label{ss:ResOpt}

The mathematical study of the problem of optimization of an individual resonance 
was initiated in the pioneering paper \cite{HS86}  with the aims to obtain 
an optimal bound on the resonance width and to estimate resonances of random Schrödinger Hamiltonians.
For the time-independent 3-dimensional (3-D) Schrödinger equation, 
the resonance width, roughly speaking, can be measured via the negative of the 
imaginary part of a resonance in the second `nonphysical' sheet of the two-sheeted
 Riemann surface for $\sqrt{\cdot}$,
the first sheet of which is the `energy plane' (see, e.g., \cite{RS78IV}).

The problem of minimization of resonance width falls in the class of nonselfadjoint spectral optimization problems,
which include also other types of optimization of transmission properties \cite{LShV03} and 
of eigenvalues of nonselfadjoint operators or matrices  \cite{CZ95,BE17}. Such problems 
are much less studied in comparison with the selfadjoint spectral optimization, 
which go back to the Faber-Krahn solution of 
Lord Rayleigh's problem on the lowest tone of a drum.  
We refer to \cite{EFH07,GCh13,KL18,vdBI13} for reviews and more recent studies of variational problems 
for eigenvalues of  selfadjoint operators 
and would like to note that some of these studies \cite{EFH07,GCh13,KL18} are directly or 
indirectly connected with resonance optimization, in particular, because 
square roots $\ka \in \ii \RR_+ := \{\ii c : c \in \RR_+\}$ of nonpositive eigenvalues $\ka^2$ are often considered to be 
resonances \cite{Z17} for associated selfadjoint operators.

Nonselfadjointness brings new difficulties into eigenvalue optimization. The two of the difficulties
discovered in \cite{HS86} were
connected with the existence of optimizers and with appearance of 
multiple eigenvalues (see also the discussions of these points in \cite[p. 425]{KS08} and the introductions to \cite{K13,K14}). 
 
During the last two decades variation problems for transmission and resonance effects 
attracted considerable attention in connection with active studies of photonic crystals \cite{AFKRYZ18,DLPSW11}
and high quality-factor (high-Q) optical cavities \cite{AASN03,LShV03,NKT08,BPSZ11}. 
The problem of high-Q design  was partially motivated  by 
rapid theoretical  and experimental advances in cavity quantum electrodynamics \cite{VLMS02}, in particular, in connection with 
`Schrödinger cat'-type experiments with cats replaced by  photons \cite{HR06}.

For the idealized model involving a layered optical cavity  and normally passing electromagnetic (EM) waves,
the Maxwell system can be reduced to the wave equation of a nonhomogeneous string 
$\vep(s) \pa_t^2 v (s,t) = \pa_s^2 v (s,t)$,  \ $s \in \RR$, 
where 
$\vep(\cdot)  
$ 
is the spatially varying dielectric permittivity of layers 
(assuming that the speed of light in vacuum is normalized 
to be equal to 1).
It is a step function that can take several positive values 
$\wh \ep_1$, \dots, $\wh  \ep_{m}$, 
corresponding to the materials available for 
fabrication. 
Mathematically, it is convenient to assume that $\vep (\cdot)$ is a uniformly positive $L^\infty (\RR)$-function. Its varying part in a finite interval $s \in [s_-,s_+]$, 
represents the nonhomogeneous 
structure of the resonator. 

Outside of this interval $\vep (\cdot)$ equals to the constant permittivity $\ep_\infty = \n^2_\infty$ 
(where $\n_\infty >0$) of the homogeneous outer medium, i.e.,
\begin{equation} \label{e:Fbinf}
\text{$\vep(s) = \n^2_\infty$ for a.a. $s \in \RR \setminus [s_-,s_+]$.}
\end{equation}
(In the sequel, 
we will omit `almost all` (a.a.) and `almost everywhere' (a.e.) where 
these words are clearly expected from the context.)

Resonances $\ka$ associated with the coefficient $\vep (\cdot)$   can be defined in several equivalent ways.
One can define resonances as poles of 
the meromorphic extension of the cut-off version of the  resolvent $\chi_{(s_-,s_+)} \left( - \frac{1}{\vep (s)} \pa_s^2 - \la \right)^{-1} \chi_{(s_-,s_+)}$,
where $\la = \ka^2$ and the extension is done from 
the upper complex half-plane $\CC_+ := \{\ka \in \CC :  \im \ka >0 \}$ to 
the whole plane $\CC$ (see \cite{Z17} and references therein), and then to associate them with zeros of a specially constructed 
analytic function \cite{AH84,CZ95,HBKW08,K13,KLV17}. 

This definition is equivalent to the time-harmonic 
generalized eigenvalue  problem 
\begin{gather}
y''(s)  =  -\ka^2 \vep(s) y (s)   \quad \text{ for a.a. $s \in \RR$},  \label{e:ep} 
\end{gather}
equipped with $\ka$-dependent outgoing (radiation, or damping) boundary conditions
\begin{align}
\frac{y' (s)}{\ka}  =  \pm \ii \n_\infty \  y(s)  \quad \text{for $s=s_\pm$}.
\label{e:BCpm}
\end{align}
Thus, a resonance associated with $\vep (\cdot)$ is 
a complex number $\ka \neq 0$ such that the (generalized) eigenproblem consisting of (\ref{e:ep}) and the two conditions 
$y' (s_\pm)/\ka =  \pm \ii \n_\infty \  y(s_\pm)$ 
has a nontrivial solution $y \in \Ws_\loc^{2,\infty} (\RR)$ (nontrivial means that $y$ is not identically $0$ 
in the $\Ls^\infty (\RR)$-sense).
Such a solution $y$ is called a (resonant) \emph{mode} associated with $\ka$ and $\vep (\cdot)$.

\begin{rem}\label{r:extension}
Note that for a nontrivial solution $y(\cdot)$ to (\ref{e:ep}),
equality (\ref{e:BCpm}) is satisfied for $s=s_\pm$ if and only if 
it is satisfied for certain $s $ such that $\pm s > \pm s_\pm $.
Indeed, in the both cases, $y(s)  =   C_\pm \exp (\pm \ii \n_\infty \  \ka s) $ for all 
$\pm s \ge \pm s_\pm $ with certain constants $C_\pm \neq 0$.
\end{rem}


The present paper employs the usual convention \cite{DEL10,K13,KLV17} that the value $\ka = 0$ is 
explicitly excluded from \emph{the set of resonances $\Si (\vep)$ associated with} $\vep (\cdot)$ 
(the ``zero resonance'' corresponds to the background level of the EM field \cite{CZ95}, which can be put to zero by a change of coordinates).
Then it is well known \cite{CZ95,KLV17} that 
\begin{gather} \label{e:SSym}
\text{$\Si (\vep) $ is a subset of $\CC_- $ symmetric with respect to (w.r.t.) the imaginary axis $\ii \RR$},
\end{gather}
where $\CC_- := \{ z \ : \ \im z <0 \}$.
The constant value of $\vep (\cdot)$ in the outer semi-infinite intervals $(-\infty,s_-)$ and $(s_+,+\infty)$ will be fixed in 
the process of optimization, and so 
$\Si (\vep)$ depends only on the part of the function $\vep (\cdot)$ in $[s_-,s_+]$, to which we refer simply as 
\emph{the resonator}. 

Resonances and the associated resonant modes, roughly speaking, describe the behavior of the wave-field $v$ inside 
the cavity for large times $t$ (see \cite{CZ95,Z17}).
The real part $\alpha=\re \ka$ and the negative  of the imaginary part $\beta= -\im \ka$ of a resonance $\ka$
correspond to the (real angular) \emph{frequency} 
and the (exponential) \emph{decay rate}, resp., of the eigenoscillations $e^{-\ii \ka t} y(s)$.
The value $(-2 \im \ka)$ is called the \emph{bandwidth} of a resonance $\ka$.

Simulations for perspective designs of  high-Q optical cavities \cite{AASN03,NKT08} 
have lead to the mathematical question \cite{KS08,HBKW08} of minimization of 
the decay rate $\Dr (\ka; \vep) := - \im \ka$ (or maximization of the $Q$-factor 
$Q = \frac{|\re \ka|}{- 2 \im \ka}$) for  a resonance 
$\ka \in \Si (\vep)$ by structural changes of $\vep (\cdot)$ assuming that $\vep (\cdot)$ satisfies (\ref{e:Fbinf}) and certain constraints in $[s_-,s_+]$.

\subsection{Review of decay rate minimization and basic definitions}
\label{ss:Rew}

Mathematically,
it is convenient to consider minimization of $\Dr (\ka;\vep)$ over the relaxed family $\F_{s_-,s_+}$ of feasible resonators $\vep (\cdot)$
that consists of all $\Ls^\infty (s_-,s_+)$-functions  
(rigorously, of all $\Ls^\infty (s_-,s_+)$-equivalence classes)
satisfying the constraints 
\begin{gather} \label{e:ep1<ep<ep2}
\text{$\n_1^2 \le \vep (s) \le \n_2^2 $  for 
$s \in (s_-,s_+) $, 
}
\end{gather}
where $0 < \n_1 <\n_2 $ and $\ep_1 =\n_1^2$ ($\ep_2 = \n_2^2$) is the minimal (resp., maximal) of the admissible permittivities.
Note that $\n_1$, $\n_2$, and $\n_\infty$ are the \emph{refractive indices} in the corresponding media (the magnetic permeability is assumed to be 1 in all materials under consideration).

Over the feasible family $\F_{s_-,s_+}$ the problem is well-posed in the Pareto sense of the paper \cite{K13}, i.e., 
the minimal decay rate $\beta_{\min} (\alpha)$ for a frequency $\alpha \in \RR$ is defined by
\begin{align} \label{e:betaminI}
 \beta_{\min} (\alpha) := \ \inf \{ \beta \in \RR \ : \ \alpha - \ii \beta  \in \Si [\F_{s_-,s_+}]  \} , \quad \al \in \RR,
 \end{align}
where $ \Si [\F_{s_-,s_+}] := \bigcup_{\vep \in \F_{s_-,s_+} } \Si (\vep) $
is the \emph{set of achievable resonances}.
(Here and below we use the convention that $\inf \varnothing = + \infty$.)
An achievable resonance $\ka_0$ is called the \emph{resonance of minimal decay for (the frequency)} 
$\al_0 = \re \ka_0$ if $\ka_0$ belongs to the \emph{Pareto frontier of minimal decay} 
\[
\Pa_{\Dr} := \{ \al - \ii \beta_{\min} (\al) \ : \ \al \in \dom \beta_{\min} \},
\] where 
$\dom \beta_{\min} := \{ \al \in \RR \ : \ \beta_{\min} (\al) < +\infty \}$ is the \emph{set of achievable frequencies}
(for general theory of Pareto optimization, see \cite{BV04}).
If $\ka_0 \in \Pa_{\Dr} $, then $\vep \in \F_{s_-,s_+}$ such that $\ka_0 \in \Si (\vep)$ is called the \emph{resonator 
of minimal decay for} $\al_0$.

The compactness argument \cite{HS86} implies that $\Si [\F_{s_-,s_+}]$ is closed \cite{K13,KLV17}. So 
for every achievable frequency $\al$, there exists a resonance of minimal decay $\ka = \al - \ii \beta_{\min} (\al)$ and an associated optimal resonator  $\vep (\cdot) \in \F_{s_-,s_+}$  generating $\ka$ (uniqueness is discussed in Section \ref{s:dis}). Other  definitions of optimizers were considered in \cite{HS86,OW13,K14,KLV17}. 

Most of research for 1-D photonic crystals  have been done under the assumption 
 that 
\begin{gather} \label{as:sym}
\text{$\vep (\cdot)$ is symmetric w.r.t.  the \emph{resonator center} $s^\cent = \tfrac{s_-+s_+}{2}$ (see \cite{KS08,NKT08,K13,OW13})}
\end{gather}
in the sense that $\vep (\cdot - s^\cent)$ is an even function.
For such symmetric $\vep (\cdot)$, the \emph{resonant mode} 
$y (\cdot)$ is either an even, or odd function w.r.t. $s^\cent$ (i.e., $y(\cdot - s^\cent)$ 
is even, or odd), and therefore satisfies 
\begin{gather} \label{as:EvenOdd}
\text{either the condition $y'(s^\cent) =0$, or the condition $y(s^\cent) = 0$. }
\end{gather} 
These conditions can be treated as boundary conditions and 
can be used to simplify the problem. Shifting $s^\cent$ to zero and getting $s_+ = - s_- = \Lh$ with a certain $\Lh>0$, 
we introduce the family 
\[
\F_{\Lh} ^\sym = \{ \vep \in \F_{-\Lh,\Lh} \ : \ \vep(s) = \vep(-s) \text{ a.e.}\} 
\]
and, for $\vep \in \F_{\Lh}^\sym$, introduce the set $\Si^{\even} (\vep)$ (the set $\Si^\odd (\vep)$) of resonances $\ka$  such that 
the corresponding mode $y $ is an even (resp., odd) function.
We will say that the corresponding $\ka$ is an \emph{even-mode} resonance (resp., \emph{odd-mode} resonance) of $\vep (\cdot)$.

Replacing $\Si [\F_{s_-,s_+}] $ in (\ref{e:betaminI}) by the closed sets of achievable even- and odd-mode resonances 
$\Si^{\even (\odd)} [\F_{\Lh}^\sym] := \bigcup_{\vep \in \F_{\Lh}^\sym } \Si^{\even (\odd)} (\vep) $ \cite{KLV17}, one obtains the corresponding functions $\beta_{\min}^{\even (\odd)}$ and 
Pareto optimal frontiers 
\[
\Pa^{\even (\odd)}_{\Dr} := \{ \al - \ii \beta_{\min}^{\even (\odd)} (\al) \; : \; \al \in \dom \beta_{\min}^{\even (\odd)} \}
\]
of even- and odd-mode resonances of minimal decay (see Fig. \ref{f:Par} (b)).

\begin{figure}[h]
{\footnotesize\textbf{a)}}\includegraphics[width=0.47\textwidth]{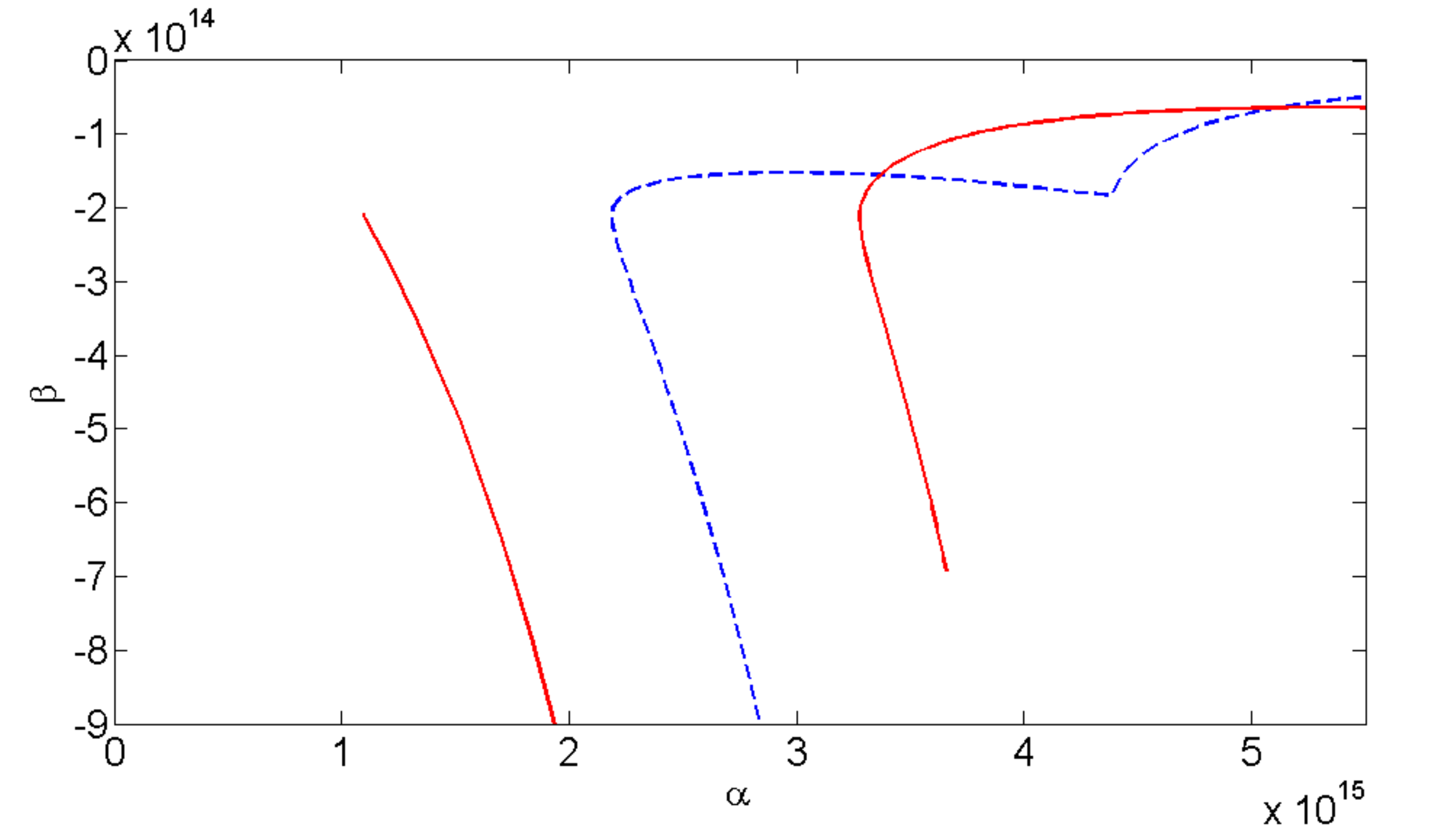}    {\footnotesize\textbf{b)}}\includegraphics[width=0.472\textwidth]{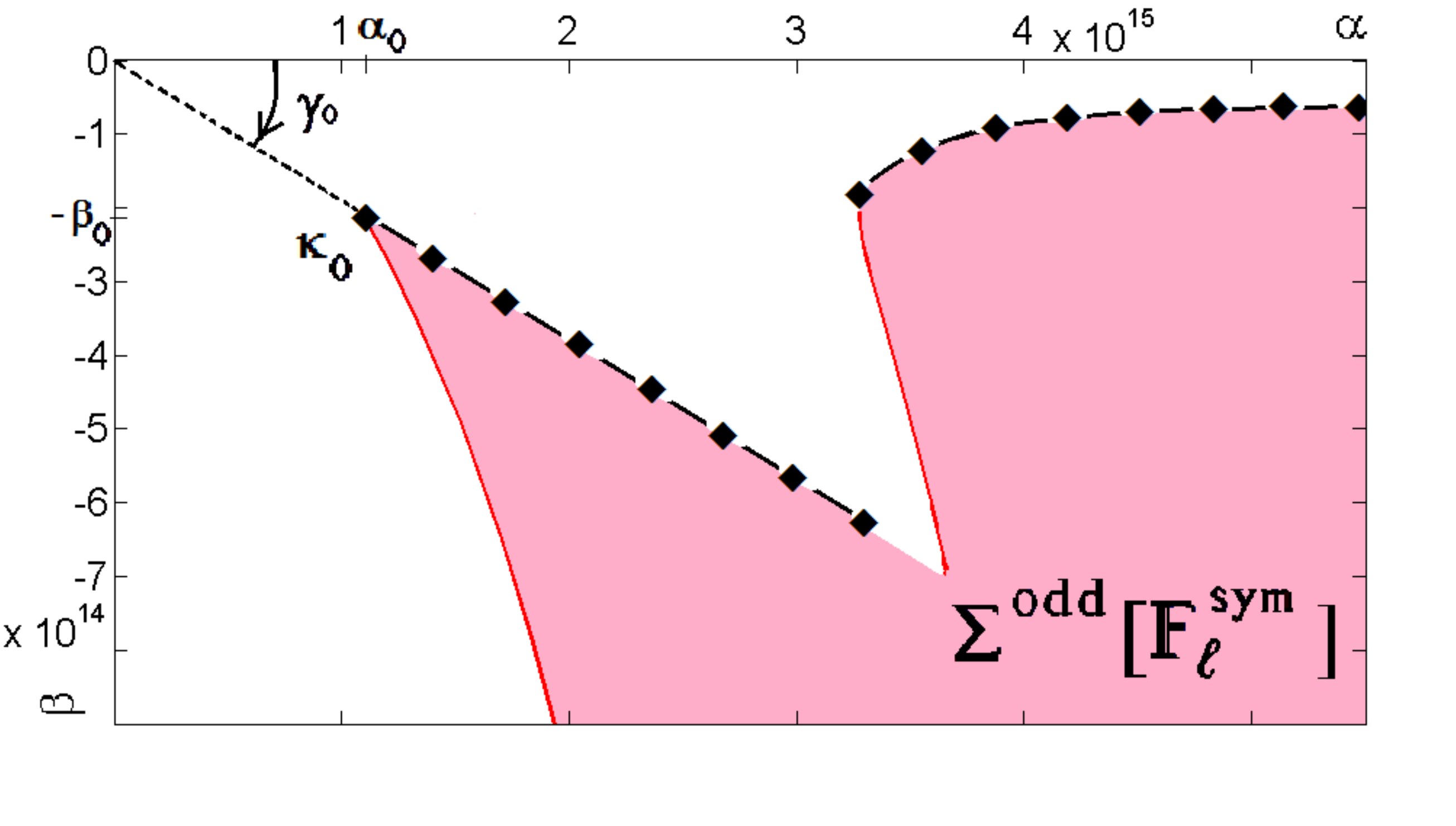} 

\caption{\label{f:Par} \footnotesize
\textbf{(a)} The computed parts of Pareto frontiers of even-mode resonances $\Pa^\even_{\md} := \{ \ee^{\ii \ga}\rho_{\min} (\ga, 0,\n_\infty)  :  \ga \in \dom \rho_{\min} (\cdot, 0,\n_\infty) \}$ (blue dashed line) and odd-mode resonances $\Pa^\odd_{\md} :=\{\ee^{\ii \ga} \rho_{\min} (\ga, \infty, \n_\infty)  :  \ga \in \dom \rho_{\min} (\cdot, \infty, \n_\infty) \}$ (red solid line) of minimal modulus (see Sections \ref{ss:Min|om|} and \ref{ss:Pareto})
for $\n_1 = \n_\infty = 1 $,
$\n_2 = 3.46$, and $\Lh = 0.1243\cdot10^{-6}$. 
 \ \textbf{(b)} The based on (a) and Theorem \ref{t:PF} drawing of the set $\Si^{\odd} [\F_\Lh^\sym] $ of achievable odd-mode resonances.  
The line marked  `\textbf{--$\blackdiamond$--}' shows the corresponding Pareto frontier of minimal decay $\Pa_{\Dr}^{\odd} := \{\al - \ii \beta_{\min} (\al)  :  \al \in \re \Si^{\odd} [\F_{\Lh}^{\sym}]\}$. The parts $\Pa_{\Dr}^{\odd}$ and $\Pa^\odd_{\md}$ of the boundary $\Bd \Si^{\odd} [\F_{\Lh}^{\sym}] $  do not coincide. Theorem \ref{t:ExOpt2} analytically proves that $\ka_0 =\al_0 - \ii \beta_0 = \ee^{\ga_0} \rho_0 $ is Pareto optimal and that there exists a jump of $\Pa^\odd_{\md}$ at $\ka_0$. 
}
\end{figure} 

\subsection{Known facts and open questions about the structure of optimizers}

It was noticed in the numerical experiments of \cite{KS08,HBKW08} that
for the relaxed problem only two extreme values $\n_1^2$ and $\n_2^2$ of $\vep (\cdot)$ appears from a certain iteration of the steepest ascent simulations.
This fact was analytically proved in \cite{K13} for Pareto optimal resonators of minimal decay.
It was shown in  \cite{K13} that a resonator of minimal decay switches between $\n_1^2$ and $\n_2^2$
according to a nonlinear eigenvalue problem with a special bang-bang term (see equation (\ref{e:Eys}) in Section \ref{s:Syn} and also \cite{KLV17}, where this result was extended on all resonators $\vep (\cdot)$ generating $\ka \not \in \ii \RR$
on the boundary $\Bd \Si [\F_{s_-,s_+}]$).
This eigenproblem can be seen as an analogue of the Euler–Lagrange equation.

The steepest ascent numerical experiments of \cite{KS08,HBKW08,OW13} and the shooting 
method for the aforementioned Euler-Lagrange eigenproblem on a fixed interval  \cite{KLV17} suggest that  Pareto optimizers $\vep (\cdot)$ are 
close to structures that consist of periodic repetitions of two layers of permittivities $\n_1^2$ and $\n_2^2$ with a possible introduction of defects. 
This is in good agreement with Photonics studies of high-Q cavities \cite{AASN03,NKT08,BPSZ11}, 
which are based on sophisticated designs of defects in periodic Brag reflectors,
and with Physics intuition, which says that oscillations 
of EM field are expected to accumulate on the defects if 
their frequency $\re \ka$ is in a stopband of the Brag reflector.

However, presently available analytic and numerical methods do not give a clear answer: (i) about the structure of optimal defects and the lengths of alternating layers with permittivities $\n_1^2$ and $\n_2^2$ in the original unperturbed Brag reflector (see Section \ref{s:dis}), (ii) about  symmetry and uniqueness of optimizers (the last two questions are obviously connected, see the discussion in \cite{AK17} and Section \ref{s:dis}), (iii) about the existence of a global minimizer of the decay rate $\Dr (\cdot;\cdot)$ without constraints imposed on the frequency (on the base of numerical experiments of \cite{KS08,HBKW08,OW13} it it was conjectured in \cite{OW13,KLV17} that $\argmin_{\substack{ \vep \in \F_{s_-,s_+} \\ \ka \in \Si (\vep) \quad}} \Dr (\ka;\vep) = \varnothing$; here and below we use the standard $\argmin$ notation \cite{BV04} for the set of minimizers).

These open questions show that new theoretical tools and more accurate numerical approaches are needed to understand the structure of optimal resonators.

To address these goals several other approaches were proposed. 
One of directions \cite{GCh13,LS13} suggests to study resonant properties with the help of certain associated selfadjoint spectral problems avoiding in this way the difficulties of nonselfadjoint spectral optimization. One more direction employs 
`solvable' models of Schrödinger operators with point interactions \cite{O14_Th,AK17}.

The aim of the present paper is to propose analytical and computational methods based a completely different idea.
We connect the problem of optimization of resonances with optimal control theory, more precisely,
with a collection of special minimum-time problems on the Riemann sphere 
$\wh \CC = \CC \cup \{\infty\}$ perceived as a smooth 2-D real manifold.

\textbf{Notation}. 
We use the convention that $\inf \varnothing = + \infty$.
The sets $\CC_1$, $\CC_2$, $\CC_3$, and $\CC_4$ are the open quadrants in $\CC$
 corresponding  to the combinations of signs $(+,+)$, $(-,+)$, $(-,-)$, and $(+,-)$ for $(\re z,\im z)$.
Other sets used in the paper are: the 
compactifications $\wh \RR = \RR \cup \{\infty\}$ and $\wh \CC = \CC \cup \{\infty\}$ perceived as smooth 1-D and 2-D real  manifolds, resp.,
open half-lines $\RR_\pm = \{ s \in \RR: \pm s >0 \}$ and 
half-planes $\CC_\pm = \{ z \in \CC : \pm \im z >0 \}$, 
closed line-segments $[z_1,z_2] := \{ (1-t) z_1 + t z_2 \ : \ t \in [0,1]\}$ (which we call $\CC$-intervals),
$\CC$-intervals $(z_1,z_2) = [z_1,z_2] \setminus \{ z_1,z_2 \} $
with excluded endpoints $z_{1,2} \in \CC$, 
circles $\TT_\de (\zeta) := \{z \in \CC : | z - \zeta | = \de \}$ with $\de >0$, $\TT = \TT_1 (0)$,
open discs $\DD_\de (\zeta) := \{z \in \CC : | z - \zeta | < \de \}$ 
with $\de \ge 0$ and $\zeta \in \CC$, $\DD = \DD_1 (0)$,
the infinite sector (without the origin $z=0$)
\begin{equation} \label{e Sec def}
\Sec (\xi_1,\xi_2) := \{ c e^{\ii \xi} \, : \, c > 0  \text{ and } \xi \in (\xi_1, \xi_2) \} ,  \ \ \ \xi_1 < \xi_2 , \ \ \ \xi_{1,2} \in \RR .
\end{equation} 

For a normed space $W$ over $\CC$ with a norm $\| \cdot \|_W$, 
 $ \BB_\de (w_0; W) := \{w \in W \, : \, \| w - w_0 \|_W < \de \}$ 
are open nonempty balls with $w_0 \in W$ and $\de >0$.
For $E \subset W$ and $z \in \CC$, we write $z E  + w_0 := \{ z w + w_0 \, : \, w \in E \} $.
The closure (the interior) of a set $E $ 
in the norm topology is denoted by $\overline{E}$ (resp., $\Intr E$).


For a function $f$ defined on a set $E$, $f[E]$ is the image of $E$ and, in the case when $f$ maps to $(-\infty,+\infty]$,
the domain of $f$ is $\dom f := \{ s\in E \ : \ f (s) < +\infty\}$.
A line over a complex number $z$ or over a $\wh \CC$-valued function denotes complex conjugation, i.e.,
$\overline y $ means that $y(s) = \overline{y(s)}$; 
$\overline{\infty} = \infty$. We simplify the notation $(f(s))^j =: f^j (s)$, while $f^{\{-1\}} (\cdot)$ is the inverse function.
By $\pa_s f$, $\pa_z f $, etc., we denote (ordinary or partial)
derivatives w.r.t. $s$, $z$, etc.;
 $\vep (s \pm 0) $ are the one-side limits of a function $\vep$ at a point $s \in \RR$.
For an interval $I \subset \RR$, $ \Ls^p  (I)$ and
$
\Ws^{k,p}  (I) := \{ y \in  \Ls^p  (I) \ : \ \pa_s^j y \in  \Ls^p (I), \  j \leq k \}
$
are the complex Lebesgue and Sobolev spaces with standard norms $\| \cdot \|_p$ and $\| \cdot \|_{\Ws^{k,p}}$. To denote the corresponding Lebesgue and and Sobolev spaces of real-valued functions, we use subscript $\RR$ ($\Ls^p_\RR$ , etc.).  The space of continuous complex valued functions with the uniform norm is denoted by $C [s_-,s_+]$.
The $\loc$-notation $y \in \Ws_{\loc }^{k,p} (\RR)$ means that $y \in \Ws^{k,p} (s_-,s_+)$ for every finite interval $(s_-,s_+) \subset \RR$ (the same is applied to the space $C_\loc (\RR)$).


By $z^{1/2} $, $\Arg_0 z$, $\Ln z$  we denote the continuous in 
$\CC \setminus \overline{\RR_-}$ 
branches of multi-functions $\sqrt{z}$, complex argument $\Arg z$,
and complex natural logarithm $\ln z$
fixed by $1^{1/2} =1$ and $\Ln 1 = \ii \Arg_0 1 = 0$. 
For the multifunction $\sqrt{\cdot}$ we use also the notation $[\cdot]^{1/2}$. 
For $z \in \RR_- $, we put $\im \Ln z = \Arg_0 z :=  \pi$ and $z^{1/2} := \ii (-z)^{1/2}$.
The characteristic function of a set $E$ is denoted by $\chi_E (\cdot)$, i.e.,
$\chi_E (s) = 1$ if $s \in E$ and $\chi (s) = 0$ if $s \not \in E$.


\section{Overview of methods and results of the paper}
\label{s:Results}

As a starting point, we introduce a dual problem of minimization of length of a resonator $\vep (\cdot)$ under the assumption that $\vep (\cdot)$ produces a given resonance $\ka \in \CC_-$. 

Since the definition of the length of a resonator is ambiguous (see Remark \ref{r:extension}), consider an additional \emph{feasible  family $\F $ of  (permittivity) coefficients}. The family $\F$ consists of positive functions 
$\vep (\cdot) \in \Ls_\RR^{\infty} (\RR)$ such that there exists 
$s_\pm \in \RR$ satisfying $s_-  \le s_+ $ and conditions (\ref{e:Fbinf}), (\ref{e:ep1<ep<ep2}). 
We consider also the subfamily $\F^\sym$ of symmetric $\vep$, i.e.,
\[
\F^\sym = \{ \vep \in \F \ : \ \vep(s) = \vep(-s) \} .
\]

\begin{defn} \label{d:length}
For any given $\vep(\cdot)$ that is not equal to the constant function $\n^2_\infty$,
we denote by $[s_-^{\vep},s_+^{\vep}]$ the shortest interval $[s_-, s_+]$ satisfying (\ref{e:Fbinf}), 
and by $\Lr (\vep) := s_+^{\vep} - s_-^{\vep} $
\emph{the effective length of the resonator} defined by the coefficient $\vep(\cdot)$. 
If $\vep(\cdot) = \n^2_\infty$ (in $L^\infty (\RR)$-sense),
we put $s_-^{\vep}:=0$, $s_+^{\vep} := 0$, and $\Lr (\vep) = 0$. 
\end{defn}

Let us consider the two following length minimization problems
\begin{gather} \textstyle
\argmin_{\substack{ \vep \in \F \quad \\ \ka \in \Si (\vep)}} \Lr (\vep)  ,
\label{e:argminF} 
\\
\textstyle
\argmin_{\substack{ \vep \in \F^\sym \\ \ka \in \Si (\vep)}} \Lr (\vep) ,
\label{e:argminFsym}
\end{gather}
where the resonance $\ka \in \CC_- \setminus \{0\}$ and the constraint parameters $\n_\infty $, $\n_1$, $\n_2 $ 
are fixed.

\begin{rem} \label{r:sym}
The \emph{symmetric optimization problem} (\ref{e:argminFsym}) 
can be split into two, the \emph{odd-mode} and \emph{even-mode} problems in the way described in  Section \ref{ss:Rew}. 
Note that $\Si^\odd (\vep) \cap \Si^\even (\vep) = \varnothing$ (otherwise, the odd- and even-modes 
form a fundamental system of solutions and at least one of them does not satisfy (\ref{e:BCpm})).
Hence, $\vep(\cdot)$ is a minimizer for (\ref{e:argminFsym}) if and only if it is a minimizer 
for exactly one of the two problems:
\begin{equation} \textstyle
\argmin_{\substack{\vep \in \F^\sym \quad \\ \ka \in \Si^{\odd } (\vep)}} \Lr (\vep) , 
 \qquad  \argmin_{\substack{\vep \in \F^\sym \quad \\ \ka \in \Si^{\even} (\vep)}} \Lr (\vep) .
\label{e:argminOddEven}
\end{equation}
\end{rem}

 One can define the corresponding minimum lengths by 
\[ \textstyle
\Lr_{\min} (\ka) :=   \inf_{\substack{\vep \in \F \\ \ka \in \Si (\vep)}} \Lr (\vep) , \quad
\Lr_{\min}^{\odd (\even)} (\ka) :=   
\inf_{\substack{\vep \in \F^\sym \\ \ka \in \Si^{\odd (\even)} (\vep)}} \Lr (\vep) .
\]


In Section \ref{ss:MinTime} we give equivalent reformulations of these 4 problems in terms of minimum-time control
for the system 
\begin{equation} \label{e:R}
x' (s) = f (x (s),\vep (s)), \text{ with $f(x,\ep) := \ii \ka (- x^2  + \ep )$ }
\end{equation}
where $\vep(\cdot)$ is interpreted as a \emph{control},
 $x(s)$ evolves in the state space $\wh \CC = \CC \cup \{\infty\}$ and is connected with a solution $y(s)$ to (\ref{e:ep}) by 
$x(s) = \frac{y'(s)}{\ii \ka y(s)}$  \quad (where $x(s)=\infty$ if  $ y(s) = 0$)
The trajectory $x(\cdot)$ of (\ref{e:R})  blows-up in the time-like points $s$ such that $y(s) = 0$.
The evolution of $x$ in the neighborhood $\wh \CC \setminus \{0\}$ of $\infty$ can be described in the following way:
\begin{gather} \label{e:R2}
\text{$\wt x (s) = -1/x(s) $ \ satisfies \ 
$\wt x'  = \wt f (\wt x , \vep)$, \ where \ $\wt f(\wt x,\ep):=\ii \ka (-1 + \ep  \wt x^2 )$.}
\end{gather} 

With such settings the outgoing boundary condition (\ref{e:BCpm})
at $s_-$ (at $s_+$) become the initial state value $x(s_-) = - \n_\infty$ (resp., the terminal state $x(s_+) = \n_\infty$). The problem (\ref{e:argminF}) turns into the problem of minimum-time control from $(-\n_\infty)$ to $\n_\infty$.
The symmetric problems (\ref{e:argminOddEven}) are essentially the problems of minimum-time control  from $\infty$ and from $0$ to $\n_\infty$ (or, equivalently, from $(-\n_\infty)$ to $\infty$ and to $0$).

Note that the resonator optimization in the case $\ka \in \ii \RR$ 
is simpler because it 
has some features of selfadjoint spectral optimization \cite{K13,KLV17_MFAT}. 
Besides,  the minimum-time control problem for the value $\ka$ of the spectral parameter is equivalent to that for the value $(-\overline{\ka})$ (see (\ref{e:SSym}) and Section \ref{ss:Monotonicity}). 
Therefore a substantial part of the paper is focused on the case $\ka \in \CC_4
 :=\{ z \in \CC : \re z >0, \im z <0\}$.

We return from the dual problem to the original problem of minimization of the decay rate stated in  Section \ref{ss:Rew} in several steps:
\begin{itemize}
\item[(i)] Global controllability and small time local controllability (STLC) of system (\ref{e:R}), (\ref{e:R2}) are studied in Section \ref{s:existence}. It is important for our needs to ensure that these properties are locally uniform w.r.t. the spectral parameter $\ka$. So, actually, we study the whole collection of control systems (\ref{e:R}), (\ref{e:R2}) indexed by $\ka \in \CC_4$.

\item[(ii)] We give another reformulation for the minimum-time control of (\ref{e:R}), (\ref{e:R2}) in terms of Pareto minimization of the modulus $| \ka |$ of a resonance for a resonator $\vep \in \F_{s_-,s_+}$. 
This gives the second class of Pareto frontiers $\Pa_{\md}$ and $\Pa_{\md}^{\odd (\even)}$, which consist of achievable $\ka$ with minimal possible modulus $|\ka|$ for a given complex argument $\Arg \ka$ (see Section \ref{ss:Min|om|} and Fig. \ref{f:Par} (a)).

\item[(iii)]  Combining the results obtained in steps (i) and (ii), we show in Theorems \ref{t:MinDec1}-\ref{t:MinDec2} that under the assumption $\n_1 \le \n_\infty \le \n_2$ the original Pareto problem of  minimization of the decay rate can be ``partially'' reduced to the minimum-time control problem (``partially'' in the sense that at least one minimizer $\vep (\cdot)$ for each point $\ka \in \Pa_{\Dr}$ is obtained, but possibly not all the minimizers). In the case $\n_1 < \n_\infty < \n_2$, the above reduction is complete.  
\end{itemize}

Note that the relation $\n_1 \le \n_\infty \le \n_2 $ between the constraint  parameters is reasonable from the applied point of view (it means that the permittivity of the outer medium is in the range of permittivities admissible for fabrication,  see \cite{KLV17}).

There are several advantages of the reduction to the minimum-time control:
\begin{itemize}
\item From the analytical point of view, this reduction 
brings the tools of Optimal Control to the study of optimal designs of resonators. This includes Hamilton-Jacobi-Bellman (HJB) equations \cite{BC08,SL12},  their theory on manifolds \cite{CV03}, Pontryagin Maximum Principle (PMP), and the concepts of extremal synthesis \cite{BP03,SL12}. 

\item From numerical point of view, we propose in Section \ref{s:N} a method of accurate computation of optimal symmetric resonators by the minimum-time shooting to the turning point. The turning point $p_0$ of a mode $y(s)$ is the $s$-point, where the trajectory of $y$ changes of the direction of its rotation around $0$ in the complex plane \cite{KLV17}. So the most important step of the method, the computation of $p_0$, is essentially the computation of a zero of a  monotone function (see Lemma \ref{l:MonotFunc}). This allows us to compute the widths of layers of optimal resonators with high accuracy.
\end{itemize} 

The HJB equation associated with the value function
$V (x) := T^{\min}_\ka (x,\n_\infty) $ is $0 = 1 - \max \{ - \nabla_{f(x,\ep)} V (x) \ : \ \ep = \ep_j, \ j=1,2 \}$. Here $T^{\min}_\ka (x,\n_\infty)$ is the minimum time needed to reach $\n_\infty$ from the initial state $x \in \wh \CC$,   
$\nabla_z V(x) := \lim_{\substack {\zeta \to 0 \\ \zeta \in \RR}} \frac{V(x+\zeta z)-V(x)}{\zeta}$ is the directional derivative and  the direction $z \in \CC$ is perceived as an $\RR^2$-vector.
The HJB equation can be formally transformed to the boundary value problem  
\begin{align} \label{e:HJB}
0= V (\n_\infty), \quad 
0 = 1 - \nabla_{\ii \ka (x^2 - (\ep_2+\ep_1)/2)} V (x) - \frac{\ep_2 -\ep_1}{2} |\nabla_{\ii \ka} V (x)| , \quad x \in \CC \setminus \{\n_\infty\},
\end{align}
However, in this form the boundary condition at $x=\infty$ is not specified because it is encoded in the HJB equation for (\ref{e:R2}) in the $\wt x$-state space.

That is why we employ in Section \ref{ss:HJB} the manifold state-space approach of \cite{CV03} to formulate rigorously  the HJB equation and the corresponding existence and uniqueness theorem for a proximal solution on the 2-D real manifold $\wh \CC$. 
This result, Corollary \ref{c:V}, is written as a symmetrical coupling of two HJB boundary value problems for the backward and forward value functions because we believe that it is right to emphasize time-reversal symmetry of system (\ref{e:R}), (\ref{e:R2}).
This and other symmetries of (\ref{e:R}), (\ref{e:R2}) are discussed in Section \ref{ss:Monotonicity}.

Some regularity properties of the backward value function $V$ can be seen from controllability results of Section \ref{s:existence}. If $\n_1 \le \n_\infty \le \n_2$, the domain $\dom (V) $ of $V$ is the whole state-space $\wh \CC$. If additionally $\n_1 < \n_\infty < \n_2$, then $V$ is continuous.

The structural analysis of optimal and extremal controls and trajectories is the subject of Sections \ref{s:bang-bang}-\ref{s:Syn2} (for the main principles and definitions of geometric optimal control, see \cite{AS13,BP03,SL12}). Our terminology and notation are oriented to the specific of resonance problems and somewhat differs of that of \cite{BP03,SL12}, while the ideas of extremal synthesis \cite{BP03,SL12} are behind many statements of Sections \ref{s:bang-bang}-\ref{s:Syn2}. 

Our analysis of extremals  consists of the following steps:
\begin{itemize}
\item The application of PMP in Theorem \ref{t:PMP} 
gives not only the absence of singular arcs and another derivation of the Euler-Lagrange eigenproblem of \cite{K13,KLV17},
but also an additional necessary condition $\im ( \vep (s) y^2 (s)+ \ka^{-2} (y' (s))^2 ) \ge 0$ of optimality of $\vep$,
which occurred to be important for the description of optimal normal extremals of Section  \ref{ss:Norm}. 

\item The combination of PMP with the result of \cite{KLV17} on rotation of resonant mode $y$ implies that optimal controls $\vep (\cdot)$ are of bang-bang type 
(and without chattering effects). 
The intervals $(b_j , b_{j+1})$ of constancy of $\vep (\cdot)$  can be interpreted from Optics point of view as layers.

\item We introduce the \emph{ilog-phase} $\vth $, which is function in $s$ that appears in special representations of $x$ and $y$ trajectories and is convenient from the point of view of iterative analytic calculation of the positions of switch points $b_j$ (see Sections \ref{ss:const} and \ref{ss:CalcSwitch}). The ilog-phase $\vth$ is used then to obtain estimates on the widths $b_{j+1} - b_j$ of layers  (Theorems \ref{t:AbnExtWidth}, \ref{t:NextWidth} and Corollaries \ref{c:Icent}, \ref{c:Norm}).

\item It is shown by Theorems \ref{t:AbnExtWidth}, \ref{t:NextWidth} and Section \ref{s:Syn2} that abnormal extremals correspond to resonators that have parts consisting of periodic alternation of quarter wave layers with permittivities $\n_1^2$ and $\n_2^2$.
Such structures are called quarter-wave stacks (see Remark \ref{r:1/4stack}) and are widely used in Photonics (in the context of resonance optimization they were discussed in \cite{OW13}).

 \item We study the discrete dynamical system $x(b_j) \mapsto x(b_{j+1}) $ that describe the evolution of values of extremal $x(\cdot)$ at the set $\{b_j\}_{j=-\infty}^{+\infty}$ of switch points. Our analysis shows that abnormal extremals can be optimal only for exceptional values of $\Arg \ka$ in $(-\pi/2,0)$ (see Lemma \ref{l:Mep12} and Corollary \ref{c:1/4per}). So the appearance of quarter-wave stacks pieces in the structures of Pareto optimal resonators is at least atypical.
 
\item From the point of view of structural classification of extremal trajectories 
\linebreak $x[[s_-,s^+]]:= \{x(s) : s \in [s_-,s_+]\} $, the above results are summarized in Section \ref{s:Syn2}. This classification is used then to prove the explicit examples of optimizers in 
Theorems \ref{t:ExOpt} and \ref{t:ExOpt2}.
\end{itemize}

While the construction of optimal synthesis for (\ref{e:R})-(\ref{e:R2}) is not finished in the present paper, the effectiveness of our analysis of extremals is demonstrated by several results:

\begin{itemize}
\item For optimizers $\vep(\cdot)$ of the symmetric optimization problems,
we rigorously prove in Section \ref{s:Corol} a series of necessary conditions on width of layers and permittivity of the central layer.

\item We analytically prove that in the case $\n_1 \le \n_\infty < \n_2$ a half-wave layer with permittivity $\n_2^2$ and  its first odd-mode resonance, which we denote here $\ka_0$, form an optimal pair $\{\ka_0 ,\vep \}$ simultaneously for odd-mode Pareto problems of minimal decay for $\re \ka_0$ and of minimal modulus for $\Arg \ka_0$ 
(see Theorem \ref{t:ExOpt}). 
This gives the first analytically derived example of a resonance belonging to a Pareto frontier outside of the imaginary line $\ii \RR$.

\item It is shown that the Pareto frontier $\Pa_{\md}^{\odd}$ of minimal modulus has a jump at the aforementioned $\ka_0$ 
(see Fig. \ref{f:Par}). The existence of this jump and qualitative descriptions of the Pareto frontiers near $\ka_0$ are given by Theorem \ref{t:ExOpt2}.
\end{itemize}

Optimal symmetric resonators with high quality factor $Q$ are studied numerically in Section \ref{s:N}. Conclusions from our analytical and numerical results are summarized in Section \ref{s:dis}, where we also compare them with high-Q designs of earlier studies.

\section{Three reformulations of the dual problem}
\label{s:DualRef}


\subsection{Minimum-time control problem and Riccati equations}
\label{ss:MinTime}

In this subsection, the variable $s$ will be interpreted as time, $\ka \neq 0$ is a fixed parameter.
Functions $\vep (\cdot) \in \Ls^\infty_{\RR} [s_-,+\infty)$ will be called \emph{controls}.

The family $ \F_{s_-} $ 
of \emph{feasible controls} is defined by 
\[
\F_{s_-} := \{ \vep(\cdot) \in \Ls_\RR^\infty (s_-,+\infty) \ : \ \n^2_1 \le \vep (s) \le \n^2_2 \text{ for } s > s_- \}.
\]
To modify the differential equation (\ref{e:ep}) into a control system in a state-space $\CC^2$,
one denotes $Y_0 (s) = y(s)$, $Y_1 (s)= \frac{y'(s)}{\ii \ka}$, forms a column vector $Y (s) = ( Y_0 ; Y_1 )^\top \in \CC^2$, 
and write (\ref{e:ep}) as 
$ Y' (s) = \ii \ka \begin{pmatrix} 0 & 1 \\
\vep (s) & 0 
\end{pmatrix}  
Y(s) $.

Since this system is linear one can consider the associated dynamics on 
the complex projective line, which can be identified with the Riemann sphere $\wh \CC = \CC \cup \{ \infty \}$.
From the  point of view of elementary ODEs, 
this is the standard reduction to the associated 
Riccati differential equation.
Namely, for a nontrivial solution  $y(\cdot)$ to (\ref{e:ep}), the dynamics 
of the function $x(\cdot)$ defined by 
$x := y'/ (\ii \ka y)$  
is described in $\CC$ by the control system  $x'=\ii \ka (- x^2  + \ep )$.
The solution $x(\cdot)$ of (\ref{e:R})  blows-up in the time-points $s$ such that $y(s) = 0$.
The dynamics of $x$ near $\infty$ is given by (\ref{e:R2}).

Recall that a state $x(s_-) $ of the 
system (\ref{e:R})-(\ref{e:R2}) is said to be in 
the \emph{time-$t$-controllable set} $\Contr_{\{t\}} (\eta_+,\ka)$
 to a state $\eta_+$ 
 if there exists a feasible control $\vep \in \F_{s_-}$ such that 
$x(s_- +t ) = \eta_+ $. 
For $t \in [0,+\infty]$, we put
$\Contr_{[0,t)} (\eta_+,\ka) := \bigcup_{0 \le t_0 < t} \Contr_{\{t_0\}} (\eta_+,\ka) $,
and write $\eta_+ \in \Reach_{[0,t)} (\eta_-,\ka)$ if $\eta_- \in \Contr_{[0,t)} ( \eta_+,\ka)$.
A state $\eta_1 $ is said to be \emph{reachable from} $\eta_2$ 
(resp., \emph{controllable to} $\eta_2$)
if it belongs to the set $\Reach_{[0,+\infty)} (\eta_2) $ (resp., $\Contr_{[0,+\infty)} (\eta_2)$).

A feasible control $\vep \in \F_{s_-}$ is said to be a \emph{minimum-time control}
from $x(s_-)$ to $\eta_+$ if $x(s_- + t ) = \eta_+$ in the minimal possible time
$t=T^{\min}_\ka (x(s_-),\eta_+)$, which can be defined by 
\begin{align} \label{e:Tmin}
T^{\min}_\ka (x(s_-),\eta_+) := \inf \{t \ge 0 \ : \ x(s_-) \in \Contr_{\{t\}} (\eta_+,\ka) \} .
\end{align}
If  $\eta_-$ is not controllable to $\eta_+$, we put by definition 
$T^{\min}_\ka (\eta_-,\eta_+) := +\infty$.

For $ (\eta_- ; \eta_+) \in \wh \CC^2$, 
we say that $\ka \in \CC \setminus \{0\}$ is an $ (\eta_- ; \eta_+)$-eigenvalue of $\vep (\cdot)$
on an interval $(s_-,s_+)$ if equation (\ref{e:ep}) has a nontrivial solution $y$ satisfying 
the two boundary conditions 
\begin{gather}\label{e:BCxpm} 
\frac{y' (s_\pm)}{\ii \ka y(s_\pm) }  =  \eta_\pm  \quad
\text{(which are understood as $y(s_\pm) = 0$ when $\eta_\pm =\infty$)} .
\end{gather} 
We denote the set of $ (\eta_- ; \eta_+)$-eigenvalues by $\Si_{\eta_-,\eta_+}^{ s_-, s_+} (\vep) $.

One sees that the following statements are equivalent for $t>0$:
\begin{itemize}
\item[(C1)] $\ka \in \Si_{\eta_-,\eta_+}^{ s_-, s_+} (\vep)$, where $s_+ = s_- +t$;
\item[(C2)] $\ka \neq 0$ and 
the control $\vep (\cdot)$ steers system (\ref{e:R})-(\ref{e:R2})
from $\eta_-$ to $\eta_+$ in time $t$.
\end{itemize}

\begin{prop} \label{p:eqv} For a fixed $\ka \neq 0$, 
the following minimization problems are equivalent.

\item[(i)] Problem (\ref{e:argminF}) is equivalent to the minimum-time control  of 
(\ref{e:R})-(\ref{e:R2}) 
from $(-\n_\infty)$ to $\n_\infty$
in the sense that 
$\Lr_{\min} (\ka) = T^{\min}_\ka (-\n_\infty,\n_\infty)$
and the two families of $L^\infty (s_-^{\vep},s_+^{\vep})$ functions 
produced by the restriction of minimizers $\vep (\cdot)$ for these two problems 
to the interval $(s_-^{\vep},s_+^{\vep})$ coincide after a suitable shift of $s_-$ to $s_-^{\vep}$. 

\item[(ii)] The odd-mode (even-mode) problem in (\ref{e:argminOddEven}) is equivalent to 
the minimum-time control of (\ref{e:R})-(\ref{e:R2}) 
from $\infty$ (resp., from $0$) to $\n_\infty$  in the sense that 
$\Lr_{\min}^{\odd} (\ka) = 2 T^{\min}_\ka (\infty, \n_\infty)$ 
(resp., $\Lr_{\min}^{\even} (\ka) = 2 T^{\min}_\ka (0,\n_\infty)$) and
their minimizers $\vep (\cdot)$ can be identified by a suitable shift of $s_-$ to $s_-^{\vep}$.
\end{prop}

\begin{figure}[h]
\begin{minipage}[t]{0.495\linewidth}
\centering \includegraphics[width=\textwidth]{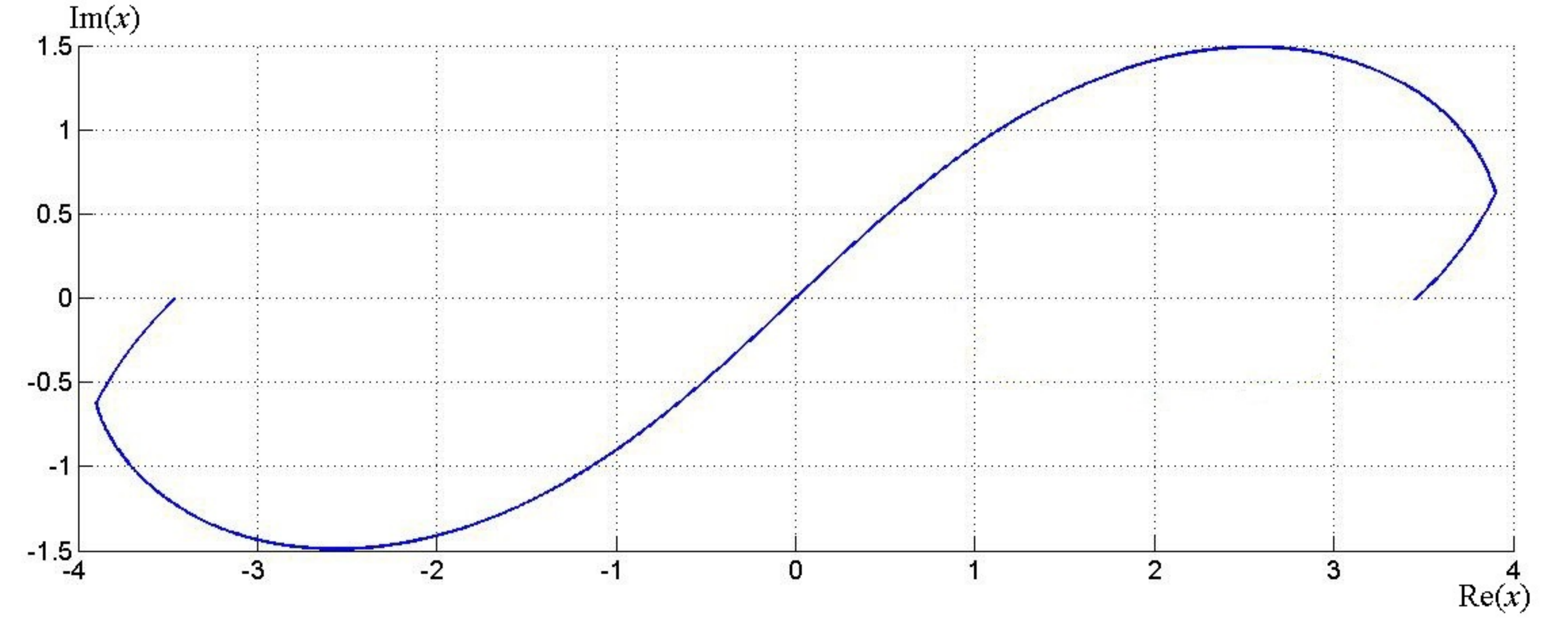}
\footnotesize \textbf{a)}
\end{minipage}
\begin{minipage}[t]{0.495\linewidth}
\centering
\includegraphics[width=\textwidth]{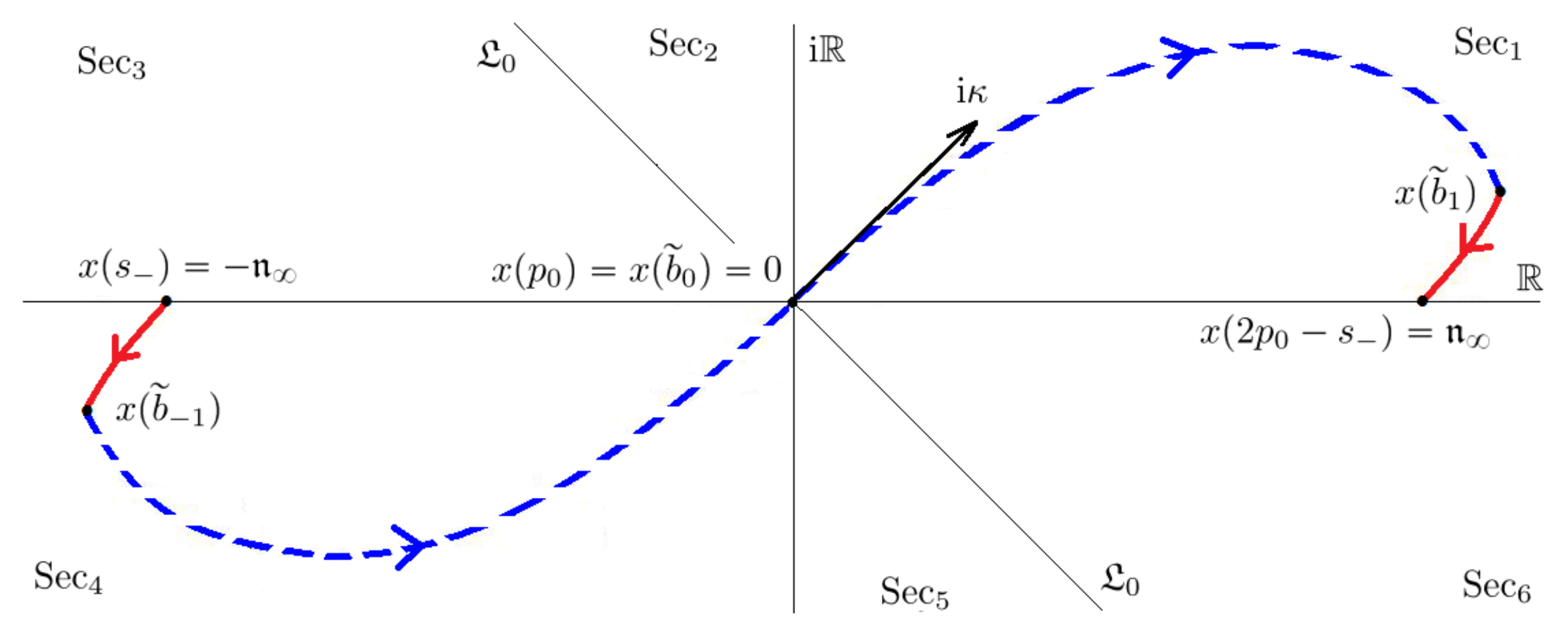}
\footnotesize \textbf{b)}
\end{minipage}
\vspace{-2ex}
\caption{\footnotesize \textbf{(a)} The computed trajectory of the $x$-extremal corresponding to an optimizer $\vep (\cdot)$ of the even-mode minimal length problem (\ref{e:argminOddEven}) with the parameters $\ka = 1  - \ii$, $\n_1 = 1$, $\n_2 = \n_\infty = 3.46$. The optimizer $\vep (\cdot)$ consists of 3 layers and has two switch points $\wt b_{\pm1}$ (see Sections \ref{s:bang-bang}-\ref{s:Corol}; for the computation method see  Section \ref{s:N}).
\textbf{(b)} The drawing based on (a). The part of trajectory $x[[s_-,\p_0]]$ (the part $x[[\p_0,2\p_0-s_-]]$) corresponds to the minimum-time control from $(-\n_\infty)$ to $0$ (resp., from $0$ to $\n_\infty$).  The two parts of $x$-trajectory corresponding to the two edge layers $I^\pm$ with $\vep(s) = \ep_\edge = \n_1^2$ are given by the red solid curves, the part corresponding to the layer $I^\cent$ with $\vep (s) = \ep_\cent = \n_2^2$ by the blue dashed curve.  The `no-return line of turning points' $\Li_0$ defined in Section \ref{ss:Monotonicity} and the lines $\RR$, $\ii \RR$, separate $\CC$ to the six sectors $\Sec_j$, $j=1,\dots,6$, important for the extremal synthesis. The trajectory has the structure (\ref{e:NS4}) with $m_1=m_2=0$. \label{f:x}}
\end{figure}

The proposition follows from the following identities: 
$\Si (\vep) = \Si_{-\n_\infty,\n_\infty}^{ s_-^{\vep}, s_+^{\vep}} (\vep)$
\ for $\vep \in \F$;
$\Si^{\odd} (\vep) = \Si_{-\n_\infty, \infty}^{s_-^{\vep},0 } (\vep) = \Si_{\infty,\n_\infty}^{0,s_+^{\vep}} (\vep)$ 
 and 
$\Si^{\even} (\vep) = \Si_{-\n_\infty, 0}^{s_-^{\vep} , 0} (\vep) 
 = \Si_{0, \n_\infty}^{0,s_+^{\vep}} (\vep)  $
\  for $\vep \in \F^\sym$. 


Assuming that there exists a minimum-time control $\vep$ from $x(s_-) = \eta_- $ 
to $\eta_+$,  one sees that the trajectory of the associated optimal solution $x$ of 
 (\ref{e:R}), (\ref{e:R2}) has no self-intersections in the sense that
\begin{equation} \label{e:noSelfint}
\text{$x (s) \neq x (\wt s)$  if $s \neq \wt s$ and 
$s,\wt s \in \left( s_-,s_- + T^{\min}_\ka (\eta_-,\eta_+) \right)$.} 
\end{equation}

\subsection{Pareto frontier of resonances of minimal $|\ka|$}
\label{ss:Min|om|}

Recall that $\Arg_0 (\cdot)$ is a continuous in $\CC \setminus \overline{\RR_-}$ branch of 
multi-valued complex argument $\Arg (\cdot)$ fixed by $\Arg_0 1 = 0$.


In this subsection, we reformulate the problems of length minimization (\ref{e:argminF})-(\ref{e:argminOddEven})
and, more generally, of the minimum-time control of (\ref{e:R})-(\ref{e:R2}) 
as the problem of minimization of modulus $|\ka|$ of an $(\eta_- ; \eta_+)$-eigenvalue $\ka$ for a given  
complex argument $\ga = \Arg_0 k$ over 
the family 
\begin{gather}
\F_{s_-,s_+} := \{ \vep (\cdot) \in \Ls^\infty_\RR (s_-,s_+) \ : \  
\n^2_1 \le \vep (s) \le \n^2_2 \text{ for $s \in (s_-,s_+)$}\} ,
\end{gather}
where the finite interval $(s_-,s_+)$ with $s_-<s_+$ and the tuple $(\eta_-;\eta_+ )$ are fixed.

The main tool for this reformulation is the natural scaling of eigenproblem 
(\ref{e:ep}), (\ref{e:BCxpm}):
\begin{gather} \label{e:scalingSi}
\text{if $\ka \in \Si_{\eta_-,\eta_+ }^{s_-,s_+} (\vep)$ and 
$\wt \vep (s)= \vep (\tau s)$ for $\tau \in \RR_+$, then $\tau \ka \in \Si_{\eta_-,\eta_+ }^{\tau^{-1} s_-, \tau^{-1}s_+ } (\wt \vep)$.
}
\end{gather}


Let us introduce the set 
\begin{equation} \label{e:Si[]}
 \Si_{\eta_-,\eta_+}^{s_-,s_+} [\F_{s_-}] := 
\bigcup_{\vep \in \F_{s_-,s_+} } \Si_{\eta_-,\eta_+}^{s_-,s_+} (\vep) 
\end{equation}
of \emph{achievable  $(\eta_- ; \eta_+)$-eigenvalues} (over $\F_{s_-,s_+}$).
We define the set of \emph{achievable $(\eta_- ; \eta_+)$-arguments} by 
\[
\Arg_0 \Si_{\eta_-,\eta_+}^{s_-,s_+} [\F_{s_-}] := \{ \Arg_0 \ka \ :\ \ka \in \Si_{\eta_-,\eta_+ }^{s_-,s_+} (\vep) 
\text{ for certain } \vep \in \F_{s_-,s_+} \},
\] 
and 
the \emph{minimal modulus} $\rho_{\min}  (\ga) = \rho_{\min}  (\ga, \eta_-,\eta_+) $ by
\begin{equation} \label{e:RhoMin}
\rho_{\min} (\ga) := \inf \{ |\ka| \ : \ \ka \in \Si_{\eta_-,\eta_+}^{s_-,s_+} [\F_{s_-}] \text{ and }
\Arg_0 \ka = \ga \}  .
\end{equation}
The function $\rho_{\min}$ takes values in $[0,+\infty]$ 
and depends on $\ga$, $\eta_\pm$, and $s_+-s_-$.
We omit $s_\pm$ and sometimes $\eta_\pm $ from the list of variables of $\rho_{\min}$
when they are fixed.

If $k_\ga^{\min} :=\ee^{\ii \ga} \rho_{\min} (\ga)$ belongs to 
$ \Si^{\eta_-,\eta_+}_{s_-,s_+} (\vep^{\min}_{\ga}) $
for a certain $\vep^{\min}_{\ga} (\cdot) \in \F_{s_-,s_+}$, i.e., if minimum 
is achieved in (\ref{e:RhoMin}), 
then we say that
\begin{align} \label{e:ep^min_ga}
\text{$\vep^{\min}_{\ga} (\cdot)$ is a 
resonator of minimal modulus $|\ka|$ for (the complex argument) $\ga$.}
\end{align}

The set $\Pa_{\md}^{\eta_-,\eta_+} :=\{ \ee^{\ii \ga} \rho_{\min} (\ga) \ : \, \ga \in \Arg_0 \Si_{\eta_-,\eta_+}^{s_-,s_+} [\F_{s_-,s_+}]\}$ 
forms the \emph{Pareto optimal 
frontier} for the problem of minimization of the 
modulus $|\ka|$ of an $(\eta_- ; \eta_+)$-eigenvalue $\ka$ over $\F_{s_-,s_+}$
(see Fig. \ref{f:Par} (a)).

The minimum-time control problem for the system (\ref{e:R})-(\ref{e:R2})
and the problem of finding of resonators of minimal modulus 
for given $\ga$ over $\F_{s_-,s_+} $ are equivalent in the sense of the following theorem,
which includes also a result on the existence of optimizers.

\begin{thm} \label{t:ExOptContr}
Let $\eta_- \neq \eta_+$, $\ka \neq 0$, and  $\ga =\Arg_0 \ka$. 
Then the following statements are equivalent:
\item[(i)] $\eta_- \in \Contr_{[0,+\infty)} (\eta_+,\ka)$, i.e.,  (\ref{e:R})-(\ref{e:R2}) is controllable 
from $\eta_-$ to $\eta_+$;
\item[(ii)] there exists a minimum-time control $\vep (\cdot) \in \F_{s_-}$ 
for  (\ref{e:R})-(\ref{e:R2}) that steers 
$\eta_-$ to $  \eta_+$ in the minimal time $ T^{\min}_{\ka} (\eta_-,\eta_+)$.
\item[(iii)] $\ga \in \Arg_0 \Si^{s_-,s_+}_{\eta_-,\eta_+} [\F_{s_-,s_+}]$, i.e., the complex argument 
$\ga$ is achievable over $\F_{s_-,s_+}$;
\item[(iv)]  there exist at least one resonator $\vep_\ga^{\min} (\cdot)$ of minimal modulus  
for $\ga$ over $\F_{s_-,s_+}$.

If statements (i)-(iv) hold true, then 
\begin{equation} \label{e:Tmin=rho}
T^{\min}_\ka (\eta_-,\eta_+) = \frac{(s_+-s_-) \rho_{\min} (\ga , \eta_-, \eta_+)}{|\ka|} .
\end{equation}
If, additionally, $s_\pm$ are chosen so that $s_+-s_-=T^{\min}_\ka (-\eta_-,\eta_+)$, then 
the families of minimum-time controls $\vep (\cdot) $ and of resonators of minimal
modulus $\vep_\ga^{\min} (\cdot)$ coincide. 
\end{thm}

\begin{proof}[Proof of Theorem \ref{t:ExOptContr}] 
The statement (i) $\Leftrightarrow$ (ii) is a standard existence result 
(see, e.g., Filippov's theorem in \cite{AS13}). 
 Equivalences (i) $\Leftrightarrow$ (iii), (ii) $\Leftrightarrow$ (iv), 
formula (\ref{e:Tmin=rho}), and the last statement of the theorem 
follow from the scaling (\ref{e:scalingSi}) and equivalence (C1) $\Leftrightarrow$ (C2) 
of Section \ref{ss:MinTime}. 
\end{proof}


\subsection{HJB equations for optimal resonators}
\label{ss:HJB}

Let $\eta_\pm \in \wh \CC$ and $\ka \in \CC \setminus \{0\}$ be fixed. Let us define for all states 
$x \in \wh \CC$ the \emph{backward  
value function} 
$V^-_{\eta_+} (x) := T^{\min}_\ka (x,\eta_+) $, which, according to (\ref{e:Tmin}),
takes values in $[0,+\infty]$ and, in the standard terminology \cite{BC08,CV03} is called simply the value function.
Recall that the domain of $V^-_{\eta_+} (\cdot)$ is the 
set 
\[
\dom V^-_{\eta_+} (\cdot) := \{ x \in \wh \CC \ : \ V^-_{\eta_+} (x)<+\infty\} = \Contr_{[0,+\infty)} (\eta_+,\ka) .
\]
Similarly, we define the \emph{forward value function} $V^+_{\eta_-} (x) := T^{\min}_\ka (\eta_-,x) $, $x \in \wh \CC$.
  
The Hamilton-Jacobi-Bellman (HJB) equation formally associated 
with the backward value function for the minimum time problem of Section \ref{ss:MinTime} can be written as (\ref{e:HJB}). To pose the corresponding boundary value problem rigorously on the 2-D real manifold $\wh \CC$ 
we use the approach of \cite{CV03}.

We assume $V: \wh \CC \to (-\infty,+\infty]$ and use 
the notation $x=x_1+\ii x_2 = (x_1;x_2) \in \RR^2$, 
\[
f (x,\ep) = f_1 (x,\ep) +\ii f_2 (x,\ep)= (f_1(x,\ep);f_2(x,\ep)) = (\re f (x,\ep);\im f (x,\ep)). 
\]
The similar notation is used for the $\RR^2$-coordinates of $\wt x = -1/x$ and $\wt f (\wt x)$. It is enough to consider only the minimal atlas for $\wh \CC$ that 
consists of the cover $\{\CC , \wh \CC\setminus\{0\} \}$ and the associated charts  
$\vphi: x \mapsto x=(x_1;x_2)^\top $, $\wt \vphi: x \mapsto \wt x =  (\wt x_1; \wt x_2)^\top  $ onto $\CC = \RR^2$.

We write (\ref{e:R}), (\ref{e:R2}) as the differential inclusion $ x'  \in F_\ka (x)$ on $\wh \CC$ defining the multifunction $ F_\ka $, which maps each $x \in \wh \CC$ to a subset $F_\ka (x)$ of the tangent space 
$T_x$ via its representation in the local coordinates
\begin{align*}
\vphi_* F_\ka (x) & = 
\{  f_1 (x,\ep) \pa_{x_1}+  f_2 (x,\ep) \pa_{x_2} \ : \ 
\ep \in [\n_1^2,\n_2^2] \} , \quad x \in \CC , \\
\wt \vphi_* F_\ka (\wt x) & = 
\{  \wt f_1 (\wt x,\ep) \pa_{\wt x_1} +  \wt f_2 (\wt x,\ep) \pa_{\wt x_2} \ : \ 
\ep \in [\n_1^2,\n_2^2] \} , \quad x = -\wt x^{-1} \in \wh \CC \setminus \{0\} .
\end{align*}
Note that $f$ and $\wt f$ depend also on the spectral parameter $\ka$ and that this dependence for $F_\ka$
is shown explicitly as the lower index.

Following \cite{CV03},  we say that a cotangent vector  $\la \in T^*_x \wh \CC$ is a proximal subgradient of $V$ at $x$ if $x \in \dom V$ and there exists a $C^2$-function $m$ defined in a certain neighborhood 
$\Om \subset \wh \CC$ of $x$ that satisfies $\la = \dd m (x)$ and $V(z)-V(x) \ge m(z)-m(x)$ for all $z \in \Om$.

The $\RR$-valued Hamiltonian $\H_\ka (x,\la)=\sup_{v \in F_\ka (x)} \< v , \la \>$ is defined 
for all $x \in \wh \CC$ and $\la \in T^*_x \wh \CC$. 
A lower semicontinuous function $V:\wh \CC \to (-\infty,+\infty]$ is called a \emph{proximal solution} 
to the HJB boundary value problem 
\begin{align}
& 0 =V(\eta_+),  \qquad 0  =1-\H_\ka (x,-\dd V(x) ) , \quad x \in \wh \CC \setminus \{\eta_+\}, \label{e:HJBwhC} 
\end{align}
if the following three conditions hold: (i) $0 = 1-\H_\ka (x,-\la)$ for all $\la \in \pa^{\Pd} V (x)$ and all $x \neq \eta_+$,
(ii) $0 \le 1-\H_\ka (\eta_+,-\la)$ for all $\la \in \pa^{\Pd} V (\eta_+)$, and (iii) $0= V(\eta_+)$,
where $\pa^{\Pd} V (x) \subset T^*_x \wh \CC$ is the set of all proximal subgradients 
of $V$ at $x$ \cite{CLSW95,CV03}. 


Part (i) of the next result provides a rigorous form to the statement 
that the HJB equation (\ref{e:HJB}) corresponds to the problem (\ref{e:argminF})
of the minimization of resonator length.

\begin{cor} \label{c:V}
Let $\ka \in \CC \setminus \{0\}$. Then the function $T^{\min}_\ka (\cdot,\cdot)$ satisfies the following collection of HJB boundary value problems:
\item[(i)] For each $\eta_+ \in \wh \CC$, the function 
$V^-_{\eta_+} (\cdot) := T^{\min}_\ka (\cdot,\eta_+)$ is 
lower semicontinuous and is a unique proximal solution to (\ref{e:HJBwhC}).
\item[(ii)] For each $\eta_- \in \wh \CC$, the function 
$V^+_{ \eta_-} (\cdot) := T^{\min}_\ka (\eta_-,\cdot) $ is 
lower semicontinuous and is a unique proximal solution to 
\begin{align*}
& 0 =V_{\eta_-}^+ (\eta_-),  \qquad   0  =1-\H_{-\ka} (x,-\dd V^+_{ \eta_- } (x) ) , \qquad   x \in \wh \CC \setminus \{\eta_-\}. \label{e:HJBfrom} 
\end{align*}
\end{cor}
\begin{proof} (i) follows from \cite{CV03}. (ii) follows from the combination of \cite{CV03} with the time-reversal symmetry of (\ref{e:R})-(\ref{e:R2}) (cf. (S0), (S2) in Section \ref{ss:Monotonicity}) and the existence result of Theorem \ref{t:ExOptContr} (the latter implies  
$ T^{\min}_\ka (\eta_-,\eta_+) = \min \{t \ge 0 \ : \ \eta_+ \in \Reach_{\{t\}} (\eta_-,\ka) \}$, cf. (\ref{e:Tmin})). 
\end{proof}

\section{Controllability and existence of minimizers}
\label{s:existence}

In the preparatory Subsection \ref{ss:Monotonicity}, 
we consider symmetries of the system  (\ref{e:R})-(\ref{e:R2}) and 
give technical Lemma \ref{l:MonotFunc} about lines of no return, which is used in  subsequent sections. Theorem \ref{t:GlContr} gives the global controllability result needed for the study of optimization problems (\ref{e:argminF}), (\ref{e:argminOddEven}) in the case 
$\n_1 \le \n_\infty \le \n_2$. The dependence of small time local controllability (STLC)
 for (\ref{e:R})-(\ref{e:R2}) on the target point $\eta_+$ can be easily obtained 
 with the use of the Kalman test and the conditions of \cite{S83} on equilibrium points.
Namely, (\ref{e:R})-(\ref{e:R2}) is STLC to $\eta_+$ iff $\eta_+ \in (-\n_1,-\n_2) \cup (\n_1,\n_2)$. Subsection \ref{ss:STLC} studies the uniform in $\ka$ version of STL-controllability, which is crucial for the proof of Theorem \ref{t:MinDec1}.
 
The first application of these results is the following existence statement, 
which immediately follows
from Theorem \ref{t:ExOptContr} and the controllability result of Theorem \ref{t:GlContr}.

\begin{cor}  \label{c:exist}
Let $\ka \in \CC_- \setminus \ii \RR$ and $\n_1 \le \n_\infty \le \n_2 $.
Then the following four sets 
\begin{equation} \label{e:argmin4prob} 
\argmin_{\substack{\vep \in \F \\ \ka \in \Si (\vep)}} \Lr (\vep) , \quad \argmin_{\substack{\vep \in \F^\sym \\ \ka \in \Si (\vep)}} \Lr (\vep) , \quad \argmin_{\substack{ \vep \in \F^\sym \\ \ka \in \Si^{\odd (\even)} (\vep)}} \Lr (\vep) 
\quad \text{are nonempty.}
\end{equation}

\end{cor}

\subsection{Lines of no return, monotonicities, and symmetries}
\label{ss:Monotonicity}

Let us consider first the symmetries of  (\ref{e:R})-(\ref{e:R2}), which appear 
because the right hand side in (\ref{e:R}) is an even function in $x$ and $\F_{s_-,s_+}$ is symmetric 
w.r.t. $s^\cent$. 

Considering dynamics of $(-x(-s))$, $x(-s)$, and $\overline{x(s)}$, 
one sees that the statements that $\vep(\cdot)$ steers $x(s_-) = \eta_-$ to $\eta^+ $ in time $t$ 
for the following choices of 
$(\eta_-;\eta_+)\in \wh \CC^2$, $\ka$, and $\vep (\cdot)\in \F_{s_-}$ are equivalent:
\begin{itemize}
\item[(S0)] $(\eta_-;\eta_+) = (\eta_1;\eta_2 ) $ , $\ka = \ka_0$, and $\vep (s) = \vep_0 (s) $;
\item[(S1)] $(\eta_-;\eta_+) = (- \eta_2; - \eta_1 )$, $\ka = \ka_0$, and $\vep (s) = \vep_0 (-s)$;
\item[(S2)] $(\eta_-;\eta_+) = (\eta_2; \eta_1 )$, $\ka = - \ka_0$, and $\vep (s) = \vep_0 (-s)$;
\item[(S3)] $(\eta_-;\eta_+) = (\overline{\eta_1}; \overline{\eta_2} )$, $\ka = - \overline{\ka_0}$, and 
$\vep (s) = \vep_0 (s)$.
\end{itemize}


Equivalence $(S0) \Leftrightarrow (S3)$ implies that 
if $(\eta_-;\eta_+) \in \wh \RR^2$ and $\vep (\cdot) \in \F_{s_-,s_+}$, then
$\Si_{\eta_-,\eta_+}^{s_-,s_+} (\vep)$ and 
$\Si_{\eta_-,\eta_+}^{s_-,s_+}  [\F_{s_-}]$
are symmetric w.r.t. the imaginary axis $\ii \RR$.
 

\begin{lem} \label{l:MonotFunc}
Let $\vep \in \F_{s_-,s_+} $,  $\Arg_0 \ka \in (-\pi/2,0) $, 
and $y $ be a nontrivial solution to (\ref{e:ep}) in $[s_-,s_+]$.
Then, 
for each $\tau \in [0, - 2 \Arg_0 \ka ]$, the function 
$
G_\tau (s) = \re \left(e^{\ii ( \tau - \pi/2)} \overline{y (s)} \pa_s y (s) \right)
$ 
is strictly increasing in $s$. In particular, there exists at most one point $\p_\tau \in [s_-,s_+]$ such that $G_\tau (\p_\tau) = 0$. 
If $G_\tau (\p_\tau) = 0$, then  
$x(\p_\tau)$ is the only point where the trajectory $x [[s_-,s_+]] := \{ x (s) : s \in [s_-,s_+] \} $ of $x=\frac{y'}{\ii \ka y}$ 
intersects the (extended) line 
$\Li_\tau:= \ii \ee^{-\ii (\tau+\Arg \ka) } \wh \RR$.
\end{lem}

\begin{proof}
For $\tau \in [0, -2 \Arg_0 \ka ]$, 
\begin{gather} \label{e:Gom'}
G'_\tau =  
|y'|^2 \re  \ee^{\ii (\tau-\pi/2)} + |y|^2 \vep \re \left(\ee^{\ii (\tau-\pi/2)} (-\ka^2) \right) \ge 0 .
\end{gather}
It is easy to see that $G'_\tau$ cannot vanish on an interval for a nontrivial $y$. 
The equality $G_\tau (s) = |y|^2|\ka|\re \left(e^{\ii ( \tau + \Arg \ka)} x(s) \right) $ 
completes the proof.
\end{proof}

\begin{rem} \label{r:noreturn}
In the settings of Lemma \ref{l:MonotFunc}, $\Li_\tau$ is 
an \emph{line of no return} in the following sense: if $x(s_0)$ is in the $\wh \CC$-closure 
$\Hp_\tau^+ := \Li_\tau \cup \ii e^{-\ii (\tau+\Arg \ka) }\CC_- $
of the half-plane $ \ii e^{-\ii (\tau+\Arg \ka) }\CC_-$, then $x(s) \in 
\ii e^{-\ii (\tau+\Arg \ka) }\CC_-$ for $s>s_0$. This shows that 
(\ref{e:R})-(\ref{e:R2}) is not controllable from $\eta_- \in \Hp_\tau^+$ 
to $\eta_+ \in \Hp_\tau^- \setminus \{\eta_-\}$, where $\Hp_\tau^- := - \Hp_\tau^+$. The line $\Li_0$ is shown in Fig. \ref{f:x}.
\end{rem}

\begin{rem} \label{r:monotFuncR}
\emph{(i)} Let $\ka \in \ii \RR_-$. Then Lemma  \ref{l:MonotFunc} and 
Remark \ref{r:noreturn} are valid for $\tau \in (0,\pi)$.\\
\emph{(ii)} Let $\ka \in \RR$ and $\vep (\cdot) \in \F_{s_-,s_+}$. 
Then (\ref{e:Gom'}) applied to $\tau = 0$ implies that $G_0 (\cdot)$ is a constant function and 
that (\ref{e:R})-(\ref{e:R2}) is not controllable from $\eta_- \in \ii \CC_-$ to 
$\eta_+ \in (- \Hp_0^+)$. 
\end{rem}

\begin{rem} \label{r:k_in_iR}
Let $\ka \in \ii \RR_-$. Let $\eta_- = x(s_-) = y' (s_-)/(\ii \ka y (s_-))$ be in $\wh \RR$ like in (\ref{e:BCpm}). 
Then the control problem and the resonance optimization problems are essentially 
simpler since the trajectory of $x$ lies in $\wh \RR $. On one hand, the control problems can be solved 
by the application of the simple no-selfintersection principle (\ref{e:noSelfint}). On the other hand, the solution of the problems  
(\ref{e:argmin4prob}) for $\n_\infty \in \RR_+$ and $\ka \in \ii \RR_-$ can be obtained from the results of \cite{KLV17_MFAT}.
\end{rem}

\begin{rem} \label{r:C4}
Note that the constancy of $G_0 (\cdot)$ for $\ka \in \RR$, the existence of no-return lines 
$\Li_\tau $ for $\ka \in \CC_4 \cup \ii \RR_-$, and 
 the symmetries (S0)--(S3) imply the well-known fact 
 that $\Si (\vep) \subset \CC_-$ for $\vep \in \F$. 
 This, the symmetries (S0)--(S3),  and Remark \ref{r:k_in_iR} on the case $\ka \in \ii \RR_-$ show 
 that we can restrict our attention to the case $\ka \in \CC_4$.
 \end{rem}
 
 \subsection{Global controllability to stable equilibria}
\label{ss:const}

Let us denote by $\J :\wh \CC \to \wh \CC$ the Kutta-Zhukovskii transform $\J(z) = (z + z^{-1})/2$ and by 
$\I:\wh \CC \to \wh \CC$ the linear fractional transform of the form $\I (z) := (1-z)/(1+z)$. Note that 
\begin{gather*} \label{e:I2J-1}
\text{$\I (-z) = (\I (z))^{-1}$, $\I (z^{-1}) = - \I (z)$, $\I$ is involution (i.e., $\I (\I(z)) \equiv z$), \ } 
\end{gather*}
and that the inverse $\J^{\{-1\}} (\cdot)$ of $\J$ given by the expression
$ \J^{\{-1\}} (\zeta) := \zeta + \sqrt{\zeta^2 -1} $
is a 2-valued analytic function with the branching points $\pm 1$.
For $\zeta \in \CC \setminus [-1,1]$,  two branches $\J_{\pm 1}^{\{-1\}} (\zeta) $ of $\J^{\{-1\}}$ can be 
singled out by the condition $ \pm (|\J_{\pm 1}^{\{-1\}} (\zeta) | -1 ) > 0 $.

Let us consider the forward and backward evolution of (\ref{e:R})-(\ref{e:R2}) 
under the constant control $\vep (\cdot) \equiv \ep \in \RR_+$ from an initial state 
$x(s_-) = \eta_- \in \wh \CC$. 
Let us define the continuous $\wh \CC$-valued  function
\begin{equation} \label{e:vthep}
\vth_{\ep} (s) = \I (\ep^{-1/2} x(s)) .
\end{equation}

The evolution of $x(\cdot)$ and of the associated solution $y$ to (\ref{e:ep}) have the form 
\begin{gather} \label{e:yxRepr}
\text{ $y(s) = c_+ \ee^{\ii \vka s} + c_- \ee^{-\ii \vka s} $ and 
$x (s) = \ep^{1/2} \frac{c_+ - c_- \ee^{-2 \ii \vka s} }{c_+  + c_- \ee^{-2\ii \vka s} }$
\quad  with $\vka :=\ep^{1/2} \ka $}
 \end{gather}
and certain $c_\pm \in \CC$ satisfying $|c_-|+|c_+| \neq 0$.
The identity $\I (\I (z)) \equiv z$ implies 
\begin{gather} \label{e:vth=evth}
\vth_\ep (s) = \frac{c_-}{c_+} \ee^{-2\ii \vka  (s-s_-)} =\vth_\ep (s_-) \ee^{-2\ii \vka  (s-s_-)}, \qquad s \in \RR ,
\end{gather}
which have to be understood as $\vth_\ep (s) \equiv \infty$ if $c_+ = 0$.
The latter case corresponds the unique unstable equilibrium solution $x(\cdot) \equiv -\ep^{1/2}$.
The equation $x' (s) = \ii \ka (- x^2  + \ep )$ has also one stable equilibrium solution $x(\cdot) \equiv \ep^{1/2}$. 
Summarizing, we see that $\eta_- \mapsto x(t)$ is an holomorphic evolution on 
$\wh \CC$ given by a continuous subgroup of the group $\mathrm{Aut} (\wh \CC)$ of Möbius transformations. 

Assume now that $\eta_- \neq \pm \ep^{1/2} $. Then 
\begin{align} 
y (s) & = A \cos \left( \frac{\ii}{2} \Ln_\star \vth_\ep (s) \right) = 
A \cos \left( \vka  (s-s_-) + \frac{\ii}{2} \Ln_\star \vth_\ep (s_-) \right), \label{e:yRepr} \\
x (s) & = \ii \ep^{1/2 } \tan \left(\frac{\ii}{2} \Ln_\star \vth_\ep (s)  \right)
=\ii \ep^{1/2 } \tan \left( \vka  (s-s_-) + \frac{\ii}{2} \Ln_\star \vth_\ep (s_-)  \right) ,
\label{e:xRepr}
\end{align}
where $\Ln_\star \vth_\ep (\cdot)$ is an arbitrary continuous on $\RR$ branch of 
the multifunction $\ln \vth_\ep (\cdot)$ and $A \in \CC \setminus \{0\}$ is a constant
depending on $c_\pm$ and the choice of the branch $\Ln_\star \vth_\ep (\cdot)$.

Because of the form of right sides of (\ref{e:yRepr}), (\ref{e:xRepr}), we will say that $\vth_\ep (\cdot)$ 
is the \emph{ilog-phase} for the constant control $\ep$. A curve $x(\cdot)$ produced by the constant control $\ep$ has the form 
$C_1 \tan \left( \ep^{1/2} \ka  (s-s_-) + C_2 \right)$. 
In the case $\ka \in \CC \setminus \ii \RR$, 
$x [\RR]$ is the image of the logarithmic spiral (\ref{e:vth=evth}) under 
the linear fractional map $z \mapsto \ep^{1/2} \I (z)$.
This gives the rigorous sense to the statement that, as $s \to \pm \infty$,
the trajectory of $x(s)$ asymptotically approaches 
a logarithmic spiral with the pole at $\pm \ep^{1/2}$.

Using (\ref{e:vth=evth}) it is easy to see that, if $\eta_- \neq \eta_+$ and each of 
$\vth_\ep (s_\pm) = \I (\ep^{-1/2} \eta_\pm)$ is not in the set $\{0,\infty\}$,
then the $(\eta_-,\eta_+)$-spectrum produced by the constant resonator $\vep (\cdot) = \ep$ is
\begin{align} \label{e:Si(ep)}
\Si_{\eta_-,\eta_+}^{s_-,s_+} (\ep) = \{ k_n (\ep) \}_{n \in \ZZ}, \quad 
k_n (\ep) := -  \frac{ \ii }{ 2 (s_+ - s_-) \ep^{1/2} } \Ln \frac{\vth_\ep (s_-)}{\vth_\ep (s_+)}
+ \frac{\pi n}{(s_+ - s_-) \ep^{1/2}} .
\end{align}

Recall that $\Contr_{[0,+\infty)} (\eta,\ka)$ is the set of points controllable to $\eta$ by feasible controls.

\begin{thm} \label{t:GlContr}
Let $\ka \in \CC_4 $. Let $\eta_+$ belongs to the interval $[\n_1,\n_2]$ of admissible refractive indices. 
Then (\ref{e:R}), (\ref{e:R2}) is globally controllable to $\eta_+$ in the sense 
$\Contr_{[0,+\infty)} (\eta_+,\ka) = \wh \CC $.
\end{thm}
\begin{proof}
It is enough to consider the case $\ka \in \CC_4$.
Geometrically, $\Contr_{[0,+\infty)} (\eta,\ka) = \wh \CC $ can be proved using the fact that $x (\cdot) \equiv \eta_+$ is a
stable equilibrium solution to (\ref{e:R}) with the constant control $\vep (\cdot) \equiv \eta_+^2$.
The corresponding trajectory $x (s)$ with $x(s_-) = \eta_- \neq \eta_+ $ asymptotically approaches 
a logarithmic spiral with the pole at $\eta_+$ as $s \to +\infty$. 
One can concatenate this asymptotic spiral with a backward trajectory through $\eta_+$
under another feasible constant control $\vep (\cdot) \equiv \ep_0 \neq \eta_+^2$. 
This argument can be made rigorous with the use of the map from (\ref{e:vthep})
and the dynamics of the ilog-phases for $\ep = \eta_+^2$ and $\ep = \ep_0$.

Let us give another proof that can be used in less geometrically transparent situations and is 
based on the Pareto optimization equivalence of Section \ref{ss:Min|om|}.
Let us fix an interval $(s_-, s_+)$ and $\eta_- \neq \eta_+$. 
We want to show that 
\begin{align} \label{e:arg)}
\Arg_0 \Si_{\eta_-,\eta_+}^{s_-,s_+} [\F_{s_-,s_+}] \supset (-\pi/2,0) 
\end{align}
and use Theorem \ref{t:ExOptContr}.
Consider for $n \in \NN$ the behavior of 
$k_n (\ep)$ from (\ref{e:Si(ep)}) as the constant $\ep$ 
runs through $[\n_1^2,\n_2^2] \setminus \{\eta_+^2\}$. 
On one hand, such $k_{n} (\ep)$ lie in the strip 
$
 \re z \in \frac{\pi }{s_+-s_-} \left[ \frac{n-1/2}{\n_2}, \frac{n+1/2}{\n_1} \right] .
$
On the other, 
$ \im k_{n} (\ep) \to -\infty $ as $\ep \to \eta_+^2$ since $x(s_+) = \eta_+$ and $\vth_\ep (s_+) \to 0$. So 
$
\{ \Arg_0 k_{n} (\ep) \ : \ \ep \in [\n_1^2,\n_2^2] , \ \ep \neq \eta_+^2 \} 
\supset \left( -\pi/2,  \Arg_0 k_n (\ep_0) \right] ,
$
where $\ep_0$ is any fixed number in $[\n_1^2,\n_2^2] \setminus \{\eta_+^2\}$.
Since $\Arg_0 k_n (\ep_0) \to 0$ as $n \to +\infty$, we see from (\ref{e:Si(ep)})
that (\ref{e:arg)}) holds true. 
\end{proof}

\subsection{Uniform in $\ka$ small-time local controllability}
\label{ss:STLC} 

Let us recall the definition of regular equilibria \cite{S83}.
If there exists a feasible control $\vep \in \F_{s_-}$ such that the constant function 
$x(s) = \eta_-$,
$s \in [s_-,+\infty)$, is a solution to  (\ref{e:R}), (\ref{e:R2}), then $\eta$ 
is an \emph{equilibrium point for} (\ref{e:R}), (\ref{e:R2}).
The set of equilibrium points of (\ref{e:R}), (\ref{e:R2}) equals $[-\n_2,-\n_1] \cup [\n_1,\n_2]$.
The equilibrium points in the set $(-\n_2,-\n_1) \cup (\n_1,\n_2)$ are called \emph{regular},
the equilibrium points $\pm \n_1$ and $\pm \n_2$ \emph{singular}.

Recall that the system (\ref{e:R}), (\ref{e:R2}) is \emph{small-time local controllable} (STLC) from $\eta_-$ (STLC to $\eta_+$)
if $\eta_- \in \bigcap_{t>0} \Intr \Reach_{[0,t)} (\eta_-,\ka)$
(resp., $\eta_+ \in \bigcap_{t>0} \Intr \Contr_{[0,t)} (\eta_+,\ka)$), where $\Intr E$ is the interior of a set $E$.
Since we need a version of this definition that is uniform over a set $\Om$ of spectral parameters $\ka$,
we quantify STLC with the use of the function 
\begin{align*} 
\ra^{\max} (t;\eta_+,\ka) := \sup \{ \ra \ge 0 \ : \ \DD_\ra (\eta_+) \subset \Contr_{[0,t)} (\eta_+,\ka) \} ,
\end{align*}
where $\eta_+ \in \CC$ and $\DD_\ra (z_0) := \{ |z-z_0| <\ra \} $. 
As a function of $t$, $\ra^{\max} (\cdot;\eta_+,\ka)$ is nondecreasing, takes finite values for small enough $t$,
and satisfies $\ra^{\max} (0;\eta_+,\ka) =0$.

\begin{defn} \label{d:uniSTLC} Let $\Om$ be a subset of $\CC$.
We say that (\ref{e:R}), (\ref{e:R2}) is \emph{STLC to $\eta \in \CC$ uniformly over $\Om$}
if there exists $\de>0$ and a strictly increasing continuous  function 
$ \ra^{\max}_{\Om ,\eta} : [0,\de) \to [0,+\infty)$ such that 
$\ra^{\max}_{\Om,\eta} (t) \le \ra^{\max} (t,\eta,\ka)$
for all $\ka \in \Om$ and $t \in [0,\de)$.
The uniform over $\Om$ version of STL-controllability \emph{from} $\eta$ can be defined similarly.
\end{defn}

Let us denote by $x_{\eta_-} (s, \vep, \ka )$ the solution to (\ref{e:R}), (\ref{e:R2})
satisfying $x(s_-) = \eta_-$. 
Let $\ka \in \CC$, $\eta_- \in \CC$, and $s_+ = s_-+t_0$.
 Then for small enough $\de_1>0$
and small enough $t_0 \in (0, T^{\min}_{\ka} (\eta_-,\infty))$, the map 
$(\vep (\cdot);\ka) \mapsto x_{\eta_-} (\cdot, \vep, \ka )$ from 
$\F_{s_-,s_+} \times \DD_{\de_1} (\ka_0)$ to $W^{1,\infty}_\CC [s_-,s_+]$ is analytic in a certain  neighborhood of $\F_{s_-,s_+} \times \DD_{\de_1} (\ka_0)$ in the Banach space $L^\infty_\CC (s_-,s_+) \times \CC$  
(here $\max \{\| \vep \|_\infty, |\ka| \}$ can be taken as the norm in $L^\infty_\CC (s_-,s_+) \times \CC$). 
For $\vep_0 \in \F_{s_-,s_+}$ and $\vep_1 \in L^\infty_\CC (s_-,s_+)$,
the directional derivative 
  $\frac{\pa x_{\eta_-} (s, \vep_0, \ka )}{\pa \vep} (\vep_1) = 
\pa_\zeta x_{\eta_-} (s, \vep_0 +\zeta \vep_1, \ka )$ 
is equal to
\begin{gather} \label{e:pae}
\frac{\pa x_{\eta_-} (s, \vep_0, \ka )}{\pa \vep} (\vep_1)  = 
\ii \ka \int_{s_-}^s \vep_1 (\wt s) \exp \left( -2\ii \ka \int_{\wt s}^s x_{\eta_-} (\si,\vep_0 (\si),\ka) \dd \si \right) \dd \wt s.
\end{gather}

 According to  \cite{S83},  if (\ref{e:R}), (\ref{e:R2}) is STLC to (from) $\eta \in \wh \CC$, then $\eta $ is a regular 
 equilibrium point, i.e.,  $\eta \in (-\n_2,-\n_1) \cup (\n_1,\n_2) $. 
In the case $\ka \not \in \ii \RR$, the next theorem gives a strengthened uniform version 
 of the converse implication.

\begin{thm} [uniform STLC]
\label{t:STLC}
Let $\ka_0 \in \CC \setminus \ii \RR$.
Then (\ref{e:R}), (\ref{e:R2}) is STLC to (from) $\eta \in \wh \CC$ uniformly over $\DD_\de (\ka_0)$
with a certain $\de>0$ if and only if $\eta \in (-\n_2,-\n_1) \cup (\n_1,\n_2)$.
\end{thm}

\begin{proof} 
The part `only if' follows from \cite{S83}. Let us put $\eta \in (-\n_2,-\n_1) \cup (\n_1,\n_2)$
and prove the part `if'. 
For $\de_0>0$, we define the nondecreasing function $\ra^{\max}_{\de_0} : [0,+\infty) \to [0,+\infty]$ by
$
\ra^{\max}_{\de_0} (t) := \inf_{|\ka - \ka_0| < \de_0} \ra^{\max} (t,\eta,\ka).
$
Note that $\dom \ra^{\max}_{\de_0}$ contains $ \left[ 0,T^{\min}_{\ka_0} (\infty,\eta) \right)$, where $T^{\min}_{\ka_0} (\infty,\eta) = T^{\min}_{\ka_0} (-\eta,\infty) >0$ 
is the \emph{escape time} from $(-\eta)$ to $\infty$.
Assume that for a certain $\de_0 >0$ and $\de$ satisfying 
$0 < \de < \de_1 :=\min \{1,T^{\min}_{\ka_0} (\infty,\eta)\}$,
\begin{align} \label{e:rhode}
\text{$\ra^{\max}_{\de_0} (t) > 0 $ for all $t\in (0,\de)$}.
\end{align} 
Then the function $\rho_{\Om,\eta}^{\max}$ defined on $[0,\de)$ by 
$\rho_{\Om,\eta}^{\max} (t) :=  \int_0^t  \ra^{\max}_{\de_0} (s) \dd s$
satisfies Definition \ref{d:uniSTLC} with $\Om= \DD_{\de_0} (\ka_0)$.
Hence, to prove (ii) it is enough to prove (\ref{e:rhode})  
for certain $\de_0, \de >0$. 

Put $\eta_- = - \eta$. 
Using the symmetry (S0) $\Leftrightarrow$ (S1), 
it is easy to see that if, for each $t\in (0,\de)$,
\begin{align} \label{e:Reach>D}
\text{$\Reach_{[0,t)} (\eta_-) \supset \DD_{\ra} (\eta_-) $ with a certain $\ra>0$ (depending on $t$)
for all $\ka \in \DD_{\de_0} (\ka_0) $},
\end{align}
then (\ref{e:rhode}) is fulfilled, and so (\ref{e:R}), (\ref{e:R2}) is STLC from $\eta_-$
and STLC to $\eta $ uniformly over $\DD_{\de_0} (\ka_0)$.

Let us prove (\ref{e:Reach>D}) for $\ka_0 \not \in \ii \RR$.
Put $d_0 = \min \{|\eta^2 - \n_1^2| , |\eta^2 - \n_2^2 |\}$,
and  $\vep_0 (s) \equiv \eta^2 $. So $x_{\eta_-} (s, \vep_0, \ka ) \equiv \eta_-$ and 
(\ref{e:pae}) takes the form 
\begin{gather} \label{e:paeRCP}
\frac{\pa x_{\eta_-} (s, \vep_0, \ka_0 )}{\pa \vep} (\vep_1)  = 
\ii \ka \int_{s_-}^s \vep_1 (\wt s) \ee^{ 2\ii \ka_0 (\wt s - s) \eta_-} \dd \wt s
\end{gather}

Let us take $\de_0, \de >0$ such that $2 \de_0 < |\re \ka_0|$ and $\de |\eta \re \ka_0| < \pi/6 $. Let $s_+ \in (s_-,s_- + \de)$.
Then 
$
-\frac{\pi}{2}<\inf_{\substack{|\ka - \ka_0| \le \de_0 \\ s \in [0,\de] \quad}} \Arg_0 (\ee^{2\ii \ka_0 s })
< \sup_{\substack{|\ka - \ka_0| \le \de_0 \\ s \in [0,\de] \quad
}} \Arg_0 (\ee^{2\ii \ka_0 s }) < \frac{\pi}{2} 
$
and it follows from (\ref{e:paeRCP}) that taking as $\vep_j (\cdot)$, $j=1,2$, 
characteristic (indicator) functions of disjoint small enough subintervals
of $[s_-,s_+]$, we can ensure that the Fréchet derivatives $\frac{\pa h_\ka (0)}{\pa \zeta}$ at 
$\zeta = (\zeta_1;\zeta_2) =0$
of the functions $ h_\ka : (-d_0,d_0)^2  \to \RR^2$,
\[
h_\ka (\zeta_1,\zeta_2) :=  x_{\eta_-} (s_+, \vep_0+\zeta_1 \vep_1 +\zeta_2 \vep_2, \ka),
\]
satisfy $\sup_{|\ka - \ka_0|\le \de_0} \| (\frac{\pa h_\ka (0)}{\pa \zeta} )^{-1} \|_{\RR^2 \to \RR^2}  < +\infty$.
Since the images of $(-d_0,d_0)^2$ under $h_\ka$ are subsets of 
the images of $\F_{s_-,s_+}$ under the mapping $\vep \mapsto x_{\eta_-} (s_+, \vep, \ka )$,
one can see (e.g., from the Graves theorem) that $\Reach_{[0,s_+)} (\eta_-,\ka) \supset \DD_{\de_2} (\eta_-)$ with certain $\de_2>0$
for all $\ka \in \DD_{\de_0} (\ka_0) $.
This completes the proof of (\ref{e:rhode}) and of the theorem.
\end{proof}

\begin{rem}\label{r:LC}
It can be seen from Theorem \ref{t:GlContr} and its proof that 
for $\eta \in [-\n_2,-\n_1] \cup [\n_1,\n_2]$ and $\ka \in \CC_4$,
(\ref{e:R}), (\ref{e:R2}) is locally controllable to $\eta$ in the sense that 
$\eta \in \Intr \Contr_{[0,+\infty)} (\eta,\ka)$. Thus, for $\eta = \pm \n_j $, $j=1,2$,
(\ref{e:R})-(\ref{e:R2}) is locally controllable to $\eta$, but is not STLC to $\eta$.
\end{rem}

\begin{rem} \label{r:Vcont}
Let $k \in \CC_4$. Then it follows from Theorems \ref{t:GlContr}, \ref{t:STLC} and 
\cite[Chapter 4]{BC08} that the value function 
$V^-_{\eta_+} (\cdot):= T^{\min}_\ka (\cdot,\eta_+)$ 
has the following properties: (i)  $\dom V^-_{\eta_+} (\cdot )= \wh \CC$ whenever $\eta_+ \in [\n_1,\n_2]$, (ii) if $\eta_+ \in (-\n_2,-\n_1) \cup (\n_1,\n_2)$, then $\dom V^-_{\eta_+} (\cdot )$ is an open subset of $\wh \CC$ and $V^-_{\eta_+} (\cdot)$ is continuous on $\dom V^-_{\eta_+}$.
\end{rem}

\section{Maximum principle and rotation of modes}
\label{s:bang-bang}

In the rest of the paper it is assumed that $\ka \in \CC_4$ 
(see Section \ref{s:Results} for explanations).
For the optimal control terminology used in this section, 
we refer to \cite{SL12}.

The goal of the section is to combine the monotonicity result of Lemma \ref{l:MonotFunc} 
and the properties the complex argument $\Arg y$ 
of an eigenfunction $y$ of (\ref{e:ep}), (\ref{e:BCxpm}) with the Pontryagin Maximum Principle (PMP) 
in the form of \cite[Section 2.8]{SL12}.
The rotational properties of $y$ are important to show 
that singular arcs and the chattering effect are absent, and so 
minimum-time controls are of bang-bang type in the sense described below.

Let $[s_-,s_+]$ be a bounded interval in $\RR$.
Recall that a function $\vep (\cdot)$  is called \emph{piecewise constant in  $[s_-,s_+]$} 
if there exist a finite partition
$s_- = b_0 < b_1 < b_2 < \dots < b_{n} < b_{n+1} = s_+$, $n \in \{0\}\cup\NN $,
such that, after a possible correction of $\vep (\cdot)$ on a set of zero (Lebesgue) measure, 
\begin{equation} \label{e:pwconst}
\text{$\vep (\cdot)$ is constant on each interval $(b_{j} , b_{j+1}) $, $0 \le j \le n$.}
\end{equation}
 For such functions, 
 we always assume that the aforementioned correction that ensures (\ref{e:pwconst}) has been done. 
 A point $b_j$, $1 \le j \le n$, of the partition is called 
 \emph{a switch point of $\vep (\cdot)$} if $\vep (b_j - 0) \neq \vep (b_j+0)$. 
 
Suppose $\vep (\cdot) \in \F_{s_-,s_+}$. 
Then $\vep (\cdot)$ is said to be a \emph{bang-bang control}
on $[s_-,s_+]$ if it is a piecewise constant function that takes only the values 
$\n^2_1$ and $\n^2_2$
(after a possible correction at switch points). 

When $\vep (\cdot)$ represents a layered structure of an optical resonator, 
it is natural to say that the maximal intervals of constancy $(b_j , b_{j+1}) $ 
of a piecewise constant control $\vep (\cdot) \in \F_{s_-,s_+}$ are \emph{layers} of \emph{width} $b_{j+1} - b_j$ 
with the constant permittivity equal to $\vep (s)$, $s \in (b_j , b_{j+1})$.

Assume that $\vep (\cdot) \in \F_{s_- ,s_+ }$ and that $y$ is a 
nontrivial solution to (\ref{e:ep}) on $[s_- ,s_+ ]$. 
If the point $\p_0$ introduced in Lemma \ref{l:MonotFunc} exists in $[s_-,s_+]$ it 
is called the \emph{turning point} of $y $ \cite{KLV17}. Note that 
if $y(s_0) y'(s_0)= 0$ for a certain $s_0 \in [s_-,s_+]$, then $s_0 = \p_0 = \p_\tau$ 
for every $\tau \in [0, - 2 \Arg_0 \ka ]$ (for definition of $\p_\tau$, see Lemma \ref{l:MonotFunc}). 
In particular, the set $\{ s : y(s) = 0\}$ consists of at most one point. If the point $\p_0$ 
does not exist in the interval $[s_-,s_+]$, we will assign to $\p_0$ a special value outside of $[s_-,s_+]$, which will be specified later.

The special role of $\p_0$ is that the trajectory of $y$ rotates clockwise for $s < \p_0 $,
and counterclockwise  for $s > \p_0 $. More rigorously, 
the multifunction $\Arg y (\cdot)$ has a branch $\Arg_\star y (\cdot) \in W^{1,\infty} [s_-,s_+]$ 
that is defined and continuously differentiable  on $[s_-,s_+] \setminus \{s: y(s) = 0 \} $ 
and possesses the following properties:
(A1) if an interval $I \subset [s_-,s_+]$ 
does not contain $\p_0$, 
then the derivatives $\pa_s \Arg_\star y (s) $ are nonzero and of the same sign  for all $s \in I$;
(A2) in the case $\p_0 \in [s_-,s_+]$, we have  
\begin{align}
\text{$\pa_s \Arg_\star y (s) < 0$ if  $s < \p_0$, \qquad $\pa_s \Arg_\star y (s)> 0 $ if $s > \p_0$;}
\label{e arg dec arg inc}
\end{align}
(A3) if $\p_0 \in (s_-,s_+)$ and $y (\p_0) = 0$, there exist finite limits 
$\Arg_\star y (\p_0 \pm 0)$ satisfying
\begin{align} \label{e:A3}
\text{ 
$\Arg_\star y (\p_0 + 0)= \pi+\Arg_\star y (\p_0 - 0) $ } .
\end{align}

The existence of such $\Arg_\star y (\cdot)$ with properties (A1), (A2) 
is essentially proved in \cite{KLV17}. Formally, \cite{KLV17} works with the case where 
$\eta_\pm \in \RR_\pm$ 
in (\ref{e:BCxpm}). However the proof can be extended without changes. 
It is based on the facts that
$\im (\ii \ka x (s)) = \pa_s \im \ln y (s)  = \pa_s \Arg_\star y (s) $
for a suitable branch of $\ln y (s)$ differentiable on $\RR \setminus \{ s : y (s) = 0\}$,
and that the function 
$|y |^2 \im (\ii \ka x)  = \im (\overline{y} \pa_s y) = G_{0}  $ 
is strictly increasing (see Lemma \ref{l:MonotFunc} for $\tau =0$).
Concerning the property (A3),  one sees that $y' (\p_0) \neq  0$ in the case $y (\p_0) = 0$
and so $\Arg_\star y (\p_0 \pm 0) = \pm \Arg y' (\p_0) \ (\modn 2 \pi)$. 
This ensures that $\Arg_\star y$ can be chosen such that (A3) holds.

From now on, we assume that, for each $y$, the function $\Arg_\star y$ 
is a certain fixed branch of $\Arg y$ satisfying the above properties.

\begin{thm} \label{t:PMP}  Let $\ka \in \CC_4$, $\eta_\pm \in \wh \CC$, and $\eta_- \neq \eta_+$. 
Let $\vep (\cdot)$ be a minimum-time control from the initial state $x(s_-) =\eta_-$ of the system 
(\ref{e:R})-(\ref{e:R2}) to the target state $x (s_+) = \eta_+$. Then $\vep (\cdot)$ is a bang-bang 
control on $[s_-,s_+]$, and there exist a constant $\la_0 \in [0,+\infty)$ and an eigenfunction 
$y $ of the problem (\ref{e:ep}), (\ref{e:BCxpm}) such that 
\begin{gather} \label{e:MinC}
\vep (s) = \left\{ \begin{array}{l}
 \n^2_1 , \text{ if } \im ( y^2 (s)) <0;  \\
 \n^2_2 , \text{ if } \im ( y^2 (s)) >0;
\end{array} \right. 
\\
\text{$\im ( \vep (s) y^2 (s)+ \ka^{-2} (y' (s))^2 ) = \la_0 $ for all $s \in [s_-,s_+]$.}
\label{e:TransC}
\end{gather}
\end{thm}


\begin{rem} \label{r:SwitchPoints}
If $\vep (\cdot) \in \F_{s_-,s_+}$ 
and an eigenfunction $y (\cdot)$ of (\ref{e:ep}), (\ref{e:BCxpm})
are connected by (\ref{e:MinC}), then $\vep (\cdot)$ is a bang-bang control
and the set of switch points of $\vep (\cdot)$ is exactly 
$\{s \in  (s_-,s_+) \setminus \{\p_0\} \ : \ y^2 (s) \in \RR \}$ .
Indeed, (\ref{e arg dec arg inc}) and $\pa_s \Arg_\star y \in L^\infty (s_-,s_+)$
imply that 
for each eigenfunction $y $ of (\ref{e:ep}), (\ref{e:BCxpm}) the set 
$\{ s \in (s_-,s_+) : \im y^2 (s) = 0 \}$ is finite.
From the fact that the direction of rotation 
of $y$ changes at the turning point $\p_0$, one sees that 
$\p_0$ is not a switch point (see Fig. \ref{f:x}).
\end{rem}

\begin{proof}[Proof of Theorem \ref{t:PMP}]
By Lemma \ref{l:MonotFunc}, the trajectory 
$x [[s_-,s_+]] $
intersects $\ii \wh \RR$ at most once. 
So we can apply PMP in the form of \cite[Section 2.2]{SL12} to the following 
two cases with restricted state spaces:
(i) $x [[s_-,s_+]] \subset \CC $ and $\CC$ is the state space for dynamics of $x(\cdot)$,
(ii) $x [[s_-,s_+]] \subset \wh \CC \setminus\{0\} $  and $\CC$ is the state space for dynamics of $\wt x(\cdot) $.

We perceive $\CC$ as $\RR^2$ and equip it with the $\RR$-valued scalar product 
$\< \la , x \>_{\RR^2} = \re (\overline{\la}x) = \re (\la\overline{x})$.
Let $\vep (\cdot)$ be a minimum time control that steers optimal $x (\cdot)$ from $\eta_-$ to $\eta_+$.

\emph{Case (i).} We assume that optimal trajectory $x [[s_-,s_+]]$
 lies in $\CC=\RR^2$.
The $\CC$-valued adjoint variable $\la (\cdot)$ and the state variables  $x(\cdot)$
are considered as $\RR^2$-valued functions defined on $[s_-,s_+]$ with real-valued components $\la_1$, $\la_2$ and $x_1$,$x_2$, i.e., $\la = \la_1 + \ii \la_2$, $x= x_1 + \ii x_2$. 
 The time-independent control Hamiltonian $H$ equals
 $H (\la_0,\la,x,\vep) = \la_0 + \< \la , f (x,\vep) \>_{\RR^2} $, where $\la_0 \ge 0$.

The adjoint equation is $\pa_s \la = -  \pa_x \< \la , f (x,\vep)\>_{\RR^2}  $, where
\begin{align*}
\pa_x  \< \la , f \>_{\RR^2}  & =  \< \la , \pa_{x_1} f \>_{\RR^2}  + \ii  \< \la , \pa_{x_2} f \>_{\RR^2} =
\< \la , \pa_{x_1} f \>_{\RR^2} + \ii  \< \la , \ii \pa_{x_1} f \>_{\RR^2} 
 \\ 
& 
= \re (\la \overline{[-2\ii \ka x]}) + \ii \im (\la\overline{[-2\ii \ka x]})
= -2\overline{ (\ii\ka x)} \la 
\end{align*}
(the complex differentiability of $f$ in $x$ is used here, namely, $\pa_{x_2} f = \ii  \pa_{x_1} f$).
From $x [[s_-,s_+]] \subset \CC$, we see that 
$y_\star \neq 0$ everywhere for any eigenfunction $y_\star$ of (\ref{e:ep}), (\ref{e:BCxpm}),
and that
$
\la'  = 2 \overline{(\ii \ka x)} \la = 2 \la \overline{y'_\star}/\overline{y_\star}
$.
The adjoint variable $\la (\cdot)$ 
takes the form $\la (s) = \overline{y_\star}^2 $ 
for a certain eigenfunction $y_\star$, and so 
$H = \la_0 + \re (\ii \ka y_\star^2 [-x^2 +\vep])  
= \la_0 + \im(y^2 [ x^2 - \vep]),$
where $y :=\ka^{1/2} y_\star $ is another eigenfunction.
Since the Hamiltonian is time-independent, the version of PMP for minimum-time problems 
implies 
\begin{gather} \label{e:H}
\text{$0 = H = \la_0 + \im(y^2 [ x^2 - \vep]) = \la_0 - \im ( \ka^{-2} (y')^2 +  y^2 \vep)$ for all $s$.} 
\end{gather}
This gives (\ref{e:TransC}).
The minimum condition 
$ \im  (- y^2 \vep (s)  )  = 
\min_{\ep \in [\n^2_1,\n^2_2]}  \im (-  y^2 \ep ) $  gives 
(\ref{e:MinC}).  Remark \ref{r:SwitchPoints} implies that $\vep (\cdot)$ is bang-bang.

\emph{Case (ii).} The scheme is the same with the change in of the
 state variable to $\wt x = -1/x$. 
 This leads to a different adjoint variable $\la = \overline{ y_\star'}^2 $, where 
 $y_\star$ is a certain eigenfunction, but to
 the same formulae at the end. 
 \end{proof}

\section{Extremals and quarter-wave stacks} 
\label{s:Syn}


We say that a nondegenerate  interval $[\wt s_-, \wt s_+]$ is stationary for $x$ if 
$x (s) = c$ for all $s \in [\wt s_-, \wt s_+]$ and a certain constant $c $.

Let an eigenfunction $y $ of (\ref{e:ep})-(\ref{e:BCxpm}), 
a constant $\la_0 \in [0,+\infty)$, and functions $\vep \in \F_{s_-,s_+}$ 
and $x $ satisfy (\ref{e:MinC}), (\ref{e:TransC}), and (\ref{e:R})-(\ref{e:R2}). 
Then $y (\cdot)$, $x (\cdot)$, and $\vep (\cdot)$ are called a \emph{y-extremal}, an 
\emph{x-extremal} (see Fig. \ref{f:x}), and 
an \emph{extremal control ($\vep$-extremal)}, resp.,
associated with the \emph{extremal tuple} $(x,\vep,\la_0,y)$ (our terminology 
is adapted to the needs of resonance optimization 
and differs slightly from the standard terminology of extremal lifts \cite{BP03,SL12}).
If, additionally, $\la_0 >0$ ($\la_0 = 0$), these extremals are called \emph{normal} 
(resp., \emph{abnormal}).

Let $(x,\vep,\la_0, y)$ be an extremal tuple on $[s_-,s_+]$. Then, 
by Remark \ref{r:SwitchPoints} and (\ref{e:MinC}), 
$\vep$ is a bang-bang control on $[s_-,s_+]$ and $y$ is a nontrivial solution to 
the autonomous equation
\begin{gather} \label{e:Eys}
-y'' (s) = \ka^2 y(s) \ \E (y) (s) , \quad \text{ where }
\E (y) (\cdot ) := \n^2_1   + (\n^2_2   - \n^2_1 )  \ChiCpl ( y^2 (\cdot) ) 
\end{gather}
and $\ChiCpl (\cdot)$ is the characteristic function of $\CC_+ $.
Besides, $\vep (\cdot) = \E (y) (\cdot) $ on $[s_-,s_+]$.
A function $y \in W^{2,\infty}_{\CC,\loc} (\RR)$ is called a solution to
(\ref{e:Eys}) if (\ref{e:Eys}) is satisfied for a.a. $s \in (s_-,s_+)$.
A solution $y$ is called trivial if $y = 0$ a.e. on $(s_-,s_+)$.
 
Any nontrivial solution $y$ to equation  (\ref{e:Eys}) on $[s_-,s_+]$ 
can be extended from this interval to whole $\RR$ in a unique way (see 
the existence and uniqueness theorem in \cite{KLV17}).
Then it is easy to see from (\ref{e:Eys}) that 
$\im ( \vep (s) y^2 (s)+ \ka^{-2} (y' (s))^2 ) $ is a constant independent of $s$.
Let us denote this constant by $\wt \la_0 $ and define $x (\cdot)$ by  $x(s) = y'(s)/(\ii \ka y(s))$. 
If $\wt \la_0 \ge 0$, the extended to the whole $\RR$ solution $y$ (the corresponding extended tuple $(x (\cdot), \E (y) (\cdot) , 
\wt \la_0 , y ) $)
will be called an \emph{extended y-extremal} (resp., 
\emph{extended extremal tuple}).

By Theorem \ref{t:PMP} and (\ref{e:noSelfint}), every minimum time control is a 
restriction of an extended extremal control on a finite interval $[s_-,s_+]$ containing 
no stationary subintervals. This statement explains the logic of definitions and 
studies of this section. 


In the rest of the section we assume (if it is not explicitly stated otherwise) that $y $ is an extended y-extremal, 
$(x , \vep , \la_0 , y^2 ) $ is the associated extended extremal tuple,
and $\{b_j\}$ is the set of switch points of $\vep (\cdot)$ indexed by the following lemma.

\begin{lem} 
Let $\vep (\cdot)$ be an extended extremal control. 
Then the set of switch point of $\vep$ has a form of strictly increasing 
sequence $\{ b_j \}_{j=-\infty}^{+\infty}$ satisfying $\lim_{j\to\pm \infty} b_j = \pm \infty$.
\end{lem} 

\begin{proof}

Let us show that for arbitrary $s_0 \in \RR$, there exists a switch point of $\vep$ in $(s_0,+\infty)$. 

Indeed, assume that there are no switch points in $(s_0,+\infty)$ and 
$\vep (s) = \ep $ for $s > s_0 $ with a constant $\ep \in \{\n_1^2,\n_2^2\}$.
Then, for $s>s_0$, one one hand, we have $y^2[(s_0,+\infty)]) \cap \RR = \varnothing$,
and, on the other hand  (\ref{e:yxRepr}) holds with $|c_-|+|c_+| \neq 0$.
This easily lead to a contradiction.

By Remark \ref{r:SwitchPoints}, there exists at most finite number of switch points in each finite interval.
The symmetry (S0)$\Leftrightarrow$(S1) completes the proof.
\end{proof}

\begin{rem} \label{e:L0}
The special role of the no-return line $\Li_0$ and the extended half-planes $\Hp_0^\pm$ of Remark \ref{r:noreturn} 
is that 
$x(s) \in \Li_0$  if and only if $s$ is the turning point $\p_0$ of $y$ (see Fig. \ref{f:x}).
If the trajectory $x[\RR]$ does not intersect $\Li_0$, then 
either $x[\RR] \subset \Intr \Hp_0^-$ and we put $\p_0 = +\infty$, or $x[\RR] \subset \Intr \Hp_0^+$ and 
we put $\p_0 = -\infty$.
\end{rem}

\begin{rem} \label{r:wtbj}
Let $\p_0 \in \RR$. Then we can introduce an additional enumeration for switch points denoting 
by $\wt b_j$, $j \in \NN$ ($j \in -\NN$), all switch points greater than $\p_0$ 
(resp., all switch points less than $\p_0$) in increasing order. 
We put $\wt b_0 := \p_0$, but recall that $\p_0 $ is not a switch point
 due to (\ref{e arg dec arg inc}) and (\ref{e:MinC}) (see Fig. \ref{f:x}).
 \end{rem}

\begin{rem} \label{r:StInt}
If $[\wt s_-, \wt s_+]$ is a maximal stationary interval for $x(\cdot)$, then it has 
the form $[b_{j},b_{j+1}]$ for a certain $j$ and, in the interval $(\wt s_-, \wt s_+)$,
either $ \vep (\cdot) \equiv \n_1^2 \equiv x^2 (\cdot) $, or $\vep (\cdot) \equiv \n_2^2 
\equiv x^2 (\cdot)$. 
Stationary intervals cannot be discarded as pathologies irrelevant to resonance optimization problems.
Their role is shown in Theorems 
\ref{t:MinDec2}, \ref{t:ExOpt2} and can bee seen in the existence of the straight segment 
of $\Pa^\odd_{\Dr}$ in Fig. \ref{f:Par} (b).
\end{rem}


\subsection{Iterative calculation of switch points}
\label{ss:CalcSwitch}


For all $s \in \RR$, let us define the $\wh \CC$-valued  function
\begin{equation} \label{e:vth=ep}
\vth (s) = \I (\vep^{-1/2} (s) x(s)) \qquad  \text{(recall that $\I (z) := (1-z)/(1+z)$)},
\end{equation}
which we call the \emph{ilog-phase} for the control $\vep(\cdot)$.
Let us take an arbitrary layer $I=(b_j,b_{j+1})$ (i.e., a maximal interval of constancy of $\vep (\cdot)$).
Then, in $I$, the function $\vth (\cdot) $ is continuous and equals to  
$\vth_{\ep} (\cdot)$ of (\ref{e:vthep}) with $\ep := \vep (b_j+0)$.
In $I$, the representations (\ref{e:yxRepr}) and (\ref{e:vth=evth}) with 
$\vka :=\ep^{1/2} \ka $ 
give 
\begin{gather*} 
\vth (s) = \vth (\wt s) \ee^{-2\ii \vka  (s-\wt s)} , \qquad s, \wt s \in I ,
\end{gather*}
which have to be understood as $\vth (s) \equiv \infty$ on $I$ if $\vth (\wt s) = \infty$.
Note that 
$I$ is an interval of constancy for $x$ exactly when $c_+ = 0$, or $c_- = 0$ in (\ref{e:yxRepr}).
Hence, in the case when $I$ is not an interval of constancy, (\ref{e:yRepr}) and (\ref{e:xRepr}) holds in $I$.

\begin{lem} \label{l:bj+1} 
Assume that $I = (b_j,b_{j+1})$ is not a stationary interval 
for $\vep (\cdot)$. Let $\ep = \vep (b_j+0)$ and $\vka = \ep^{1/2} \ka $.
Let a function $\Phi ( \cdot ) \in C [b_j , +\infty )$ satisfy 
\begin{align} \label{e:Phi}
\Phi ( s ) = \vth^{1/2} (b_j+0) \ee^{-\ii \vka  (s-b_j)} .
\end{align} 
Denote by $\tp (\zeta)$ the minimal 
$s \in (b_j , +\infty)$ so that 
\begin{gather*} \label{e:J-1(cp)}
\text{$\Phi ( s ) \in \J^{\{-1\}} [\zeta \RR_+]$, where $\J^{\{-1\}} [\zeta \RR_+] := \{ z \in \CC \ : \J (z) = \zeta c  \text{ for a certain }  c \in \RR_+  \} $} 
\end{gather*}
if such $s$ exists, otherwise put $\tp (\zeta) = +\infty$.

\item[(i)] Let $\p_0 \not \in I $ 
and $\mp \pa_s \Arg_\star y (s) >0$ on $I$. Then 
$b_{j+1} = \tp (\mp \ii \J(\Phi (b_j)))$.

\item[(ii)] Let $\p_0 \in \overline{I}=[b_j,b_{j+1}]$ and $x (\p_0) \neq \infty$. Then 
$b_{j+1} = \tp (\J(\Phi (b_j))) $.

\item[(iii)] Let $\p_0 \in \overline{I}$ and $x (\p_0) = \infty$. Then 
$b_{j+1} = \tp (-\J(\Phi (b_j))) $.
\end{lem}

\begin{proof}
By the definition of $\Phi (\cdot)$ and (\ref{e:vth=evth}), 
\begin{align}\label{e:vth=Phi2}
\text{$\vth (s) = \Phi^2 (s) $ for $s \in I$.}
\end{align}
It follows from (\ref{e:yRepr}) that, for a certain constant $\wt A \neq 0$,  one has 
$y (s) = \wt A \J (\Phi (s))$, $s \in I$.
Since $b_j$ is a switch point, $y^2 (b_j) \in \RR \setminus \{0\}$.
The switch point $b_{j+1}$ is the smallest time point  after the time point $b_j$ 
when the trajectory of 
$y^2$ again reaches $\RR \setminus \{0\}$.

In the case of statement (i), the rotation of $y$ around $0$ goes in one direction.
So $\Arg_\star y$ changes on $\pi/2$ when $s$ passes from $b_j$ to $b_{j+1}$. This and $y (s) = \wt A \J (\Phi (s))$, $s \in I$, give the desired statement.
In the case (ii), $y(s)$ rotates in one direction till $s=\p_0$, then rotate back 
(see (\ref{e arg dec arg inc})) and 
$\Arg_\star y $ reaches at $s=b_{j+1}$ the value of $\Arg_\star y (b_j)$. So 
$y (\cdot) = \wt A \J (\Phi (\cdot))$ gives statement (ii).
In the case (iii), $y(s)$  passes through $0 = y(\p_0)$ changing the direction of rotation 
and, due to (\ref{e:A3}),
 $\Arg_\star y (s)$ reaches at $s=b_{j+1}$ the value $\pi/2 + \Arg_\star y (s)$.
Due to $y (\cdot) = \wt A \J (\Phi (\cdot))$, this translates to (iii) 
\end{proof}

The condition $\p_0 \in I$ can be checked with the use of the following fact:
the no-return lines $\Li_\tau $ 
(with $\tau \in [0,-2\Arg_0 \ka$]) 
are mapped by $\I (\cdot)$ onto
\begin{align} \label{e:I[noreturn]}
& \I [\Li_\tau] = \TT_{R_\tau} (Z_\tau) , 
\end{align}
where 
$Z_\tau =   \frac{ \wh \ka^2 + e^{ \ii (\pi - 2\tau)}}{\wh \ka^2 - e^{ \ii (\pi - 2\tau)}} $, 
$R_\tau = \frac{2}{|\wh \ka^2 - e^{ \ii  (\pi - 2\tau)}|} $, and $\wh \ka = \ka/|\ka|$.
 In particular,  due to (\ref{e:vth=ep}),
\begin{align} \label{e:TT0}
\text{ $\mp \pa_s \Arg_\star y (s) >0$} & \text{ if and only if 
$\pm |\vth (s\pm 0) - Z_0 | \mp R_0 >0$}; \\
\text{ $x \in \ii \CC_\pm $} & \text{ if and only if 
$\pm |\vth (s \pm 0) | \mp 1 >0$}. 
\label{e:TTiRR}
\end{align}


\subsection{Abnormal extremals and quarter wave layers}
\label{ss:Abn}

\begin{prop} \label{p:AbnExt} 
Let $(x , \vep , \la_0 , y ) $  be an extended extremal tuple.

\item[(i)]  Let $\la_0 = 0$ (i.e., the extremal tuple is abnormal). 
Then $s$ is a switch point of $\vep (\cdot)$ if and only if 
$x^2 (s) \in \RR \setminus \{0\}$ and $s$ is not an interior point of a 
stationary interval for $x (\cdot)$.

\item[(ii)] Let $x^2 (b_j) \in \RR$ for at least one switch point $b_j$, then $\la_0 = 0$.
\end{prop}

\begin{proof}
\emph{(i)} 
From (\ref{e:H}) and $\la_0 = 0$, we get the equality
$\im(y^2 [ x^2 - \vep]) \equiv 0 $. 

Consider a switch point $b_j$. By Remark \ref{r:SwitchPoints},
$b_j$ is not the turning point $\p_0$ of $y$, $y^2 (b_j) \in \RR \setminus \{0\}$, 
and $x(b_j) \not \in \{0,\infty\}$. From this and $\im(y^2 [ x^2 - \vep]) \equiv 0 $ 
we see that 
$x^2 (b_j) \in \RR \setminus \{0\}$. 
Remark \ref{r:StInt} completes the proof of the `only if' part of (i).

Assume now that $x^2 (s) \in \RR \setminus \{0\}$ and $s$ is not an interior point of a 
stationary interval for $x (\cdot)$. Then either $x^2 (s-0) \neq \vep (s-0)$, 
or $x^2 (s+0) \neq \vep (s+0)$. Plugging the corresponding inequality 
into $\im(y^2 [ x^2 - \vep]) \equiv 0 $, we see that $\im y^2 (s) = 0$,  and so $s$ is a switch point due to Remark \ref{r:SwitchPoints} and the fact that $s \neq \p_0$ (the latter follows from Remark \ref{e:L0}).

(ii) It follows from (\ref{e:H}), that 
$ \la_0 = - \im(y^2  (b_j) [ x^2 (b_j) - \vep (b_j \pm 0)]) $. Since $y^2  (b_j)$ and 
$x^2 (b_j)$ are real, one gets $\la_0 = 0$.
\end{proof}

\begin{prop} \label{p:[ep1ep2]} 
Let $(x , \vep , 0 , y) $  be an extended abnormal extremal tuple. 

\item[(i)] The intersection of the trajectory $x[\RR]$ with the interval 
$[\n_1,\n_2]$ 
(with $[-\n_2,-\n_1]$) consists of at most 
one point $x_+$ (resp., $x_-$). Assume that the point $x_+ $ does exist. 
Then one of the following three cases takes place:
\begin{itemize}
\item[(i.a)] In the case $x_+ \in (\n_1,\n_2)$, there exists a unique
$s \in \RR$ such that $x(s) = x_+$. This $s$ equals to a certain switch point $b_j$, $j \in \ZZ$; moreover, $\im x'(b_j - 0) $ and $\im x'(b_j + 0) $ are nonzero and of opposite sign.
If $\mp \im x'(b_j - 0) >0$
and $\pm \im x'(b_j + 0) >0$, then 
$\vep (b_j \mp 0) = \n^2_1$ and $\vep (b_j \pm 0) = \n^2_2$.   
\item[(i.b)] In the case $x_+ = \n_1$ (in the case $x_+ = \n_2$), the set  
$\{s \in \RR : x(s) = x_+ \}$ consists 
of one stationary interval of the form $[b_j,b_{j+1}]$, and 
$\vep (b_j - 0) = \vep (b_{j+1} + 0) \neq x_+^2$.
\end{itemize}

\item[(ii)]  If  $x(b_m) \in (\n_1,+\infty)$ and  
$\vep (b_m+0) = \n^2_1$ (for a certain $m \in \ZZ$), then 
$x(b_{m+1}) \in (0,\n_1)$ 
(and $\vep (b_{m+1} +0) = \n^2_2$).
\item[(iii)] If $x(b_m) \in (0,\n_2)$ and 
$\vep (b_m+0) = \n^2_2$, then 
$x(b_{m+1}) \in (\n_2,+\infty)$ 
(and $\vep (b_{m+1} +0) = \n^2_1$).
\end{prop}

\begin{rem} \label{r:[ep1ep2]}
One can formulate an analogue of statement (i.a)-(i.b) for the point $x_-$ using, e.g., 
the time reversal symmetry (S0) $\Leftrightarrow$ (S1).
\end{rem}

\begin{proof}[Proof of Proposition \ref{p:[ep1ep2]}]
Statements (ii)-(iii) follow from Remark \ref{r:noreturn} about the lines 
$\Li_\tau$
of no-return and the sign of the value of the bracket $(-x^2 + \vep)$ in 
$x' = \ii \ka (- x^2  + \vep )$ for $x \in \RR$. Statement (i) is a 
consequence of (ii)-(iii) as it is explained in the following remark.
\end{proof}

\begin{rem}
It follows from statements (ii)-(iii) of Proposition \ref{p:[ep1ep2]}  
that in the cases (1) $x(b_m) \in (\n_2,+\infty)$,
$\vep (b_m+0) = \n^2_1$ and (2)  $x(b_m) \in (0,\n_1)$, 
$\vep (b_m+0) = \n^2_2$, we get for $b_{m+1}$ 
the cases (2) and (1), respectively. So the alternation of the cases (1) and (2) is stable in the sense 
that, provided it has been started once, it continues infinitely.
Now note that if for a certain $j$ we get $x(b_j) \in (\n_1,\n_2)$,
then for $b_{j+1}$ we are in the case (2) or in the case (1) 
depending on value of $\vep (b_j+0)$. 
This implies $(i.a)$ of Proposition \ref{p:[ep1ep2]}. 
The modification of this argument for the case when the trajectory of $x$ gets 
into one of stationary points $\n_1$, $\n_2$ is straightforward.
\end{rem}

When the values $x(b_j)$ and  $x(b_{j+1})$ of an abnormal x-extremal in two consecutive switch points
are real, the relation between them and the corresponding length $b_{j+1} - b_j$ 
of the layer is given by the next theorem.

With a matrix $M=(a_{jm})_{j,m=1}^2$ 
belonging to $\GL_2 (\CC) = \{ M \in \CC^{2\times2} \ : \ \det M \neq 0\}$, 
one can associate the Möbius transformation  
$f_M:  z \mapsto \frac{a_{11} z+a_{12}}{a_{21} z + a_{22}}$.
The map $M \mapsto f_M  $ is a group homomorphism. 
Since $f_{cM} = f_M$ for $c \in \CC \setminus\{0\}$, every 
 Möbius transformation can be represented as $f_{\M_1}$ with 
 $M_1 \in \SL_2 (\CC) := \{M \in  \CC^{2\times2} \ : \ \det M = 1 \}$.

\begin{thm} \label{t:AbnExtWidth}
Let $(x , \vep , 0 , y) $ be an extended abnormal extremal 
tuple. 
Let $\ep = \vep (b_j+0) $ be the value of $\vep (\cdot)$ in the layer $I=(b_j, b_{j+1} )$.
Let 
$ q:=-\frac{\im \ka}{\re \ka}$ and $q_1 := \I (-e^{-q\pi}) $ (note that $q_1 = \tfrac{1+e^{-q\pi}}{1-e^{-q\pi}} >1 $).
  
\item[(i)] Assume that 
$x[ \, \overline{I} \, ] := \{ x(s) \ : \ s \in [b_j,b_{j+1}]\}$ 
does not intersect $\ii \wh \RR$. Then
\begin{gather} \label{e:1/4}
\text{$b_{j+1} - b_j = \frac{\pi}{2 \ep^{1/2} \re \ka}$ \quad (quarter-wave layer)} 
\end{gather}
i.e., the layer width is 1/4 of the wavelength in the material of the layer. Moreover, 
$x(b_{j+1}) = f_{\scriptscriptstyle M (\ep)} (x(b_j))$, 
where $M (\ep) := 
 \begin{pmatrix} 
1 & \ep^{1/2} q_1 \\
\ep^{-1/2} q_1 & 1
  \end{pmatrix} $.

\item[(ii)] Assume that $x(s_0) \in \{0,\infty\} $ 
for a certain $s_0 \in [b_j,b_{j+1}]$. Then $s_0 = (b_j+b_{j+1})/2$ and
\begin{align} \label{e:1/2}
& \text{$b_{j+1} - b_j = \frac{\pi}{\ep^{1/2} \re \ka}$ \quad (half-wave layer);} 
\\
 \label{e:x(s_0)=inf}
\text{besides, \qquad} &\text{$-x(b_j) = x(b_{j+1}) \in (0,\ep^{1/2})$ \quad in the case $x(s_0) = \infty$;}\\
&\text{$-x(b_j) = x(b_{j+1}) \in (\ep^{1/2},+\infty)$ \quad in the case $x(s_0) = 0$.}
\label{e:x(s_0)=0}
\end{align}
\end{thm}

\begin{proof}
(i) First, note that the combination of (\ref{e:1/4}) with (\ref{e:vth=evth}) gives 
$ \I (\ep^{-1/2} x(b_{j+1})) =   -\ee^{ -q\pi} \I (\ep^{-1/2} x(b_j)) $,
and, in turn, $x(b_{j+1}) = f_{\scriptscriptstyle M (\ep)} (x(b_j))$.
Let us prove (\ref{e:1/4}).

Assume that $I$ is not a stationary interval for $x$. 
Then $x(b_j) \not \in \{\pm \ep^{1/2}\}$.
Since $x[\bar I] \cap \ii \wh \RR = \varnothing$, it is easy to see from 
Remark \ref{r:noreturn} and Proposition \ref{p:AbnExt} 
that $\p_0 \not \in I$. So Lemma \ref{l:bj+1} (i) is applicable.
The case $\mp \pa_s \Arg_\star y (s) >0$ on $\overline{I}$ is equivalent to 
$\{x(b_j),x(b_{j+1}) \} \subset \RR_\mp $,
and so, due to (\ref{e:TTiRR}), 
is equivalent to $\{(\vth (b_j+0))^{\mp 1}, (\vth (b_{j+1}-0))^{\mp 1} \} 
\in (-1,1) \setminus \{0\} $.

To be specific, consider the case when 
$\pa_s \Arg_\star y (b_{j+1}) <0$ and $\vth (b_j+0) \in (-\infty,-1)$. Then we know that 
$\vth (b_{j+1}-0) \in \RR \setminus [-1,1] $, $\Phi (b_j) \in \ii (1,+\infty)$,
 and $\Phi (b_{j+1}) \in \RR \cup \ii \RR \setminus \overline{\DD}$, where 
$\Phi (s) := \vth^{1/2} (b_j+0) \ee^{-\ii \vka  (s-b_j)}$ is from Lemma \ref{l:bj+1}. 
Since $\J (\Phi (b_j)) \in \ii \RR_+$, applying Lemma \ref{l:bj+1} (i),
we see from the well-known properties of the Kutta-Zhukovskii transform that 
$
\J^{\{-1\}} [-\ii \J (\Phi (b_j))  \RR_+]= \J^{\{-1\}} [\RR_+] =
\RR_+ \cup (\TT \cap \ii \CC_-).
$
Since $|\Phi (b_{j+1})|>1$ and $b_{j+1} $ is the first point where the trajectory of $\Phi$ 
crosses $\RR_+ \cup (\TT \cap \ii \CC_-)$, we see that  $\Phi (b_{j+1}) \in (1,+\infty)$.
Thus, the definition of $\Phi (\cdot)$ implies (\ref{e:1/4}). 
The case $\pa_s \Arg_\star y (b_{j+1}) <0$, $\vth (b_j+0) \in (1,+\infty)$, 
and the cases when $\pa_s \Arg_\star y (b_j) >0$ can be considered similarly.

Assume now that $I$ is a stationary interval for $x$. Then (\ref{e:yxRepr}) holds on $I$
with $c_- =0 $, or $c_+ = 0$. Thus, (\ref{e:1/4}) follows  
from (\ref{e:MinC}).

(ii) It follows from $x(s_0) \in \{0,\infty\} $ that $\vep (\cdot)$ is symmetric 
w.r.t. $s_0$, $s_0 = \p_0 = (b_j + b_{j+1})/2$, and $x(b_j) = - x(b_{j+1}) \in \RR_-$. In particular,
either the case (ii), or the case (iii) of Lemma \ref{l:bj+1} is applicable.

Consider the case $x(\p_0) = \infty$ and $x(b_j) \in (-\ep^{1/2},0)$. 
Then $\vth (\p_0) = -1$ and 
\begin{gather} \label{e:thbj>1}\text{
$\vth (b_j+0) = (\vth (b_{j+1} - 0))^{-1} \in 
(1,+\infty)$. 
} \end{gather}
So $\Phi (b_j)$ and $\J (\Phi (b_j))$ are in $(1,+\infty)$. From Lemma \ref{l:bj+1} (iii), 
we see that $b_{j+1}$ is the smallest point in $(b_j,+\infty)$, where the trajectory of 
$\Phi (\cdot)$ reaches the set
\[
\J^{\{-1\}} [- \J (\Phi (b_j)) \RR_+] = \J^{\{-1\}} [ \RR_-] = \RR_- \cup (\TT \cap \ii \CC_+) .
\]
By (\ref{e:vth=Phi2}) and (\ref{e:thbj>1}), 
$\vth (b_{j+1} - 0) = \Phi^2 (b_{j+1}) \in (0,1)$. So $\Phi (b_{j+1}) \in (-1,0)$.
Thus, the definition of $\Phi$ implies (\ref{e:1/2}).

Let us show that in the case $x(\p_0) = \infty$, the inclusion $x(b_j) \in (-\infty,-\ep^{1/2}]$ is impossible.
If $x(b_j) = -\ep^{1/2}$, this is obvious since $I$ is stationary interval 
and $\p _0 \not \in I$.
Suppose $x(b_j) \in (-\infty,-\ep^{1/2})$. Then (\ref{e:R2}) implies $\wt x' (b_j) \in \CC_+$, $\wt x' (\p_0) = 0$, and $\wt x' (\p_0) \in \CC_+$.
So there exists $s \in (b_j,\p_0)$ so that $x(s) \in \RR_-$. Since there are no 
switch points between $b_j$ and $b_{j+1}$, we get a contradiction with Proposition \ref{p:AbnExt}.

The case $x(s_0) = 0$ can be treated similarly.
\end{proof}

Denote 
\[\text{$\eta_j := \vep^{1/2} (b_j+0)$, \quad $\ret_{j+1} := \eta_j/\eta_{j+1}$, \quad  and 
$\ret = \n_2/\n_1$}.
\]

\begin{lem} \label{l:Mep12}
The transform $f_{M(\eta^2_{j+1})M(\eta^2_j)}$ has exactly two fixed points equal to 
$ x^\pm_{\ret_{j+1}}$, where  
$
x^\pm_r = \frac{(\n_1 \n_2)^{1/2}}{2} 
\left[ \fq ( r^{-1/2} - r^{1/2}) \pm 
\left( \fq^2 ( r^{-1/2} - r^{1/2})^2 +4 \right)^{1/2} \right] 
\in \RR_\pm 
$ (for the notation see Theorem \ref{t:AbnExtWidth} and the preceding paragraph).
Besides, $f_{M(\eta^2_{j+1})M(\eta^2_j)} = f_{M_{\scriptstyle \ret_{ j+1}}}$, where
$M_{\ret_{j+1}} =
\Sim_{\ret_{j+1}}^{-1} 
   \begin{pmatrix} \la & 0\\
 0 & 1/\la  
  \end{pmatrix}
  \Sim_{\ret_{j+1}} $,  \quad 
$\la = \J_1^{\{-1\}} \left(  \frac{\fq^2 \J (\ret) +1}{\fq^2 - 1} \right) >1$, and 
$
\Sim_r :=  \begin{pmatrix} 1 & -x^-_r \\
 1 & -x^+_r 
  \end{pmatrix}
  $.
(In particular, 
the transform  $f_{M(\eta^2_{j+1})M(\eta^2_j)}$ is of hyperbolic type with 
the repulsive fixed point $ x^-_{\ret_{j+1}}$ and the attractive 
fixed point $ x^+_{\ret_{j+1}}$, for the classification see, e.g. \cite{N98}.)
\end{lem}

\begin{proof}
The determinant of the matrix
\begin{gather*}
M(\eta^2_{j+1}) M(\eta^2_j)  = 
  \left( \begin{array}{cc} 1+\fq^2 \ret_{j+1}^{-1} & \fq (\eta_j+\eta_{j+1}) \\
  \fq (\eta_j^{-1}+\eta_{j+1}^{-1})  & 1+\fq^2 \ret_{j+1} 
  \end{array}
  \right) .
\end{gather*}
equals $(\fq^2-1)$.
Then $f_{M(\eta^2_{j+1}) M(\eta^2_j)} = f_{\M} $,
where $\M := \frac{1}{\fq^2-1} M(\eta^2_{j+1}) M(\eta^2_j) $ and $\det \M = 1$.
So the matrix $\M$ has two eigenvalues $\la^{\pm1}$ with $\la \neq 0$ 
and its trace is equal to 
\[
2 \J (\la) = \frac{2 +\fq^2 (\eta_{j+1}/\eta_j + \eta_j/\eta_{j+1})}{\fq^2 -1 } = 
\frac{2 +\fq^2 (\ret + \ret^{-1})}{\fq^2 -1 }> 2.
\]
One can assume that $\la \in (1,+\infty)$ and 
\begin{gather*}
\la = \J^{\{-1\}}_1 \left( \frac{\fq^2 \J (\ret) +1 }{\fq^2 -1 } \right) = 
\frac{\fq^2 \J (\ret) +1 }{\fq^2 -1 }+ 
\sqrt{\frac{(\fq^2 \J (\ret) +1)^2 }{(\fq^2 -1)^2 } -1}  
\end{gather*}
The reduction of $M_{\ret_{j+1}}$ to a diagonal form can be obtained by standard calculations.
\end{proof}

\begin{cor} \label{c:1/4per}
Let $(x , \vep , 0 , y) $  be an extended abnormal extremal tuple.
\item[(i)] If $x (b_j) \in \RR_\pm$, then
$ f_{\scriptscriptstyle \Sim_{\scriptstyle \ret_{j+1}}} \left( x(b_{j\pm 2 m} \right) = 
\la^{\pm 2m} f_{\scriptscriptstyle \Sim_{\scriptstyle \ret_{j+1}}} \left( x(b_j) \right) $ 
for $m \in \NN$.

\item[(ii)] If $\wt b_0 = p_0 \in \RR$ and $x (\wt b_0) \in \{0,\infty\}$, then 
$f_{\Sim_{\scriptstyle \ret_1}} (x(\wt b_{ 2 m} )) = 
\la^{2m} f_{S_{\scriptstyle \ret_1}} (x(\wt b_0) ) $ for $m \in \NN$.

\item[(iii)] $\lim_{m \to \pm \infty} x(b_{j + 2m}) = x^\pm_{\ret_{j+1}}$.

(Here the notations of Lemma \ref{l:Mep12} and Remark \ref{r:wtbj} are used.)
\end{cor}
\begin{proof}
The corollary follows from Lemma \ref{l:Mep12} and Theorem \ref{t:AbnExtWidth}.
\end{proof}

\begin{ex}[periodic $x$-extremal] \label{ex:Perx-extr}
(i) Consider the extended extremal tuple  $(x , \vep , 0 , y) $ such that 
$x(b_i) = x^+_{\ret}$ and $\vep (b_i+0)=\n^2_2$ for a certain switch point $b_i$.
Without loss of generality we can choose the enumeration of $\{b_j\}_{j\in \ZZ}$ 
so that $i=0$.
Then Lemma \ref{l:Mep12} and Corollary \ref{c:1/4per} imply that
$(x , \vep , 0 , y) $  has the following properties:
(a) $x$ and $\vep$ are periodic with the period 
$\frac{\pi (\n_1^{-1}+\n_2^{-1})}{2  \re \ka}$ 
(which equals to the sum of the quarter wave length in the media $\n^2_1$ and $\n^2_2$), in particular, $x(b_{2j}) = x(b_0)$ and $x (b_{2j+1}) = x (b_1)$ for all $j \in \ZZ$;
(b) $x^+_{\ret} \in (0,\n_1)$ and $x^+_{1/\ret} = x (b_1) \in (\n_2,+\infty)$.

(ii) Applying to the extremal (i) 
the time-reversal symmetry (S0) $\Leftrightarrow$ (S1), one can reflect the curve described in (i) w.r.t. 0.
\end{ex}

\subsection{Normal extremals}
\label{ss:Norm}

If $(x , \vep , \la_0 , y) $ is a normal extremal 
tuple, then without loss of generality we assume that $\la_0=1$. Indeed, since $\la_0>0$,
one can divide (\ref{e:TransC}) on $\la_0$ and replace $y$ by $y/\la_0^{1/2}$.

The no-return line $\Li_0$ (which contains all possible states $x(\p_0)$ at turning points $s=\p_0$),
and the two axes $\wh \RR$ and $\ii \wh \RR$ split $\wh \CC$ into sectors $\Sec_j$, $j=1$, \dots, $6$,
which we number in the following way:
$\Sec_1 = \CC_1 $, \quad $\Sec_2 = \Sec (\pi/2 , \pi/2 -  \Arg_0 \ka) $, \quad $\Sec_3 := (\pi/2 -  \Arg_0 \ka ,  \pi )$, \quad $\Sec_4 = \CC_3 $, \quad $\Sec_5 = \Sec (3\pi/2, 3\pi/2 -  \Arg_0 \ka) $, \quad $\Sec_6 :=  (3\pi/2 -  \Arg_0 \ka, 2\pi) $
(this splitting is shown in Fig. \ref{f:x} and is connected with  notions of ordinary and non-ordinary points \cite{BP03}).

\begin{thm} \label{t:NextWidth}
Let $(x , \vep , 1 , y) $ be an extended normal extremal tuple, $\ep = \vep (b_j+0) $ be the value of $\vep (\cdot)$ in the layer $I=(b_j, b_{j+1} )$.
Assume $x[I] \subset  \Sec_6 \cup \Sec_1 \cup \Sec_3 \cup \Sec_4 $.
Then:
\item[(i)] If $x (b_j) \in \CC_+$, then $x (b_{j+1}) \in \CC_-$ and 
$b_{j+1} - b_j < \frac{\pi}{2 \ep^{1/2} \re \ka}$.

\item[(ii)] If $x (b_j) \in \CC_-$, then $x (b_{j+1}) \in \CC_+$ and 
$\frac{\pi}{2 \ep^{1/2} \re \ka} < b_{j+1} - b_j < \frac{\pi}{\ep^{1/2} \re \ka} $.
\end{thm}

\begin{proof}
Let us consider the case when $x[\bar I] \subset  \ii e^{-\ii \Arg \ka }\CC_+ \cap \ii \CC_+$ and $x (b_j) \in \CC_+$.

Since $\I [\CC_+] = \CC_-$, $x(b_j) \in \CC_+$ yields $\vth (b_j+0) \in \CC_-$
and $\Phi (b_j) \in  \CC_4$. Further, on one hand, (\ref{e:TTiRR}) yields 
$\Phi (b_j) \in \CC_4 \setminus \overline{\DD}$ and 
$\Phi (b_j+1) \not \in \overline{\DD}$. 
On the other hand, (\ref{e:I[noreturn]}) and (\ref{e:TT0}) imply 
that the case $\pa_s \Arg_\star y (s) <0$ of Lemma \ref{l:bj+1} (i) is applicable 
to $I$. So $b_{j+1}=\tp (-\ii \J (\Phi (b_j)))$ is the minimal 
$s > b_j $ at which $\Phi ( s ) $ intersects $\J^{\{-1\}}_{1} [-\ii \J (\Phi (b_j)) \RR_+]$
(we use here $\Phi (b_j+1) \not \in \overline{\DD}$). Thus, on one hand, 
$
\Arg_0 \J (\Phi (b_j)) - \Arg_0 \J (\Phi (b_{j+1})) = \pi/2, 
$
and on the other hand,
the squeezing properties of $\J (\cdot)$ imply that  
\[ \Arg_0 \Phi (b_j) < \Arg_0 \J (\Phi (b_j)) < 0 \text{ and }
\Arg_0 \J (\Phi (b_{j+1})) < \Arg_0  (\Phi (b_{j+1})) < -\pi/2 .
\]
Combining these inequalities we get 
$
\Arg_0 \Phi (b_j) + \pi/2 <\Arg_0 \Phi (b_j) - \Arg_0 \Phi (b_{j+1}) < \pi/2 
$
and, applying the properties of $\Phi $ from  Lemma \ref{l:bj+1}, 
\[
\pi/2 + \frac12 \Arg_0 \vth (b_j+0) <(b_{j+1} - b_j) \ep^{1/2} \re \ka < \pi/2 .
\]
This gives $b_{j+1} - b_j < \frac{\pi}{2 \ep^{1/2} \re \ka}$.
As a by-product, we get $\vth (b_{j+1} - 0) \in \CC_+$ from $\Phi (b_j) \in \CC_3$ and 
the connection (\ref{e:Phi}), (\ref{e:vth=evth}) between $\Phi$ and $\vth$.
Thus, $x(b_{j+1}) \in \CC_-$.

The case when $x[\, \overline{I}\, ] \subset  \ii e^{-\ii \Arg \ka }\CC_+ \cap \ii \CC_+$ and $x (b_j) \in \CC_-$ can be considered similarly.

The cases when $x[\, \overline{I} \, ] \subset  \ii e^{-\ii \Arg \ka }\CC_- \cap \ii \CC_-$ 
can be obtained, e.g., by the time reversal symmetry (S0) $\Leftrightarrow$ (S1).
\end{proof}

\begin{thm} \label{t:SecSwitch}
Let $(x , \vep , 1 , y) $ be an extended normal extremal 
tuple. Then:
\item[(i)] If $x(b_j) \in \Sec_1 \cup \Sec_3 \cup \Sec_5$, then $\vep (b_j-0) = \n_2^2$ and 
$\vep (b_j+0) = \n_1^2$.
\item[(ii)]  If $x(b_j) \in \Sec_2 \cup \Sec_4 \cup \Sec_6$, then $\vep (b_j-0) = \n_1^2$ and 
$\vep (b_j+0) = \n_2^2$.
\end{thm}

\begin{proof}
By (\ref{e:MinC}) and (\ref{e arg dec arg inc}), $y^2 (b_j) \in \RR \setminus \{ 0 \}$. 
By (\ref{e:H}), we have $ -1 = \im(y^2 (b_j ) [x^2 (b_j)-\vep(b_j \pm 0)]) = y^2 (b_j ) \im(x^2 (b_j))$.
Thus, $\im(x^2 (b_j)) = -1/y^2 (b_j )$.  
Combining this equality with the direction (\ref{e arg dec arg inc}) of rotation of $y$ one gets the desired statements.
\end{proof}

\section{Elements of extremal synthesis}
\label{s:Syn2}

Let $\ka \in \CC_4$. 
This section considers the question of classification of all possible $x$-extremals 
on an  interval $[s_-,s_+]$ that satisfy the following properties
\begin{multline} \label{e:xs-inxs+in}
\text{$x(s_-) \in [-\n_2,-\n_1]$, $x (s_+) \in [\n_1,\n_2]$, and}
\\
\text{ there are no stationary subintervals in $[s_-,s_+]$.}
\end{multline}
The description of various cases given below will be essentially used in subsequent sections. In particular, we will exclude these cases one by one to obtain an explicit example of optimizer in Theorem \ref{t:ExOpt} and to prove analytically  
the jump in the Pareto frontier $\Pa^\odd_{\md}$ (see Fig. \ref{f:Par}).
We do not know if all of the cases described below take place (this is not important for our needs).

Let us consider a corresponding extended extremal tuple $(x,\vep,\la_0,y)$ on $\RR$ 
(note that such extension is not necessarily unique and that sometimes abnormal and normal 
extensions are possible simultaneously).
We will say that this extended extremal tuple is \emph{symmetric} w.r.t. $\p_0$ if $x(\p_0) \in \{0,\infty\}$
(in this case, $x (s - \p_0) = - x (-s + \p_0)$ and $\vep (s - \p_0) = \vep (-s + \p_0)$).
Note that (\ref{e:xs-inxs+in}) implies $\p_0 \in (s_-,s_+)$, but in this section 
$s^\cent := (s_- + s_+)/2$ does not necessarily coincides with $\p_0$ because the case 
$x (s_-) \neq x (s_+)$ is allowed.

We consider the partition of $[s_-,s_+]$ by the points belonging to the two classes:
switch points and the points $s$ where the trajectory $x[[s_-,s_+]]$ 
crosses the set $\wh \RR \cup \ii \wh \RR \cup \Li_0$.
Each open interval $I$ of this partition is encoded by a sign `$\EM$' or `$\EP$' depending on the constant value of the 
control $\vep (\cdot)$ in this interval; the sign `$\EM$' corresponds to $\ep_1 = \n_1^2$, `$\EP$' to $\ep_2 = \n_2^2$.
Above these signs we write the index $j$ of one of the sectors $\Sec_j$, $j=1,\dots,6$ (see Section \ref{ss:Norm}), 
where the trajectory of $x (s)$ evolves for $s \in I$. The points where the trajectory of $x$ crosses 
$\wh \RR \cup \ii \wh \RR \cup \Li_0$
are of special importance. Therefore, for every such intersection we write as the lower index the subset
of the corresponding extended line ($\wh \RR$, $\ii \wh \RR$, or $\Li_0$) where the intersection takes
place. If some pattern in the description repeats $m \in \{0\}\cup \NN$ times in row, we put it in the square brackets 
with the upper index $m$, i.e., $\cdots [\cdots ]^m \cdots $.

\begin{ex} \label{ex:onelayerodd}
The notation
\begin{align} \label{e:onelayerodd}
_{(-\n_2,-\n_1]} \overset{3}{\EP}  _{\{\infty\}} \overset{6}{\EP}  _{[\n_1,\n_2)}
\end{align}
means that the exremal trajectory of $x$ starts at a point $x(s_-) \in (-\n_2,-\n_1]$, 
the trajectory goes from $x(s_-)$ to the sector $\Sec_3$, the 
 first and the only layer of $\vep (\cdot) $ is of permittivity $\n_2^2$, the trajectory passes through $\infty$ 
 at $s= \p_0 = s^\cent = (s_- + s_+)/2$ 
in the middle of this layer, goes to the sector $\Sec_6$, and ends at $x(s_+) \in [\n_1,\n_2)$.
In this particular case, it follows from $\p_0 = s^\cent$ and the symmetry w.r.t. $\p_0$ that $x (s_-) = -x(s_+) $.
\end{ex}

Another illustration of the above notation is given by Fig. \ref{f:x}, which shows an $x$-extremal with structure (\ref{e:NS4}) and $m_1=m_2=0$.

Let $(x,\vep,\la_0,y)$ be an extended extremal tuple such that 
 (\ref{e:xs-inxs+in}) holds. Then the structure of $(x,\vep,\la_0,y)$ in the interval $[s_-,s_+]$
belongs to one of the classes 
described by the following statements:

\vspace{.2ex}

\noindent
(E1) If $(x,\vep,\la_0,y)$ is abnormal and symmetric w.r.t. $\p_0$, 
then 
\[
\text{it has on $[s_+,s_-]$ either the structure (\ref{e:onelayerodd}),}
\]
 or one of the following structures with a certain number 
$m \in \NN$ of repetitions:
\begin{gather} \scriptstyle
_{[-\n_2,-\n_1)} \overset{4}{\EM}  _{\{0\}} \overset{1}{\EM} _{(\n_1,\n_2]} 
; \label{e:AS2} 
\\ \scriptstyle
_{(-\n_2,-\n_1]} \overset{3}{\EP} 
\left[_{(-\infty,-\n_2)} \overset{4}{\EM} _{(-\n_1,0)} \overset{3}{\EP} \right]^m 
\ _{\{\infty\}} \left[\overset{6}{\EP} _{(0,\n_1)} \overset{1}{\EM} _{(\n_2,+\infty)} \right]^m \overset{6}{\EP}  _{[\n_1,\n_2)} 
; \label{e:AS3} 
\\ \scriptstyle
_{[-\n_2,-\n_1)} \overset{4}{\EM} 
\left[_{(-\n_1,0)} \overset{3}{\EP} _{(-\infty,-\n_2)} \overset{4}{\EM} \right]^m 
\ _{\{0\}} \left[\overset{1}{\EM} _{(\n_2,+\infty)} \overset{6}{\EP} _{(0,\n_1)} \right]^m 
\overset{1}{\EM}  _{(\n_1,\n_2]} . \label{e:AS4}
\end{gather}
\noindent (E2) If $(x,\vep,\la_0, y^2)$ is normal and symmetric w.r.t. $\p_0$, 
then it has on $[s_+,s_-]$ one of the following structures with certain numbers 
$m_1, m_2 \in \{0\} \cup \NN$ of repetitions:
\begin{gather} \scriptstyle
 _{(-\n_2,-\n_1]} \overset{3}{\EP} 
\left[\overset{3}{\EM} _{(-\infty,-\n_1)} \overset{4}{\EM} \overset{4}{\EP} _{(-\n_2,0)} \overset{3}{\EP} 
\right]^{m_1}  
\overset{3}{\EM} _{\{\infty\}} \overset{6}{\EM} 
\left[ \overset{6}{\EP} _{(0,\n_2)} \overset{1}{\EP} \overset{1}{\EM} _{(n_1,+\infty)} \overset{6}{\EM}
\right]^{m_2} \overset{6}{\EP} _{[\n_1,\n_2)} 
; \label{e:NS1} 
\end{gather}
\vspace{-4ex}
\begin{multline} \scriptstyle
 _{(-\n_2,-\n_1]} \overset{3}{\EP} \overset{3}{\EM} _{(-\infty,-\n_1)} \overset{4}{\EM}
\left[ \overset{4}{\EP} _{(-\n_2,0)} \overset{3}{\EP} 
\overset{3}{\EM} _{(-\infty,-\n_1)} \overset{4}{\EM} \right]^{m_1}  
\overset{4}{\EP} 
_{\{0\}} 
\\ \scriptstyle
\overset{1}{\EP}  \left[ \overset{1}{\EM} _{(n_1,+\infty)} \overset{6}{\EM} \overset{6}{\EP} _{(0,\n_2)} 
\overset{1}{\EP} \right]^{m_2} \overset{1}{\EM} _{(n_1,+\infty)} \overset{6}{\EM} \overset{6}{\EP} _{[\n_1,\n_2)} 
; \label{e:NS2} 
\end{multline}
\vspace{-4ex}
\begin{multline} \scriptstyle
 _{[-\n_2,-\n_1)} \overset{4}{\EM} \overset{4}{\EP} _{(-\n_2,0)} \overset{3}{\EP} 
\left[ \overset{3}{\EM} _{(-\infty,-\n_1)} \overset{4}{\EM} 
\overset{4}{\EP} _{(-\n_2,0)} \overset{3}{\EP} 
\right]^{m_1}  
\overset{3}{\EM} _{\{\infty\}} 
\\
\overset{6}{\EM}
\left[ \overset{6}{\EP} _{(0,\n_2)} 
\overset{1}{\EP} \overset{1}{\EM} _{(\n_1,+\infty)} \overset{6}{\EM} 
\right]^{m_2} \overset{6}{\EP} _{(0,\n_2)} \overset{1}{\EP}  \overset{6}{\EM} _{(\n_1,\n_2]} 
; \label{e:NS3} 
\end{multline}
\vspace{-3ex}
\begin{gather} \scriptstyle
 _{[-\n_2,-\n_1)} \overset{4}{\EM} \left[\overset{4}{\EP} _{(-\n_2,0)} \overset{3}{\EP} 
 \overset{3}{\EM} _{(-\infty,-\n_1)} \overset{4}{\EM} \right]^{m_1}
\overset{4}{\EP} 
_{\{0\}} 
\overset{1}{\EP}  \left[ \overset{1}{\EM} _{(\n_1,+\infty)} \overset{6}{\EM} \overset{6}{\EP} _{(0,\n_2)} 
\overset{1}{\EP} \right]^{m_2} \overset{1}{\EM} _{(\n_1,\n_2]} 
; \label{e:NS4} 
\\ \scriptstyle
 _{(-\n_2,-\n_1]} \overset{3}{\EP} 
\left[\overset{3}{\EM} _{(-\infty,-\n_1)} \overset{4}{\EM} \overset{4}{\EP} _{(-\n_2,0)} \overset{3}{\EP} 
\right]^{m_1}  
 _{\{\infty\}}  
\left[ \overset{6}{\EP} _{(0,\n_2)} \overset{1}{\EP} \overset{1}{\EM} _{(n_1,+\infty)} \overset{6}{\EM}
\right]^{m_2} \overset{6}{\EP} _{[\n_1,\n_2)} 
; \label{e:NS5} 
\end{gather}
\vspace{-3ex}
\begin{multline} \scriptstyle
 _{[-\n_2,-\n_1)} \overset{4}{\EM} \overset{4}{\EP} _{(-\n_2,0)} \overset{3}{\EP} 
\left[ \overset{3}{\EM} _{(-\infty,-\n_1)} \overset{4}{\EM} 
\overset{4}{\EP} _{(-\n_2,0)} \overset{3}{\EP} 
\right]^{m_1}  
 _{\{\infty\}} 
\\ \scriptstyle
\left[ \overset{6}{\EP} _{(0,\n_2)} 
\overset{1}{\EP} \overset{1}{\EM} _{(\n_1,+\infty)} \overset{6}{\EM} 
\right]^{m_2} \overset{6}{\EP} _{(0,\n_2)} \overset{1}{\EP}  \overset{6}{\EM} _{(\n_1,\n_2]} 
; \label{e:NS6} 
\end{multline}

\noindent (E3)
If $(x,\vep,\la_0,\bar y^2)$ is abnormal and is not symmetric w.r.t. $\p_0$, 
then, on $[s_+,s_-]$, it has  in the case $x(\p_0) \in \CC_+ \cap \Li_0$ one of the following structures 
\begin{gather} \scriptstyle
_{(-\n_2,-\n_1]} \left[ \overset{3}{\EP}  
_{(-\infty,-\n_2)} \overset{4}{\EM} _{(-\n_1,0)}  \right]^{m_1} \overset{3}{\EP} _{\Li_0} \overset{2}{\EP} 
_{\ii \RR_+}  \left[ \overset{1}{\EM} _{(\n_2,+\infty)} \overset{6}{\EP} _{(0,\n_1)} \right]^{m_2} 
\overset{1}{\EM} _{(\n_1,\n_2]} ;
\label{e:ANS01}
\\ \scriptstyle
_{(-\n_2,-\n_1]} \left[ \overset{3}{\EP}  
_{(-\infty,-\n_2)} \overset{4}{\EM} _{(-\n_1,0)}  \right]^{m_1} \overset{3}{\EP} _{\Li_0} \overset{2}{\EP} 
_{\ii \RR_+}  \left[ \overset{1}{\EM} _{(\n_2,+\infty)} \overset{6}{\EP} _{(0,\n_1)} \right]^{m_2} 
\overset{1}{\EM} _{(\n_2,+\infty)} \overset{6}{\EP} _{[\n_1,\n_2)} ;
\label{e:ANS02}
\\ \scriptstyle
_{[-\n_2,-\n_1)} \left[\overset{4}{\EM} _{(-\n_1,0)} \overset{3}{\EP}  
_{(-\infty,-\n_2)} \right]^{m_1} \overset{4}{\EM} _{(-\n_1,0)}   \overset{3}{\EP} _{\Li_0} \overset{2}{\EP} 
_{\ii \RR_+}  \left[ \overset{1}{\EM} _{(\n_2,+\infty)} \overset{6}{\EP} _{(0,\n_1)} \right]^{m_2} 
\overset{1}{\EM} _{(\n_1,\n_2]} ;
\label{e:ANS03}
\\
 \scriptstyle
_{[-\n_2,-\n_1)} \left[\overset{4}{\EM} _{(-\n_1,0)} \overset{3}{\EP}  
_{(-\infty,-\n_2)} \right]^{m_1} \overset{4}{\EM} _{(-\n_1,0)} \overset{3}{\EP} _{\Li_0} \overset{2}{\EP} 
_{\ii \RR_+}  \left[ \overset{1}{\EM} _{(\n_2,+\infty)} \overset{6}{\EP} _{(0,\n_1)} \right]^{m_2} 
\overset{1}{\EM} _{(\n_2,+\infty)} \overset{6}{\EP} _{[\n_1,\n_2)} ; 
\label{e:ANS04}
\end{gather}
and in the case when $x(\p_0) \in \CC_- \cap \Li_0$ one of the following structures 
\begin{gather} \scriptstyle
  _{(-\n_2,-\n_1]} \left[ \overset{3}{\EP}  
_{(-\infty,-\n_2)} \overset{4}{\EM} _{(-\n_1,0)}  \right]^{m_1} 
\overset{3}{\EP} _{(-\infty,-\n_2)} \overset{4}{\EM} _{\ii \RR_-} \overset{5}{\EP} _{\Li_0} 
\left[ \overset{6}{\EP} _{(0,\n_1)}   \overset{1}{\EM} _{(\n_2,+\infty)} \right]^{m_2} 
\overset{6}{\EP} _{[\n_1,\n_2)} ;
\label{e:ANS05}
\\ 
\scriptstyle
 _{(-\n_2,-\n_1]} \left[ \overset{3}{\EP}  
_{(-\infty,-\n_2)} \overset{4}{\EM} _{(-\n_1,0)}  \right]^{m_1} 
\overset{3}{\EP} _{(-\infty,-\n_2)} \overset{4}{\EM} _{\ii \RR_-} \overset{5}{\EP} _{\Li_0} 
\left[ \overset{6}{\EP} _{(0,\n_1)}   \overset{1}{\EM} _{(\n_2,+\infty)} \right]^{m_2} 
\overset{6}{\EP} _{(0,\n_1)} \overset{1}{\EM} _{(\n_1,\n_2]} ;
\label{e:ANS06}
\\ \scriptstyle
_{[-\n_2,-\n_1)} \left[ \overset{4}{\EM} _{(-\n_1,0)} \overset{3}{\EP}  
_{(-\infty,-\n_2)}  \right]^{m_1} 
\overset{4}{\EM} _{\ii \RR_-} \overset{5}{\EP} _{\Li_0} 
\left[ \overset{6}{\EP} _{(0,\n_1)}   \overset{1}{\EM} _{(\n_2,+\infty)} \right]^{m_2} 
\overset{6}{\EP} _{[\n_1,\n_2)} ;
\label{e:ANS07}
\\ \scriptstyle
_{[-\n_2,-\n_1)} \left[ \overset{4}{\EM} _{(-\n_1,0)} \overset{3}{\EP}  
_{(-\infty,-\n_2)}  \right]^{m_1} \overset{4}{\EM} _{\ii \RR_-} \overset{5}{\EP} _{\Li_0} 
\left[ \overset{6}{\EP} _{(0,\n_1)}   \overset{1}{\EM} _{(\n_2,+\infty)} \right]^{m_2} 
\overset{6}{\EP} _{(0,\n_1)} \overset{1}{\EM} _{(\n_1,\n_2]} ,
\label{e:ANS08}
\end{gather}
with certain numbers 
$m_1, m_2 \in \{0\} \cup \NN$ of repetitions.

\noindent (E4)
If the extremal tuple $(x,\vep,\la_0, y)$ is normal, is not symmetric w.r.t. $\p_0$, and $x(\p_0) \in \CC_+ \cap \Li_0$, 
then it passes the no-return lines $\Li_0$ and $\ii \wh \RR$
in one of the two following ways: 
\begin{gather} \label{e:E41} \scriptstyle
 \dots \overset{3}{\EP} \overset{3}{\EM} _{\Li_0} 
 \overset{2}{\EM}  \overset{2}{\EP} _{\ii \RR_+} 
\overset{1}{\EP } \overset{1}{\EM } \dots 
\quad \text{\normalsize and } \quad
\dots  \overset{3}{\EP}  _{\Li_0} \overset{2}{\EP}   
 _{\ii \RR_+} \overset{1}{\EP } \overset{1}{\EM }  \dots 
\end{gather}
If the extremal tuple $(x,\vep,\la_0, y)$ is normal, is not symmetric w.r.t. $\p_0$, 
and $x(\p_0) \in \CC_- \cap \Li_0$, 
then it passes the no-return lines $\Li_0$ and $\ii \wh \RR$
in one the two following ways: 
\begin{gather}  \label{e:E42} \scriptstyle
 \dots  \overset{4}{\EM}
\overset{4}{\EP} _{\ii \RR_-} \overset{5}{\EP} \overset{5}{\EM} _{\Li_0}  \overset{6}{\EM} 
 \overset{6}{\EP}  \dots 
 \quad \text{\normalsize and } \quad
  \dots  \overset{4}{\EM}
\overset{4}{\EP} _{\ii \RR_-} \overset{5}{\EP} _{\Li_0}  \overset{6}{\EP} 
 \dots 
\end{gather}
All the possible structures of the extremal tuple $(x,\vep,\la_0, y)$ on $[s_-,s_+]$
in the case when it is normal and is not symmetric w.r.t. $\p_0$ can be obtained by concatenation of the 
4 sequences in (\ref{e:E41}), (\ref{e:E42}) with the beginning and ending sequences of various structures from (\ref{e:NS1})-(\ref{e:NS4})
according to the rules described in Theorems \ref{t:NextWidth} and \ref{t:SecSwitch}. 
 
The proof of (E1)-(E4) can be obtained from 
Lemma \ref{l:bj+1},
Theorems \ref{t:AbnExtWidth}, \ref{t:NextWidth}, \ref{t:SecSwitch}, and  Propositions \ref{p:AbnExt}, 
\ref{p:[ep1ep2]} taking into account 
the values of the vector fields  $f (x,\n_1^2)$ and $f (x,\n_2^2)$ at the points $x \in \RR \cup \ii \RR$,
and the equality $\wt f (0,\n_1^2) =\wt f (0,\n_2^2) = - \ii \ka$.

\begin{rem} \label{r:NSinfty}
A special clarification is needed for the sequence
\[ \scriptstyle
 \dots _{(-\infty,-\n_1)} \overset{4}{\EM} \overset{4}{\EP} _{(-\n_2,0)} \overset{3}{\EP} 
 _{\{\infty\}} \overset{6}{\EP} _{(0,\n_2)} \overset{1}{\EP} \overset{1}{\EM} _{(n_1,+\infty)} \dots
 \] 
 in the structures (\ref{e:NS5}) and (\ref{e:NS6}). When a normal $x$-extremal passes through 
 $\infty = \p_0 = \wt b_0 $ (we use the notation of Remark \ref{r:wtbj}), the values of $x(\wt b_{-1})$ and $x(\wt b_1)$ at the preceding 
 and the succeeding switch points do not necessarily belong to $\Sec_3$ and $\Sec_6$, resp.;
these values may also be in $\Sec_4$ and $\Sec_1$, respectively. The reason is that $y(\p_0)=0$, and so 
$\Arg_\star y (\p_0)$ does not exist. So the bounds on the positions of switch points $\wt b_{\pm1}$ are determined not by $\Arg_\star y (\p_0)$,
as in the case $x(\p_0) = 0$, but by the limits $\Arg_\star y (\p_0 \pm 0) $ (see (\ref{e:A3}) and 
Lemma \ref{l:bj+1}(iii)). 
\end{rem}

\section{Symmetric case: corollaries and an example}
\label{s:Corol}

 Each optimizer $\vep(\cdot)$ of one of symmetric problems (\ref{e:argminFsym}), (\ref{e:argminOddEven}) is an $\vep$-component of an extremal 
tuple $(x, \vep, \la_0, y)$  on an interval 
$[s_-,s_+]$ satisfying the following properties:
\begin{gather} \label{e:vepSym}
 \text{$\vep (\cdot) \in \F^\sym$, $s_+=-s_- = \pm s_\pm^{\vep}>0$,  and $x (s_\pm) = \pm \n_\infty$.}
\end{gather} 
Throughout this section, we assume that (\ref{e:vepSym}) is valid for an extremal on $[s_-,s_+]$ tuple $(x, \vep, \la_0, y)$  and summarize some of consequences of the above analysis for such extremals and so for optimizers of (\ref{e:argminFsym}) and (\ref{e:argminOddEven}).

First, note that $\p_0 = s^\cent = 0$ is not a switch point of $\vep$. Moreover,
either $x(0) = 0$, or $x(0) = \infty$. 
Let $\{\wt b_{j}\}_{j\in \ZZ}^{j \neq 0} = \{ b_j\}_{j \in \ZZ}$ 
be the set of switch point of the extended extremal tuple associated with 
$(x, \vep, \la_0, y)$
(we assume that the enumeration of Remark \ref{r:wtbj} is used for $\{\wt b_{j}\}_{j\in \ZZ}^{j \neq 0}$).

The control $\vep (\cdot)$ on the interval $(s_-,s_+)$ defines a layered resonator, 
which is symmetric w.r.t. $0$. 
Three special types of layers can be naturally considered in this resonator (cf. \cite{OW13}):
the central layer and two edge layers (for a formalization of the notion of a layer, see Section \ref{s:bang-bang}).  
They can be defined in the following way.

By definition, the \emph{central layer} $I^\cent$ of the resonator $\vep (\cdot)$ 
(on $(s_-,s_+)$) equals $(\wt b_{-1}, \wt b_1)$ in the case $\wt b_{\pm1} \in [s_-,s_+]$, and equals
$(s_-,s_+)$ in the case $\wt b_1\ge s_+$.
Note that $\p_0 = 0 \in I^\cent$.
The \emph{right (left) edge layer} $I^+$ (resp., $I^-$) of the symmetric resonator $\vep(\cdot)$ 
is a largest subinterval $I$ of $(s_-,s_+)$ that satisfies the following properties:
$I \ne I^\cent$, $I$ has the form $(s,s_+)$ (resp., the form $(s_-,s)$),
and $I$ does not contain switch points (see Fig. \ref{f:x} (b)). If a layers $I$ of $\vep (\cdot)$ in $(s_-,s_+)$ is not an edge layer and is not a central layer, we say that $I$ is an \emph{ordinary layer}.

By $|I|$ we denote the length of an interval $I$. 
The constant value of $\vep (\cdot)$ in $I^\cent$ is denoted by 
$\ep_\cent$, the constant value of $\vep (\cdot)$ in the layers $I^-$ and $I^+$ coincide and is denoted by $\ep_\edge$. Note that the edge layers $I^\pm$ do not exist exactly when $I^\cent = (s_-,s_+)$.

\begin{cor}\label{c:Icent} 
Assume (\ref{e:vepSym}), $\ka \in \CC_4$, and $I^\cent = (s_-,s_+)$. Then:
\item[(i)] There exists a complex constant $c \neq 0$ such that 
$(x, \vep, 0, cy)$ is an abnormal extremal tuple satisfying  (\ref{e:vepSym}).
\item[(ii)] $|I^\cent|/2 = s_+ = -s_- =  \frac{\pi}{2 \ep_\cent^{1/2} \re \ka}$ (i.e.,
$|I^\cent|$ is equal to half-wave).
\item[(iii)] If $x(0) = \infty$ ($x(0) = 0$), then $\ep_\cent = \n_2^2 \neq \n_\infty^2$ (resp., $\ep_\cent = \n_1^2 \neq \n_\infty^2$).
\end{cor}
\begin{proof}
Since there are no switch points in $(s_-,s_+)$, one can  
easily replace the $(-\n_\infty;\n_\infty)$-eigenfunction $y$ 
by another eigenfunction $y_1 = c y$ such that $y_1^2 (s_\pm) \in \RR$ and $y_1^2 (s) \not \in \RR$ for all $s \in (s_-,s_+)$.
Then $(x, \vep, 0, y_1)$ is an abnormal extremal tuple on $[s_-,s_+]$ and (ii) follows from 
Theorem \ref{t:AbnExtWidth}. Statement (iii) follows from statement (E1) of Section \ref{s:Syn2}. Note that 
$\ep_\cent \neq \n_\infty^2$ because $\pm \n_\infty$ are equilibrium solutions to the equation 
$x'=f(x,\n_\infty^2)$.
\end{proof}

\begin{rem}
The case $x(0) = \infty$ of Corollary \ref{c:Icent} takes place for the example of a minimal time 
extremal that is constructed below in Theorem \ref{t:ExOpt}.
\end{rem}

\begin{rem} \label{r:1/4stack}
Consider the case when the extremal tuple (\ref{e:vepSym}) is abnormal (i.e., $\la_0 = 0$) and 
$I^\cent \neq (s_-,s_+)$. Then the lengths of all layers are described by Theorem \ref{t:AbnExtWidth}. That is,
$|I^\cent| = \frac{\pi}{\ep_\cent^{1/2} \re \ka}$, $|I^\pm| = \frac{\pi}{2 \ep_\edge^{1/2} \re \ka}$,
and for the each ordinary layer $I=(b_j,b_{j+1}) $ we have 
$|I|=\frac{\pi}{2 \vep^{1/2} (b_j+0) \re \ka}$. The ordinary layers form two quarter-wave stacks, one on the left side 
of $I^\cent$, and the other on the right side. Since the edge layers in the abnormal case are also of quarter-wave width, they can also be included into these two quarter-wave stacks.
Under the additional condition $\n_\infty \in [\n_1,\n_2]$, statement (E1) of Section \ref{s:Syn2} implies 
the following interplay between $x(0)$ and 
the values $\ep_\cent$, $\ep_\edge$: in the case $x(0) = \infty$ (in the case $x(0) = 0$), 
one has $\n_2^2 = \ep_\cent=\ep_\edge \neq \n_\infty$ 
(resp., $\n_1^2=\ep_\cent=\ep_\edge \neq \n_\infty $).
\end{rem}

\begin{cor} \label{c:n1n2Norm}
Assume (\ref{e:vepSym}) and $\ka \in \CC_4$. 
\item[(i)] If $x(0) = \infty$ and $\n_\infty = \n_2$ (if $x(0) = 0$ and $\n_\infty = \n_1$), 
then $\la_0 >0$.
\item[(ii)] If $\la_0 >0$ and $x(0)=0$, then $\ep_\cent = \n_2^2$. 
\end{cor}
\begin{proof}
(i) follows from statement (E1) of Section \ref{s:Syn2}. (ii) follows from Theorem \ref{t:PMP}.
\end{proof}

Consider now the length of the layers in the normal case $\la_0 >0$.

\begin{cor} \label{c:Norm}
Assume  (\ref{e:vepSym}), $\ka \in \CC_4$, $\la_0 >0$, and $I^\cent \neq (s_-,s_+)$. 
\item[(i)] Let $I \subset (s_-,s_+)$ be an ordinary layer. Then 
$I = (b_j,b_{j+1})$ for a certain $j \in \ZZ$ and 
$|I|< \frac{\pi}{2 \n_1 \re \ka}$ ($\frac{\pi}{2 \n_2 \re \ka}<|I|< \frac{\pi}{\n_2 \re \ka}$) 
whenever $\vep (b_j+0) = \n_1^2$ (resp.,  whenever $\vep (b_j+0) = \n_2^2$).
\item[(ii)] If $\ep_\edge =\n_1^2$ (if $\ep_\edge =\n_2^2$), then $|I^\pm| < \frac{\pi}{2 \n_1 \re \ka}$ (resp., $|I^\pm| < \frac{\pi}{\n_1 \re \ka}$). 
In particular, this is the case when $\n_\infty = \n_2$ (resp., $\n_\infty = \n_1$).
\item[(iii)] If $x(0) = 0$, then $|I^\cent| < \frac{\pi}{\ep_\cent^{1/2} \re \ka}$. 
\item[(iv)] Let $x(0) = \infty$ and $\n_\infty \in [\n_1,\n_2]$. Assume that $\ep_\cent = \n_2^2$ (that $\ep_\cent= \n_1^2$). 
Then $\frac{\pi}{\n_2 \re \ka} < |I^\cent| < \frac{2\pi}{\n_2 \re \ka}$ (resp., $|I^\cent| < \frac{\pi}{\n_1 \re \ka}$).
\end{cor}
\begin{proof}
\emph{(i)} follows from the definition of a layer, the definitions of central and edge layers, and from 
Theorems \ref{t:NextWidth} and \ref{t:SecSwitch}. 

\emph{(ii)} follows from Theorems \ref{t:NextWidth} and \ref{t:SecSwitch}. Another proof can be obtained from (i).
Indeed, one can continue the normal extremal tuple $(x, \vep, \la_0, y)$ to a wider symmetric 
interval $[- \wt s_+, \wt s_+]$ such that $x(\wt s_+) \in \RR_+$ and denote $\wt \n_\infty = x(\wt s_+) $.
Then statement (i) is applicable to the extensions of the original edge layers, which become non-edge layers
for the extended extremal tuple.

\emph{(iii)} follows from the assumption that the extremal tuple is normal and 
the process of computation of the position of the switch point $\wt b_1$ from the position 
of the preceding switch point $\wt b_{-1}$ (see Lemma \ref{l:bj+1}).

\emph{Consider the case $\ep_\cent= \n_1^2$ of statement (iv).} It follows from (E2) that $\vep (\cdot)$
corresponds either the sequence (\ref{e:NS1}), or the sequence (\ref{e:NS3}).
This implies that $x(s) \in \Sec_3 \subset \CC_+$ for all $s$ in the left half 
$(\wt b_{-1},0)$ of the layer $I^\cent$.
Considering the corresponding part of the trajectory of $\vth (s)$ 
(see (\ref{e:vth=ep}) for the definition) in the way similar to the proof of Lemma \ref{l:bj+1},
one obtains $|\wt b_{-1}|< \frac{\pi}{2 \n_1 \re \ka}$. This implies $|I^\cent| < \frac{\pi}{\n_1 \re \ka}$.

\emph{Let us consider the case $\ep_\cent= \n_2^2$ of statement (iv).} It follows from (E2) that 
$x(\wt b_{-1}) \in \Sec_4 \subset \CC_-$. Thus, the trajectory of $x$ passes through $\RR_-$ to $\Sec_3 \subset \CC_+$ and going through $\Sec_3$
reaches $\infty$ at $s=0$. Considering the rotation of $\vth (s)$ for  $s \in (\wt b_{-1},0)$, we see 
that $ \frac{\pi}{2 \n_2 \re \ka}<|\wt b_{-1}|<\frac{\pi}{\n_2 \re \ka}$, and so $\frac{\pi}{\n_2 \re \ka}<|I^\cent| <\frac{2\pi}{\n_2 \re \ka}$.
\end{proof}

Let us give an explicit example of minimum-time control from $\infty$ to $\n_\infty \in [\n_1^2,\n_2^2)$.

\begin{thm} \label{t:ExOpt}
Let $\n_\infty  \in [\n_1,\n_2)$, $s_+ = - s_->0$, and 
\begin{equation} \label{e:k0}
\ka_0 = -\frac{\ii}{2s_+ \n_2} 
\Ln  \frac{\n_2 + \n_\infty}{\n_2 - \n_\infty}  + \frac{\pi}{(s_+-s_-) \n_2}.
\end{equation}
Then 
$\vep_0 (s)= 
\left\{ \begin{array}{rr} 
\n_2^2, & s \in [s_-,s_+] \\
\n_\infty^2,& s \in \RR \setminus [s_-,s_+]
\end{array} \right.
 $
is a unique minimizer of 
$\argmin_{\substack{\vep \in \F^\sym \quad \\\!\!\! \ka_0 \in \Si^{\odd } (\vep)}} \Lr (\vep)$, and 
the corresponding minimal value of $\Lr (\vep)$ is $\Lr_{\min}^{\odd} (\ka_0) = 2 s_+$.
\end{thm}

\subsection{Proof of Theorem \ref{t:ExOpt}}

It follows from (\ref{e:Si(ep)}) that $\ka_0 \in \Si^{s_-,0}_{-\n_\infty,\infty} (\vep_0)$.
By (\ref{e:xRepr}) and (\ref{e:vthep}), $\vth (\cdot)$ defined by (\ref{e:vth=ep}) for the control $\vep = \vep_0$
equals $\vth  (s)
= - \left( \frac{\n_2 + \n_\infty}{\n_2 - \n_\infty} \right)^{-s/s_+} \ee^{-\ii \pi s/s_+}$
for $|s| \le s_+$. 
This implies that $\vth (s) \in \CC_- $ and $x(s) \in \CC_+$ for $s \in [s_-,0)$. 
By Lemma \ref{l:bj+1} (iii), there exists an abnormal extremal tuple $(x_0, \vep_0, 0, y_0)$ on $[s_-,s_+]$
so that $x_0 (s_\pm) = \pm \n_\infty$. It is of the type (\ref{e:onelayerodd}).

Assume that an extended extremal control $\vep_1$ is such that 
$\vep_1 (\cdot) \not \equiv \n_2^2$ on a certain interval  $(s_-,s_-+t)$ and that $\vep_1$ 
steers $x(s_-) = - \n_\infty$ to $x(s_- +t) = \infty$.
Then $x(s_-+2t) = \n_\infty$, the corresponding extended extremal tuple 
$(x_1, \vep_1, \la_1, y_1)$ is symmetric w.r.t. $s_-+t$, and is of one of the types 
(\ref{e:AS3}), (\ref{e:NS1}), 
(\ref{e:NS3}), (\ref{e:NS5}), (\ref{e:NS6}) on $[s_-,s_-+2t]$ (with $m_1=m_2$ for 
(\ref{e:NS1}), 
(\ref{e:NS3}), (\ref{e:NS5}), (\ref{e:NS6})). 

We assume that $\vep_1$ is a minimum time control that steers $x_1 (s_-)=(- \n_\infty)$ to $ \infty$
in the minimum possible time $t$ (in particular, $t \in (0,s_+]$), and, considering each of the cases described above, 
show that this leads to a contradiction.
The case (\ref{e:NS1}) is special 
and require more involved arguments.

\emph{The case (\ref{e:AS3})} assumes $m \ge 1$ repetitions and is abnormal.
Due to Theorem \ref{t:AbnExtWidth}, its part 
$\scriptstyle _{(-\n_2,-\n_1]} \overset{3}{\EP} _{(-\infty,-\n_2)} $ requires the time 
$\frac{\pi}{2 \n_2 \re \ka_0} = s_+$ (the last equality follows from (\ref{e:k0})).
Since we have assumed that $\vep_1$ is minimum-time from $(- \n_\infty)$ to $ \infty$, we have $\vep_1 = \vep_0 $ on $(s_-,0)$, and so $x_1 (0) = x_1 (s_- + s_+) = \infty$. This contradicts to 
the fact that, according to (\ref{e:AS3}), $x_1 (s_- + s_+) \in (-\infty,-\n_2) $.

\emph{The case (\ref{e:NS3})} cannot correspond to the minimal time control from $(-\n_\infty)$ to $\infty$
 since Corollary \ref{c:Norm} (i) implies that the part $\scriptstyle \overset{4}{\EP} _{(-\n_2,0)} \overset{3}{\EP} $
 requires time greater than $s_+$ (note that this part corresponds to one ordinary layer), and so 
 $t>s_+$.  Almost the same arguments are applicable to \emph{the case (\ref{e:NS6})},  \emph{the sub-case $m_1=m_2 \ge 1$ of (\ref{e:NS5})}, and \emph{the sub-case $m_1 = m_2 \ge 1$ of (\ref{e:NS1})}.
 
Consider  \emph{the sub-case $m_1=m_2 = 0$ of (\ref{e:NS5})}. 
We have $\vep_0 = \vep_1 \equiv \n_2^2 $ and 
 $x_0 (\cdot) = x_1 (\cdot)$ on  $[s_-,s_+]$. Thus, this case leads to the same minimizer $\vep_0 (\cdot)$  (however, the extremal tuple is normal).
 

\emph{In the rest of this subsection we 
assume that the case (\ref{e:NS1}) take place and that $m_1 = m_2 = 0$}.
Then there exists maximal $s_0 \in (s_-, s_-+t)$ such that, for $s \in (s_-,s_0)$, we have
$\vep_0 (s) = \vep_1 (s) = \n_2^2$. So $x_0 (s_0) = x_1 (s_0) \in \Sec_3$,
$\vep_1 (s) = \n_1^2$ for $s \in (s_0,s_-+t)$,
$x_0 (s) \neq x_1 (s)$ for a.a. $s \in (s_0, 0)$.

Consider an extended abnormal extremal tuple $(x_2,\vep_2,0,y_2)$ such that for 
$s \in (s_0,s_-+t)$, we have $\vep_2 (s) = \vep_1 (s) =\n_1^2$ and $x_2 (s) = x_1 (s)$. Then 
$\vep_2 (s) = \n_1^2$ for $s \in (t^-_2,s_0)$, where $t^-_2 =s_-+t - \frac{\pi}{2 \n_1 \re \ka}$.
On the interval $(t^-_2,t^+_2)$, where $t^+_2= s_-+t + (s_-+t - t^-_2)$, 
the sequence corresponding to 
$(x_2,\vep_2,0,y_3)$ is
$\displaystyle _{(-\n_1,0)} \overset{3}{\EM}  _{\{\infty\}} \overset{6}{\EM}  _{(0,\n_1)}$.
In particular, $x_2 (t^-_2) \in (-\n_1,0)$. 

Consider a piece-wise smooth Jordan curve $\Jmf$ in $\wh \CC$ consisting of two pieces:
$\Jmf_1 : [s_-,0] \to \{-\n_\infty\}\cup \Sec_3 \cup \{\infty\} $, $\Jmf_1 (s) = x_0 (s)$,
and $\Jmf_\RR : [0,n_\infty^{-1}] \to \{\infty\} \cup (-\infty, -\n_\infty]$, $\Jmf_\RR (s)= -1/s$.
Its complement in $\wh \CC$ consists of two components $\CJ_0$ and $\CJ_\star$, where 
$\CJ_0$ is singled out by $0 \in \CJ_0 $. It is not difficult to see that
\begin{multline} \label{e:intx0x2}
\text{the trajectory $x_2 [[t^-_2,s_-+t]] $ intersects $\Jmf_1$ passing through the point}
\\
\text{$X_0 := x_2 (s_0) = x_1 (s_0) \in \Sec_3$  from  $\CJ_0$ to $\CJ_\star$.} 
\end{multline}


Indeed, the tangent vector $f(X_0,\n_2^2)$
to $\Jmf_1$ at $X_0$ 
differs from the tangent vector $f(X_0,\n_1^2)$ to the trajectory of $x_2$ at $X_0$
in the following way $f(X_0,\n_1^2)- f (X_0,\n_2^2)  = - \ii \ka (\n_2^2 -\n_1^2) $.
The angle between this difference and the the tangent vector to $\Jmf_1$ at $X_0$ is
$
\Arg_0 \frac{-\ii \ka (\n_2^2 -\n_1^2)}{f (X_0,\n_2^2)} = 
- \Arg_0 (-\n_2^2 + X_0^2) \in (0,\pi) $
since $X_0 \in \Sec_3 \subset \CC_2$.
This implies (\ref{e:intx0x2}).

Let us introduce one more abnormal extended extremal tuple 
$(x_3,\vep_3,0,y_3)$, which can be considered as a perturbation of 
$(x_2,\vep_2,0,y_2)$ because we put $x_3 (t^-_2)=x_2 (t^-_2) + \de_0$ 
with small enough $\de_0>0$ and assume that $\vep_3 (t^-_2+0) = \n_1^2 = \vep_2 (s)$ 
for $s \in (t^-_2,s_-+t)$.
Choosing small enough $\de_0$, we can ensure that:
(i) $x_2 (t^-_2) + \de_0 \in (x_2 (t^-_2),0)$, 
(ii) $x_3 [[t^-_2,s_-+t]] $ intersects $\Jmf_1$ at a certain point $X_1 \in \Sec_3$  passing from $\CJ_0$ to $\CJ_\star$.

Consider the point $s_3$ such that $x_3 (s_3) \in \ii \wh \RR$. 
Since $(x_3,\vep_3,0,y_3)$ is abnormal,
we have $\vep_3 (s) = \n_1^2$ for $s\in (t^-_2,s_3)$.
This implies that $x_2 [[t^-_2,s_-+t]] \cap x_3 [[t^-_2,s_3]] = \varnothing$
Hence, $x_3 (s_3) \in \ii \RR_+$. This and the condition (ii) on $\de_0$ implies that 
there exist  
$X_2 \in \Jmf_1 \cap \Sec_3$ such that $x_3 [[t^-_2,s_-+t]] $ intersects $\Jmf_1$ at $X_2 $ passing from $\CJ_\star$ to $\CJ_0$.

Considering at $X_2$ the values of the tangent vectors to $\Jmf_1$ and to $x_3 [[t^-_2,s_3]]$
(which are equal to $f (X_2,\n_2^2)$ and $ f(X_2,\n_1^2)$, resp.), 
one sees that the existence of $X_2$ leads to a contradiction.


\section{Back to the problem of decay rate minimization}
\label{s:DecRate}

The aim of this section is the reduction of 
the original problem \cite{K13,KLV17} of Pareto optimization of the decay rate 
to the dual problem of length minimization, and so, to the minimum-time control
in the case $\n_\infty \in (\n_1,\n_2)$, and can `partially reduced' in the case 
$\n_\infty \in \{\n_1,\n_2\}$. 
The results of Subsection \ref{ss:star} on maximal resonance free regions over $\F_{s_-,s_+}$ and on the star-convexity of this region 
serve as technical tools.


Let $s_\pm$ and $\eta_\pm$ be fixed such that $-\infty<s_- < s_+ <+\infty$ and $\eta_- \neq \eta_+$.
Recall that the set $\Si_{\eta_-,\eta_+}^{s_-,s_+} [\F_{s_-}]$ 
of achievable $(\eta_-,\eta_+)$-eigenvalues and the associated minimal modulus function $\rho_{\min} (\cdot) $ 
were introduced in Section \ref{s:DualRef} and that 
$\rho_{\min} (\ga) = +\infty $ if the complex argument $\ga$ is not achievable 
(i.e., $\dom \rho_{\min} = \Arg_0 \Si_{\eta_-,\eta_+}^{s_-,s_+} [\F_{s_-}] $).

Assume  that 
\begin{align} \label{e:eta-eta+as}
\text{$\eta_- \in \RR_- \cup \{0,\infty\}$ and $\eta_+ \in \RR_+ $}.
\end{align}
Then an  $(\eta_-,\eta_+)$-eigenvalue $\ka$ of $\vep (\cdot) \in \F_{s_-,s_+}$ 
lies in $\CC_-$ and can be interpreted  as a resonance of the operator 
$\frac{1}{\vep} \pa_x^2$, where $\vep (\cdot)$ is extended to the whole $\RR$ in a suitable way according to the values of 
$\eta_\pm$ \cite{KLV17} (the set of corresponding resonances may also contain points that are not 
$(\eta_-,\eta_+)$-eigenvalues, see Section \ref{ss:ResOpt}, \ref{ss:Rew} and \cite{KLV17} for details);  $\re \ka$ and $(-\im \ka)$ have the `physical meaning' of \emph{frequency}
and \emph{decay rate}, respectively.

The set of achievable $(\eta_-,\eta_+)$-eigenvalues $\Si_{\eta_-,\eta_+}^{s_-,s_+} [\F_{s_-}]$
is closed (see \cite{KLV17} and Proposition \ref{p:SiClosed} (i) below).
Similarly to Section \ref{ss:Rew}, we define 
the minimal decay rate $\beta_{\min} (\alpha)$ for $\alpha \in \RR$
by 
\begin{align} \label{e:betamin}
 \beta_{\min} (\alpha) := \ \inf \{ \beta \in \RR \ : \ \alpha - \ii \beta  \in \Si_{\eta_-,\eta_+}^{s_-,s_+} [\F_{s_-}] \} . 
 \end{align}
Since $\Si_{\eta_-,\eta_+}^{s_-,s_+} [\F_{s_-}]$ is closed and $\Si_{\eta_-,\eta_+}^{s_-,s_+} [\F_{s_-}] \subset \CC_-$, we see that
$\beta_{\min} (\al) : \RR \to (0,+\infty]$.

A frequency 
$\al \in \RR$ is  \emph{achievable} if 
$\al \in \re \Si_{\eta_-,\eta_+}^{s_-,s_+} [\F_{s_-}] \ ( = \dom \beta_{\min})$.
If $\al$ is achievable, then the closedness of $\Si_{\eta_-,\eta_+}^{s_-,s_+} [\F_{s_-}]$ implies that $\ka = \alpha - \ii \beta_{\min} (\alpha)$ is an $(\eta_-,\eta_+)$-eigenvalue 
for a certain $\vep (\cdot) \in \F_{s_-,s_+}$ (i.e., the minimum is achieved in (\ref{e:betamin})), and 
we say that $\ka $ and $\vep(\cdot)$ are of \emph{minimal decay for (the frequency)} $\alpha $.

The set $\Pa_{\Dr}^{\eta_-,\eta_+} := \{ \al - \ii \beta_{\min} (\al) \ : \ \al \in \dom \beta_{\min} \}$ is  
\emph{the Pareto frontier} for the problem of minimization of the decay rate $(-\im \ka)$ of an 
$(\eta_-,\eta_+)$-eigenvalue $\ka$ over $\F_{s_-,s_+}$. 
Comparing with the settings of Section \ref{ss:Rew}, one sees that 
$\Pa_{\Dr}^{-\n_\infty,\n_\infty} = \Pa_{\Dr}$. If, additionally, 
$ s_- = 0$, $s_+ = \ell $, then $\Pa_{\Dr}^{0,\n_\infty} = \Pa_{\Dr}^\even$ and 
$\Pa_{\Dr}^{\infty,\n_\infty} = \Pa_{\Dr}^\odd$.

\begin{thm} \label{t:MinDec1} 
Assume (\ref{e:eta-eta+as}) and $\eta_+ \in (\n_1,\n_2)$.
Let $\ga_0= \Arg_0 \ka \in (-\pi/2,0)$.
Then:
\item[(i)] The function $\rho_{\min} (\cdot) $ is continuous and $\RR$-valued on $(-\pi/2,0)$.
\item[(ii)] Assume additionally that $\ka$ is the $(\eta_-,\eta_+)$-eigenvalue of minimal decay for the frequency $\re \ka$. Then
\begin{align} \label{e:ka01}
 \text{$\ka = \rho_{\min} (\ga_0) e^{\ii \ga_0}$, \quad $T^{\min}_\ka (\eta_-,\eta_+) = s_+ - s_-$,}
\end{align} 
(i.e., $\ka$ is an $(\eta_-,\eta_+)$-eigenvalue of minimal modulus 
for $\ga_0$), and
\begin{multline}
\text{the family of minimum time controls for (\ref{e:R}), (\ref{e:R2}) steering $x(s_-) = \eta_-$ to $\eta_+$}
\\  
\text{coincides with the family of resonators of minimal decay for $\re \ka$.} 
\label{e:ka02}
\end{multline}
\end{thm}

The proof is postponed to Section \ref{ss:prMinDec}.

The next three theorems explain the rigorous meaning of the statement that, in the case $\n_\infty \in \{\n_1,\n_2\}$, 
the decay rate minimization can be `partially reduced' to the problem of Pareto optimization of $|\ka|$,
and so to minimum-time control.

\begin{thm} \label{t:PF}
Assume (\ref{e:eta-eta+as}) and additionally $\eta_+ \in [\n_1,\n_2]$. Then:
\item[(i)] $(-\pi,0) \setminus \{-\pi/2\} \subset \dom \rho_{\min} \subset (-\pi,0)$ and 
\begin{gather} \label{e:star}
\Si_{\eta_-,\eta_+}^{s_-,s_+} [\F_{s_- }]  = \{ c e^{\ii \ga} \rho_{\min} (\ga)  \ : \ 
c \in [1,+\infty) \text{ and $\ga$ is achievable }\} 
\end{gather}
\item[(ii)] The Pareto frontier
$\Pa_{\Dr}^{\eta_-,\eta_+} $ of minimal decay 
can be found from  the Pareto frontier 
$\Pa_{\md} $ of minimal modulus 
via the formulae (\ref{e:star}) and (\ref{e:betamin}).
\end{thm}

The proof is postponed to Section \ref{ss:prMinDec}.

\begin{thm} \label{t:MinDec2}
Suppose (\ref{e:eta-eta+as}). Assume that either $\eta_+ = \n_1$, or $\eta_+=\n_2$. 
Assume that  
\[
\text{$\ga_0= \Arg_0 \ka_0 \in (-\pi/2,0)$, \qquad  $\rho_0 := \rho_{\min} (\ga_0)$, \qquad 
$\rho_1 := \inf_{\ga \in (\ga_0,0)} \frac{\rho_{\min} (\ga) \cos \ga}{\cos \ga_0}$,}
\]
and that $\ka_0$ is the $(\eta_-,\eta_+)$-eigenvalue of minimal decay for the frequency $\re \ka_0$.
Then $\rho_0 \le \rho_1$ and $\ka_0 \in \ee^{\ii \ga_0} [\rho_0 ,\rho_1] $ (in particular, 
in the case $\rho_0 =\rho_1$, we have $\ka_0 =  \ee^{\ii \ga_0} \rho_0$).

Moreover, each of the numbers $\ka = \rho \ee^{\ii \ga_0}$ with $\rho \in [\rho_0, |\ka_0|] \cup (\rho_0 ,\rho_1) $ 
is the $(\eta_-,\eta_+)$-eigenvalue of minimal decay for the frequency $\rho \cos \ga_0$
and one of the two following cases 
take place for such $\ka$:
\item[(i)] In the case $|\ka| =\rho_0$, statements 
(\ref{e:ka01}) and (\ref{e:ka02}) hold.
\item[(ii)] In the case $|\ka| \in (\rho_0, |\ka_0|] \cup (\rho_0 ,\rho_1)$, 
we have $T^{\min}_{\ka} (\eta_-,\eta_+) < s_+ - s_-$ and, for each 
minimum time control $\vep_{\ga_0}^{\min}$ that steer $x(s_-) = \eta_-$ to $x (s_-+T^{\min}_{\ka} (\eta_-,\eta_+)) = \eta_+$,
the function $\wt \vep \in \F_{s_-,s_+}$ defined by 
\begin{align} \label{e:ResMinDecCont}
\wt \vep (s) := \left\{ \begin{array}{cc}
\vep_{\ga_0}^{\min} (s) & \text{ if } s \in [s_-,s_- + T^{\min}_\ka (\eta_-,\eta_+) ), \\
\eta_+^2 & \text{ if } s \in (s_- + T^{\min}_\ka (\eta_-,\eta_+) , s_+] ,
\end{array} \right.
\end{align}
is a resonator of minimal decay for frequency $\re \ka$ and, simultaneously, is an 
abnormal  $\vep$-extremal on $[s_-,s_+]$ such that the corresponding abnormal x-extremal (with the initial point 
$x(s_-) = \eta_-$) has $[s_- + T^{\min}_\ka (\eta_-,\eta_+) , s_+]$ as its stationary interval.
\end{thm}

The proof is given in Section \ref{ss:prMinDec}.

The case (ii) of Theorem \ref{t:MinDec2} and statements (iii)-(v) of Corollary \ref{c:rho(ga)} in the next subsection 
consider the 
situation when the minimal-modulus function $\rho_{\min} (\cdot)$ is discontinuous at a certain $\ga \in (-\pi/2,0)$.
The next result shows that this situation takes place for $\ga_0 = \Arg_0 \ka_0$ with $\ka_0$ from Theorem \ref{t:ExOpt}.

\begin{thm} \label{t:ExOpt2}
Let $\eta_- = \infty$, $\eta_+ = \n_1$, $s_- = 0$, $s_+>0$, $\ka_0 = -\frac{\ii}{2 s_+ \n_2} 
\Ln  \frac{\n_2 + \n_1}{\n_2 - \n_1}  + \frac{\pi}{2 s_+ \n_2}$, $\ga_0 = \Arg_0 \ka_0$,
$\al_0=\frac{\pi}{2 s_+ \n_2}$, $\rho_0 = \rho_{\min} (\ga_0)$, 
and $\rho_1 := \inf_{\ga \in (\ga_0,0)} \frac{\rho_{\min} (\ga) \cos \ga}{\cos \ga_0}$.
Then
\begin{align} \label{e:=rho=k<}
\lim_{\ga \to \ga_0 - 0} \rho_{\min} (\ga) = \rho_0 = |\ka_0|< 
 \rho_1, \qquad \rho_0 < \lim_{\ga \to \ga_0 + 0} \rho_{\min} (\ga) ,
\end{align}
and the following statements hold:
\item[(i)] The number $\ka_0$ is simultaneously $(\eta_-,\eta_+)$-eigenvalue of minimal decay 
for the frequency $\al_0$ and $(\eta_-,\eta_+)$-eigenvalue of minimal modulus for the complex argument $\ga_0$.
Moreover, $\vep^{\min}_{\ga_0} (\cdot)$ defined by $\vep^{\min}_{\ga_0} (s) = \n_2^2$ for all 
$s \in (0,s_+)$ is the unique function in $\F_{0,s_+}$ that generate an $(\eta_-,\eta_+)$-eigenvalue
at $\ka_0$.
\item[(ii)]  Each of the numbers $\ka = \rho \ee^{\ii \ga_0}$ with $\rho \in [\rho_0, \rho_1) $ 
is the $(\eta_-,\eta_+)$-eigenvalue of minimal decay for the frequency $\rho \cos \ga_0$. One of 
associated resonators $\vep (\cdot) \in \F_{0,s_+}$ of minimal decay can be constructed 
by the rule (\ref{e:ResMinDecCont}).
\end{thm}
The proof is given in Section \ref{ss:prMinDec}.

\subsection{Maximal star-like resonance free regions}
\label{ss:star}

Recall that a set $\Om \subset \CC$ containing $0$ is called star-shaped w.r.t. $0$ if 
$z \in \Om$ implies $[z,0] \subset \Om$. The set $\CC \setminus  \Si_{s_-,s_+}^{-\n_\infty,\n_\infty} [\F_{s_-}]  $ can be perceived as \emph{the resonance free region over} 
$\{ \vep \in \F \ : \ s_-^{\vep} = s_-  , \  s_+^{\vep} = s_+  \}.$
(We would like to notice that the estimation of resonance free strips for Schrödinger equations 
was originally \cite{HS86} one of the main motivation for resonance optimization.)
Then 
\begin{equation} \label{e:qeigfree}
\{0\} \cup \{ \ka \in \CC \ : \ \ka \neq 0  \text{ and } 
|\ka | < \rho_{\min} (\Arg_0 (\ka), -\n_\infty,\n_\infty) \} 
\end{equation}
is the maximal star-shaped (w.r.t. $0$) part of the resonance free region.
If, additionally, $\n_1 \le \n_\infty \le \n_2$, 
the star-shaped set (\ref{e:qeigfree}) exactly equals the resonance free region  
as it is shown by the following statement.

\begin{prop} \label{p:star-like}
Let $\eta_- \neq \eta_+$ and $\eta_+ \in [\n_1, \n_2]$ (or $\eta_- \in [-\n_2,-\n_1]$).  Then
$(-\pi,0) \cup (0,\pi) \subset \Arg_0 \Si_{\eta_-,\eta_+}^{s_-,s_+} [\F_{s_-}] $ 
and formula (\ref{e:star}) holds true.
\end{prop}

\begin{proof}
Let $\ga $ be achievable. 
 Then (\ref{e:ep^min_ga}) and  
 $\n_1 \le \eta_+ \le \n_2$ imply that an extension $\vep (\cdot) $ of $\vep^{\min}_{\ga} (\cdot)$ 
to $ (s_+,+\infty)$ by the constant value $\eta_+^2$ 
 belongs to $\F_{s_-,s_0}$  for any $s_0 \in [s_+,+\infty)$. For every 
such $\vep (\cdot) $, we have $x(s) = x(s_+) = \eta_+$ for all $s > s_+$.
 This and the scaling (\ref{e:scalingSi}) imply that  
 $ c \ee^{\ii \ga} \rho_{\min} (\ga) $ is an achievable  $(\eta_-,\eta_+)$-eigenvalue
 for $c \ge 1$. 
The global controllability (Theorem \ref{t:GlContr}) implies the first statement.
\end{proof}

\begin{prop} \label{p:SiClosed}
Let $\eta_- \neq \eta_+$. Then 
\item[(i)] The set $\Si^{\eta_-,\eta_+}_{s_-,s_+} [\F]$ 
is closed in $\CC$.
\item[(ii)] $\rho_{\min} (\cdot) $ is lower semicontinuous at each $\ga \in \RR$. 
\end{prop}
\begin{proof}
(i) follows from local weak*-continuity of the set-valued map 
$\vep (\cdot) \mapsto \Si^{\eta_-,\eta_+}_{s_-,s_+} (\vep)$ \cite{KLV17}, 
or from weak*-compactness arguments \cite{HS86,K13}. 
Statement (ii) is a simple corollary of (i).
\end{proof}

So $\Si^{\eta_-,\eta_+}_{s_-,s_+} [\F]$ contains 
$\Bd \Si^{s_-,s_+}_{\eta_-,\eta_+} [\F_{s_-}]$, $\Pa^{\eta_-,\eta_+}_{\Dr}$ and $\Pa_{\md}$.

\begin{prop} \label{p:Bd->MinDec}
Assume that (\ref{e:eta-eta+as}) holds and  $\eta_+ \in [\n_1,\n_2]$ (or $\eta_- \in [-\n_2,-\n_1]$).
Let $\Arg_0 \ka_0 = \ga_0 \in (-\pi/2,0)$ and $\ka_0 \in \Bd \Si^{s_-,s_+}_{\eta_-,\eta_+} [\F_{s_-}]$.
Then $\ka_0$ is an $(\eta_-,\eta_+)$-eigenvalue of minimal decay 
if and only if $\re \ka_0 < \rho_{\min} (\ga) \cos (\ga)$ for all $\ga \in (\ga_0,0)$.
\end{prop}
\begin{proof}
Since $\ka_0 \in \Si^{s_-,s_+}_{\eta_-,\eta_+} [\F_{s_-}]$, we see that 
$\ka_0 = \re \ka_0 - \ii \beta_{\min} (\re \ka_0)$ if and only if the $\CC$-interval 
$(\ka_0,\re \ka_0)=\{\la \ka_0 +(1-\la)\re \ka_0 : \la \in (0,1) \}$ does not intersect $\Si^{\eta_-,\eta_+}_{s_-,s_+} [\F_{s_-}]$.
Proposition \ref{p:star-like} completes the proof.
\end{proof}

\begin{cor} \label{c:rho(ga)}
Assume (\ref{e:eta-eta+as}) and $\ga_0 \in (-\pi,\pi)$. Then:
\item[(i)] $\lim_{\ga \to 0-0} \rho_{\min} (\ga) = +\infty$ and 
$\rho_{\min} (\ga) = +\infty$ for all $\ga \in [0,\pi]$.
\item[(ii)] If $\liminf_{\ga \to \ga_0-0} \rho_{\min} (\ga) \neq \rho_{\min} (\ga_0) \neq \liminf_{\ga \to \ga_0+0} \rho_{\min} (\ga)$,
then $\ga_0 = - \pi/2$ and \linebreak $\rho_{\min} (-\pi/2) < +\infty$.
\item[(iii)] Let $\ga_0 \in (-\pi/2,0)$ be a point of discontinuity of $\rho_{\min}$. 
Then either  $\lim\limits_{\ga \to \ga_0-0} \rho_{\min} (\ga) = \rho_{\min} (\ga_0)$, or 
$\lim\limits_{\ga \to \ga_0+0} \rho_{\min} (\ga) = \rho_{\min} (\ga_0)$.
\item[(iv)] Assume that $\eta_+ \in [\n_1,\n_2]$ (or $\eta_- \in [-\n_2,-\n_1]$), $\Arg_0 \ka_0 = \ga_0 \in (-\pi/2,0)$,  
$\ka_0 \in \Bd \Si^{s_-,s_+}_{\eta_-,\eta_+} [\F_{s_-}]$ and $\ka_0 \neq \rho_{\min} (\ga_0) e^{\ii \ga_0}$.
Then 
\begin{align} \label{e:lim-<lim+}
\text{either } \quad & \lim_{\ga \to \ga_0-0} \rho_{\min} (\ga) = \rho_{\min} (\ga_0)  < |\ka_0| \le 
\liminf_{\ga \to \ga_0+0}  \rho_{\min} (\ga) , \\
\text{or } \quad & \lim_{\ga \to \ga_0+0} \rho_{\min} (\ga) = \rho_{\min} (\ga_0)  < |\ka_0| \le 
\liminf_{\ga \to \ga_0-0}  \rho_{\min} (\ga) . \label{e:lim+<lim-}
\end{align}
Moreover, $e^{\ii \ga_0} [\rho_{\min} (\ga_0), |\ka_0|] \subset \Bd \Si^{s_-,s_+}_{\eta_-,\eta_+} [\F_{s_-}]$.
\item[(v)] Let the assumptions of statement (iv) hold and, additionally,
 $\ka_0 $ is the $(\eta_-,\eta_+)$-eigen\-value of minimal decay for the frequency $\re \ka_0$.
Then (\ref{e:lim-<lim+}) takes place and $\re \ka_0  < \rho (\ga) \cos \ga$ for all
$\ga \in (\ga_0,0)$.
Moreover, each $\ka \in e^{\ii \ga_0} [\rho_{\min} (\ga_0), |\ka_0|]$
is the $(\eta_-,\eta_+)$-eigenvalue of minimal decay for the frequency $\re \ka$.
\end{cor}

\begin{proof}
\emph{(i)} If $\ga \in [0,\pi]$, the equality $\rho_{\min} (\ga) = +\infty$ follows from
$\Si^{s_-,s_+}_{\eta_- ,\eta_+} [\F_{s_-}] \subset \CC_-$.
Then $\lim_{\ga \to 0-0} \rho_{\min} (\ga) = +\infty$ follows 
from the lower semicontinuity of $\rho_{\min} $ (see Proposition \ref{p:SiClosed}).

\emph{(ii)-(iii)} The lower semicontinuity of $\rho_{\min} $ also proves that $\rho_{\min}$ is continuous 
at each $\ga_0 \not \in \dom \rho_{\min} $ (as a map to the topological space $[0,+\infty]$).
Now, let $\ga_0 \in (-\pi,-\pi/2) \cup (-\pi/2,0)$ be a point of discontinuity of $\rho_{\min}$.
In this case, (ii) follows from (iii). \emph{Let us prove (iii).} Since $\ga_0$ is a point of discontinuity of $\rho_{\min}$, there exists 
the resonance $\ka_0 = \rho_{\min} (\ga_0) e^{\ii \ga_0} \in \CC_4$ of minimum modulus 
for $\ga_0$. It follows from the perturbation arguments 
of \cite[Appendix A]{KLV17} that there exists a non-degenerate triangle $\Tr \subset \Si^{s_-,s_+}_{\eta_-,\eta_+} [\F_{s_-}]$ 
with a vertex at $\ka_0$.
So, at least one of inequalities 
$\limsup_{\ga \to \ga_0 \pm 0} \rho_{\min} (\ga) \le \rho_{\min} (\ga_0)$ holds.
The lower semicontinuity completes the proof of (iii), and, in turn, of (ii). 

\emph{(iv)-(v)} By Proposition \ref{p:star-like}, the resonance free region 
$\CC \setminus \Si^{\eta_-,\eta_+}_{s_-,s_+} [\F_{s_- }]$
is star-shaped w.r.t. $0$ and open. So $\ka_0 \in \Bd \Si^{s_-,s_+}_{\eta_-,\eta_+} [\F_{s_-}]$
implies 
$e^{\ii \ga_0} [\rho_{\min} (\ga_0), |\ka_0|] \subset \Bd \Si^{s_-,s_+}_{\eta_-,\eta_+} [\F_{s_-}]$ 
(indeed, assuming converse we easily get a contradiction to the star property). The combination of these 
arguments with (iii) implies that exactly one of the formulae 
(\ref{e:lim-<lim+}), (\ref{e:lim+<lim-}) holds. Finally, statement (v) easily follows from (iv) and 
Proposition \ref{p:Bd->MinDec}.
\end{proof}

\subsection{Proofs of Theorems \ref{t:MinDec1}-\ref{t:ExOpt2}}
\label{ss:prMinDec}

\begin{proof}[Proof of Theorem \ref{t:MinDec1}]
(i) It follows from Proposition \ref{p:star-like} that $\dom \rho_{\min} \supset (-\pi/2,0)$, 
and so $\rho_{\min}$ is $\RR$-valued in $(-\pi/2,0)$.
Due to Proposition \ref{p:SiClosed}, to prove $\rho_{\min} (\cdot) \in C_\loc (-\pi/2,0)$,
it remains to show that $\rho_{\min} $ is upper semicontinuous at each $\ga_0 \in (-\pi/2,0)$.
This fact follows from the uniform STLC established in Theorem \ref{t:STLC} (and so uses essentially the assumption $\eta_+ \in (\n_1,\n_2)$).  
Indeed, for $\ga \in [\ga_0 -\de,\ga_0+\de] \subset (-\pi/2,0)$ 
(with small enough $\de$)
we take $\ka \in \CC_4$ with $\Arg_0 \ka = \ga_0$ and put $\wt \ka (\ga) := e^{\ii \ga} |\ka|$.
Then (\ref{e:Tmin=rho}) takes the form
$
T^{\min}_{\wt \ka (\ga)} (\eta_-,\eta_+) = (s_+-s_-) \rho_{\min} (\ga) / |\ka|
$
and it is enough to prove that 
$\limsup_{\ga \to \ga_0} T^{\min}_{\wt \ka (\ga)} (\eta_-,\eta_+) \le t_0$, 
where $t_0 = T^{\min}_\ka (\eta_-,\eta_+) $.

Let $\vep_0 (\cdot) $  be a control that steers the system (\ref{e:R}), (\ref{e:R2})
from $x(s_-) = \eta_-$ to $\eta_+$ in the minimal time 
$t_0 $ for the spectral parameter $\ka$.
Since $f$ and $\wt f$ in (\ref{e:R}), (\ref{e:R2}) are analytic in $\ka$, 
one sees that $x_{\eta_-} (s_-+t_0, \wt \ka(\ga), \vep_0) \to \eta_+$ as $\ga \to \ga_0$
(we use the notation of Theorem \ref{t:STLC}).
For small enough $|\ga - \ga_0|$, let us define $t_1 (\ga)$ by 
$| \eta_+ - x_{\eta_-} (s_-+t_0,\wt \ka(\ga), \vep_0) | = \frac12 
\ra^{\max}_{\Om,\eta_+} (t_1 (\ga))$ assuming that $\Om$
is a small enough neighborhood of $\ka $, where the uniform STLC holds 
by Theorem \ref{t:STLC}.
Then $T^{\min}_{\wt \ka(\ga)} (\eta_-,\eta_+) < t_0 + t_1 (\ga)$.
From $\ra^{\max}_{\Om,\eta_+} (0) = 0$ and the continuity of $\ra^{\max}_{\Om,\eta_+} $,
one gets $\lim_{\ga \to \ga_0} t_1 (\ga) = 0$, and, in turn, 
$\limsup_{\ga \to \ga_0} T^{\min}_{\wt \ka (\ga)} (\eta_-,\eta_+) \le t_0$.

(ii) follows from (i) and Corollary \ref{c:rho(ga)} (iv).
\end{proof}

\begin{proof}[Proof of Theorem \ref{t:PF}]
The description of $\dom \rho_{\min}$ is given in Corollary \ref{c:rho(ga)}. Statement (i) follows from Proposition
\ref{p:star-like}. Statement (ii) from statement (i).
\end{proof}

\begin{proof}[Proof of Theorem \ref{t:MinDec2}.]
Proposition \ref{p:Bd->MinDec} implies $\rho_0 \le |\ka_0| \le \rho_1$. So 
$\ka_0 \in \ee^{\ii \ga_0} [\rho_0 ,\rho_1] $.
If $\ka = \rho \ee^{\ii \ga_0}$ with $\rho \in [\rho_0, |\ka_0|]$, Corollary \ref{c:rho(ga)} implies 
that $\ka$ is of minimal decay. To obtain the same statement for $\rho \in (|\ka_0|,\rho_1)$ 
(in the case where this interval is nonempty), one can combine Propositions 
\ref{p:star-like}-\ref{p:Bd->MinDec}.

In the case (i), (\ref{e:ka01}) and (\ref{e:ka02}) are obvious because $\ka$ is of minimal modulus.

The case (ii) means that $\ka$ is of minimal decay, but is not of minimal modulus.
Let $t_0 := T^{\min}_{\ka} (\eta_-,\eta_+)$ and $s_0 :=s_- + t_0$.
Consider any control $\vep_{\ga_0}^{\min} (\cdot) \in \F_{s_-,s_0}$ 
that steers the system (\ref{e:R}),(\ref{e:R2}) from $x(s_-) = \eta_-$ to $\eta_+$ 
in the minimum possible time $t_0$.
Since $\ka$ is not of minimal modulus, (\ref{e:Tmin=rho}) implies $t_0 < s_+ -s_-$.
The control $\wt \vep \in \F_{s_-,s_+}$ defined by (\ref{e:ResMinDecCont}) is a continuation of 
$\vep_{\ga_0}^{\min} (\cdot) $ to the interval $(s_0,s_+)$ by the constant 
value $\eta_+^2$. Since $\eta_+$ is an equilibrium solution to  
$x' = \ii \ka (-x^2 + \eta_+^2)$, we see that the trajectory of 
$\wt x(s)$ corresponding to the control $\wt \vep$ stays 
at $\eta_+$ for $s \in (s_0,s_+)$. Thus, 
$\ka \in \Si^{s_-,s_+}_{\eta_-,\eta_+} (\wt \vep)$ and so $\wt \vep$ is the resonator of minimal decay 
for the frequency $\re \ka$.

It remains to show that there exists a solution $y$ to the equation (\ref{e:Eys}) 
on $[s_-,s_+]$ such that $(\wt x,\wt \vep, 0, y)$ is an extremal tuple on $[s_-,s_+]$.
Indeed, since $\wt \vep$ is the resonator of minimal decay generating $\ka$,
it follows from \cite{KLV17} that there exists a nontrivial solution $y$ to (\ref{e:Eys}) 
such that $\wt \vep (\cdot) = \E (y) (\cdot)$ on $(s_-,s_+)$ and the two boundary conditions
(\ref{e:BCxpm}) are satisfied. It follows from Section \ref{ss:MinTime} that 
$\wt x  = \frac{y'}{\ii \ka y}$ is a solution to (\ref{e:R}), (\ref{e:R2}) (with $\vep(\cdot) = \wt \vep (\cdot)$) on $[s_-,s_+]$. 
It is easy to see from (\ref{e:Eys}) that 
$\im (\wt \vep (s) y^2 (s)+\ka^{-2} (y' (s))^2) = \im (y^2 (s) [\wt \vep (s) - x^2 (s)])$ 
is a constant independent of $s$. To see that this constant (and so $\la_0$) are equal to $0$, it is enough to take any $s \in (s_0,s_+)$.
\end{proof}

\emph{Let us pass to the settings of Theorem \ref{t:ExOpt2} and give its proof.}

To prove (\ref{e:=rho=k<}) it is enough to show that $\rho_0 < \rho_1$.
Indeed, the example considered in this theorem can be seeing as `the right half' of the symmetric example of 
Theorem \ref{t:ExOpt} with $\n_\infty =\n_1$. This implies $\rho_0 = |\ka_0|$.
On the other side $\rho_0 < \rho_1$, implies $\rho_0 < \lim_{\ga \to \ga_0 + 0} \rho_{\min} (\ga) $,
and so $\lim_{\ga \to \ga_0 - 0} \rho_{\min} (\ga) = \rho_0$ follows 
from Corollary \ref{c:rho(ga)}. 

\emph{Let us show that $\rho_0 < \rho_1$ holds} using the method of the proof of Theorem \ref{t:ExOpt}
and its symmetric settings. 
For this we put $\wt s_+ = - \wt s_- = s_+$, $\wt \eta_- = -\n_1$, 
and $\wt \ka (\ga) = |\ka_0| e^{\ii \ga}$ for $\ga \in (\ga_0,0)$.
From $\rho_0 = |\ka_0| = |\wt \ka (\ga)|$ and the equalities
$T^{\min}_{\ka_0} (-\n_1,\infty)  = s_+  $,
$T^{\min}_{\wt \ka (\ga)} (-\n_1,\infty)     
=\frac{s_+ \rho_{\min}  (\ga , -\n_1, \infty)}{|\wt \ka (\ga)|} 
$, 
it is easy to see that $\rho_0<\rho_1$ is equivalent to 
\begin{align}
\text{$T^{\min}_{\ka_0} (-\n_1,\infty) \cos \ga_0 < \inf_{\ga \in (\ga_0,0)} \left( T^{\min}_{\wt \ka (\ga) } (-\n_1,\infty) \cos \ga \right)$.}
\end{align}
This inequality follows from 
$ 
T^{\min}_{\ka_0} (-\n_1,\infty)  =   \frac{\pi}{2 \n_2 \rho_0 \cos \ga_0} 
$
and the following lemma.

\begin{lem}
There exists $\de_1>0$ such that 
$\frac{\pi}{2 \n_2 \rho_0 \cos \ga} + \de_1< t_1 $
for every $\ka = \wt \ka (\ga)$ with $\ga \in (\ga_0,0)$ and 
for every extremal tuple $(x,\vep,\la_0,y)$ on $[s_-,s_-+2 t_1]$ satisfying $x(s_-)=-\n_1$
and  
 $x(s_-+t_1) = \infty$.
\end{lem}
\begin{proof}
Let $(x,\vep,\la_0,y)$ be an extremal tuple corresponding to 
$\ka = \wt \ka (\ga)$ with $\ga \in (\ga_0,0)$ and satisfying $x(s_-)=-\n_1$, 
 $x(s_-+t_1) = \infty$.

\emph{Case 1.} Assume that $(x,\vep,\la_0,y)$ is abnormal. Then it 
corresponds to the sequence (\ref{e:AS3}) with $m \in \NN$ repetitions (due to $\ga \in (\ga_0,0)$, 
the case (\ref{e:onelayerodd}) can easily excluded by simple calculations using the ilog-phase $\vth_{\n_2^2} (\cdot)$, see Section \ref{ss:const}).
Then Theorem \ref{t:AbnExtWidth} implies $t_1 > \frac{\pi (m+1)}{2 \n_2 \rho_0 \cos \ga}$,
and so, implies $\frac{\pi}{2 \n_2 \rho_0 \cos \ga} + \de_1< t_1 $.

\emph{Case 2.} Assume that $(x,\vep,\la_0,y)$ is normal. Since $x(s_-)=-\n_1$ and $x(s_-+t_1) = \infty$,
we see that only the types (\ref{e:NS1}) and (\ref{e:NS5}) of extremals are possible.  
Arguments similar to that of 
the proof of Theorem \ref{t:ExOpt} imply that $m_1 \ge 1$.
Using Corollary \ref{c:Norm} it is easy to show that the part 
$\scriptstyle \overset{4}{\EP} _{(-\n_2,0)} \overset{3}{\EP}$ requires the time $t_2 > \frac{\pi}{2 \n_2 \re \ka } 
= \frac{\pi}{2 \n_2 \rho_0 \cos \ga}$ (statements (i) and (iv) of Corollary \ref{c:Norm} 
are needed for the cases (\ref{e:NS1}) and (\ref{e:NS5}), respectively).
The starting part $\scriptstyle  _{(-\n_2,-\n_1]} \overset{3}{\EP} \overset{3}{\EM} _{(-\infty,-\n_1)} \overset{4}{\EM}$ requires certain time $t_3>0$.
This time $t_3$ can be bounded from below by a positive number 
$\de>0$ that does not depend on $\ga \in (\ga_0,0)$ and on the choice of the extremal tuple. 
To prove this fact, one can use the arguments of \cite{S83} about the absence of STLC to 
a singular equilibrium point, and make them uniform over $\ga \in (\ga_0,0)$ using the additional information 
about the structure of normal extremals provided by Theorem \ref{t:NextWidth}. This complete the proof of the lemma
and of the fact that $\rho_0<\rho_1$.
\end{proof}

\emph{Let us prove statement (i) of Theorem \ref{t:ExOpt2}.}
The facts that $\ka_0$ is  $(\eta_-,\eta_+)$-eigenvalue of minimal modulus and that 
the constant function $\vep^{\min}_{\ga_0} (\cdot)$ 
is the unique control in $\F_{0,s_+}$ such that $\ka_0 \in \Si^{0,s_+}_{\eta_-,\eta_+} (\vep^{\min}_{\ga_0})$
follow from Theorem \ref{t:ExOpt} and the symmetry w.r.t. $0$.
The fact that $\ka_0$ is $(\eta_-,\eta_+)$-eigenvalue of minimal decay follows from $\rho_0 < \rho_1$
and Proposition \ref{p:Bd->MinDec}.

\emph{Statement (ii) of Theorem \ref{t:ExOpt2}} now follows from Theorem \ref{t:MinDec2}.

\section{Numerical minimum-time method}
\label{s:N}

The derivation of the HJB equation in Section \ref{ss:HJB} 
for the problem of minimization of the resonator length makes it possible to apply 
the Dynamic Programming approach (see \cite{F97}) to approximations of equation (\ref{e:HJB}).

On the other hand, the point of view of minimum-time control allows us to improve substantially in the case of symmetric resonators 
 the shooting method for (\ref{e:Eys}) passing from 
the search of zeroes of a $\CC$-valued characteristic determinant $F_\nl (\xi,\ka)$ on a fixed interval \cite{KLV17}
to much cheaper and accurate process of finding a turning point $p_0$ of an $x$-extremal  as a zero of the $\RR$-valued monotone function $G_0 $ (see Lemma \ref{l:MonotFunc}).

In the numerical experiments of Sections \ref{ss:shooting}-\ref{ss:Pareto}
we show that this makes the shooting method effective.
In particular, for 
the values of the quality factor $Q=10^6$ and $Q=1.1 \cdot 10^6$ 
and the realistic permittivity constraints 
$ \ep_1 = 1$ (vacuum), $ \ep_2 = \ep_\infty = 11.9716 $ (silicon),
we compute symmetric resonators of minimal length 
that have 27 layers (see Fig. \ref{fig3}). 
The computed widths of layers have 11 significant digits.

Varying  $\Arg \ka$ with a fixed $|\ka|$, we compute in  Section \ref{ss:Pareto}
 parts of the two Pareto optimal frontiers of minimal decay and of minimal modulus for 
 the case $ \ep_1 =  \ep_\infty = 1$, $ \ep_2  = 11.9716 $, 
and investigate numerically the effect when the latter frontier have a jump described 
by Theorem \ref{t:ExOpt2} (see Fig. \ref{f:Par}).

The following analytical result is the base for 
our computations.
Let $\ka \in \CC_4$ and the left end $s_-$ of the resonator be fixed.
For $\xi \in \RR$,  let $\Th_\xi (s)$ be a solution to 
(\ref{e:Eys})
satisfying the initial conditions 
\begin{equation}
\Th_\xi (s_-) = \ee^{\ii \xi},  \ \ \  \pa_s \Th_\xi (s_-) = -\ii \ka \n_\infty \   e^{\ii \xi}  
\label{e:Th}
\end{equation}
and so satisfying (\ref{e:BCpm}) at the left end $s_-$.
Such a solution exists and is unique due to \cite[Theorem 6.1]{KLV17}.
Note that $\Th_\xi (\cdot)$ is the solution to (\ref{e:ep}) with 
$\vep (\cdot) = \E (\Th_\xi ) (\cdot)$. Let $p (\xi)$ be 
the turning point $p_0$ of $\Th_\xi (\cdot)$ if it exists, otherwise we put $p_0 (\xi) = +\infty$
(note that in each of these cases $\p_0 (\xi) > s_-$).
Let 
\begin{multline} \label{e:A}
l (\xi) := p (\xi) - s_- , \qquad 
\A^\odd := \{ \xi \in [0,\pi] \, : \, \Th_\xi (\p(\xi)) = 0 \} , \text{ and }\\
 \qquad
\A^\even := \{ \xi \in [0,\pi]\, : \, \pa_s \Th_\xi (\p (\xi)) = 0 \} .
\end{multline}

\begin{cor} \label{c:Lp0}
Let $\ka \in \CC_4$ and $\xi_0 \in [0,\pi]$.
Then the following statements are equivalent:
\item[(i)] $\xi_0 \in \argmin_{\xi \in \A^{\odd (\even)}} l (\xi)$;
\item[(ii)] $\Lr_{\min}^{\odd (\even)} (\ka) = 2 l (\xi_0)$ and $\xi_0 \in \A^{\odd (\even)}$;
\item[(iii)] the resonator $\vep (s) := \left\{ \begin{array}{ll} 
\E (\Th_{\xi_0} ) (s),  & s \in [s_- -p(\xi), p (\xi) -s_-] \\
\n_\infty^2 , & s \in \RR \setminus [s_- - p (\xi), p (\xi) -s_-] 
\end{array} \right.
$
is  a minimizer for the odd-mode (resp., even-mode) problem in (\ref{e:argminOddEven}).
\end{cor}

\begin{proof} 
By Theorem \ref{t:PMP}, minimizers $\vep (\cdot)$ to the problems 
(\ref{e:argminOddEven}) have on the interval $(s_-^{\vep},s_+^{\vep})$ 
the form 
$\E (y) (\cdot)$ with $y = c \Th_{\xi}$, where $c \in \RR $ and $\xi \in [0,\pi)$, 
assuming that $s_-=s_-^{\vep}$ (by technical reasons, it is convenient sometimes 
to consider the value $\xi =\pi$ instead of $\xi =0$, see Remark \ref{r:pi/2pi}). Proposition \ref{p:eqv} and the fact that 
$\E (c \Th_{\xi}) (\cdot) \equiv \E (\Th_{\xi}) (\cdot)$ complete the proof.
\end{proof}

 \subsection{Numerical experiments with the shooting method}
 \label{ss:shooting}
 
Symmetric resonators of minimal length for a given $\ka \in \CC_4$
can be obtained from Corollary \ref{c:Lp0}  computing
$\argmin\limits_{\xi \in \A^{\odd (\even)}} l (\xi)$ 
with the use of the shooting method for the Euler-Lagrange equation.
The values of $\xi$ pass through the discretized version $\{ n \pi/\nu \ : \ 0 \le n \le \nu \}$ 
of the interval $[0,\pi]$ with large enough $\nu$.
For each particular $\xi \in [0,\pi]$, the initial values for shooting are given by (\ref{e:Th}), the position of the turning point 
$p_0 (\xi)$ is computed numerically as the zero of the monotone function $G_0 (\cdot)$ from 
Lemma \ref{l:MonotFunc} taken for $y(\cdot) = \Th_\xi (\cdot)$.
This gives the approximate value $\wt l (\xi) $ of the function $l (\xi)$.

\vspace{-1ex}
\begin{figure}[ht]
\noindent \centering
\begin{minipage}[t]{0.03\linewidth}
\vspace{5ex}
\footnotesize \textbf{a)}

\vspace{15ex}
\textbf{b)}
\end{minipage}
\begin{minipage}[t][][b]{0.60\linewidth}
\centering
\includegraphics[width=\textwidth]{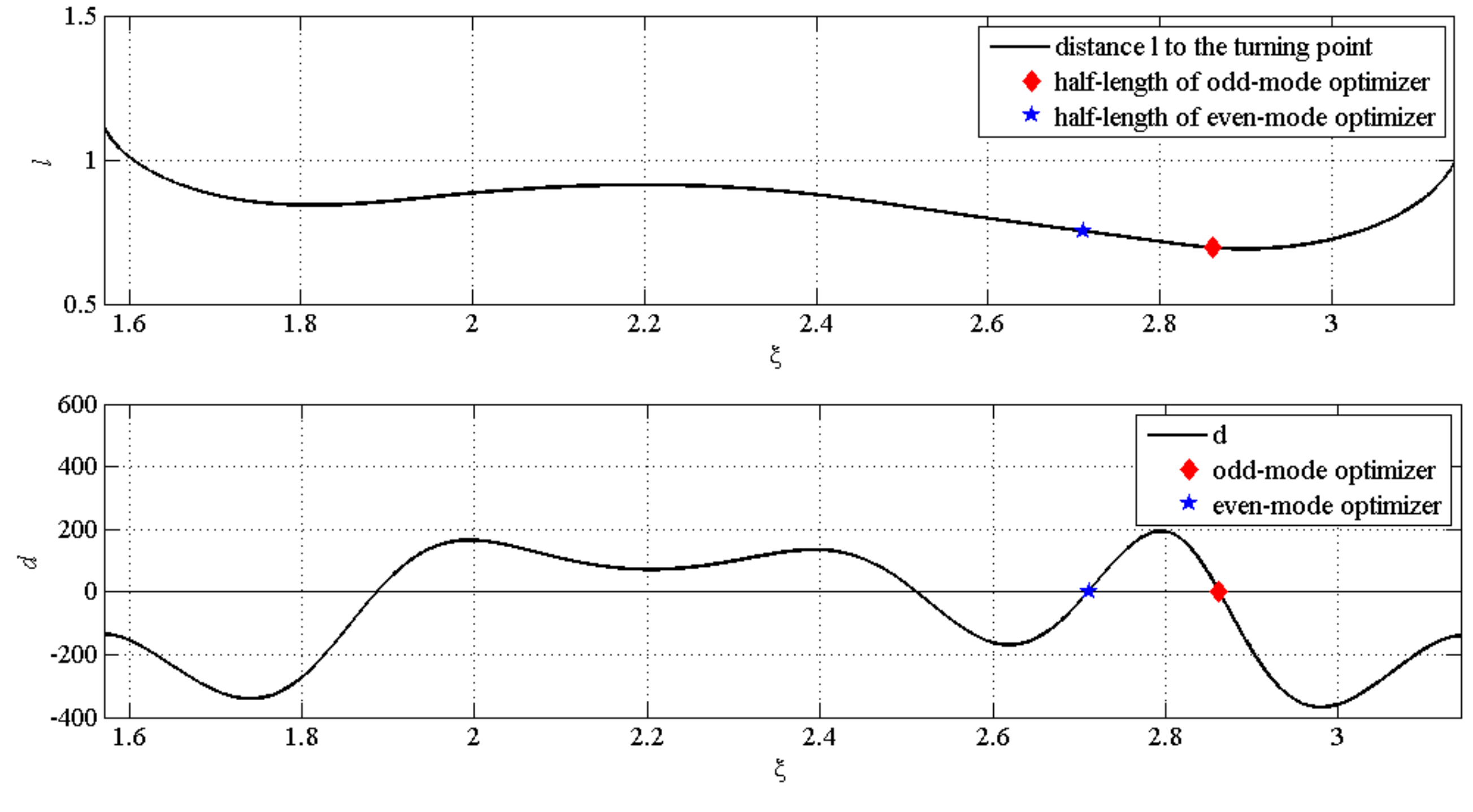}
\end{minipage}\\[2ex]
\vspace{-1ex}
\caption{\footnotesize \textbf{(a)} The distance $l=p(\xi)-s_-$ between the left end $s_-$ of the resonator and the turning point $p(\xi)$
for $\ka=\ka_1$ given by (\ref{e:ka1}). \ 
\textbf{(b)} The value of the corresponding product 
$d = \overline{\Th_\xi (p (\xi))} \pa_s \Th_\xi (p (\xi))$.
The zeros of $d (\cdot)$ in $[\pi/2,\pi]$ give the values of $\xi$ that correspond 
to symmetric w.r.t. $p(\xi)$ resonators $\vep (\cdot) = \E (\Th_\xi) (\cdot)$.
The marks `$\blackdiamond$' and `$\star$' correspond to the values of $\xi$ for 
the odd-mode and even-mode resonators of minimal length for $\ka_1$.
The symmetric resonator of minimal length has an odd mode. 
\label{fig1}}
\end{figure}

The iterative computation of switch points of $\E (\Th_\xi) (\cdot)$ and computation 
of $\Th_\xi (\cdot)$ does not require finite difference approximation because on 
each interval $(b_j,b_{j+1})$ of constancy of $\E (\Th_\xi) (\cdot)$, the explicit analytic expression of $\Th_\xi (\cdot)$
via the values of $\Th_\xi (b_j)$ and $\pa_s \Th_\xi (b_j)$ at the left end-point $b_j$ are available.

 The shooting provides also the discretized versions of the functions 
$d_0 (\xi) = \overline{\Th_\xi (p_0 (\xi))}$ and $d_1 (\xi) = \pa_s \Th_\xi (p_0 (\xi))$, 
the sets of zeroes $\{ \xi_j^\odd \}_{j=1}^{J_\odd}$ and $\{ \xi_j^\even \}_{j=1}^{J_\even}$
of which are the sets $\A^{\odd}$ and $\A^{\even}$ defined by (\ref{e:A}).
Practically, we find first the approximate zeroes $\wt \xi_j$, $j=1,\dots, J$, of the product of the two above functions
$d (\xi) = d_0 (\xi) d_1 (\xi) = \overline{\Th_\xi (p (\xi))} \pa_s \Th_\xi (p (\xi))$. 
So $\{\wt \xi_j \}_{j=1}^{J}$
approximate the set $\A^\sym := \A^{\odd} \cup \A^{\even}$. 

\begin{rem} Note that for every $\xi$ such that $p(\xi) <+\infty$ the value $d (\xi)$ is real. This follows from the fact 
that $p(\xi)$ is the turning point for the corresponding solution. While it follows from Example \ref{ex:Perx-extr} (ii)
that for certain $\n_\infty \not \in [\n_1,\n_2]$ and $\xi \in [0,\pi)$, one has
$p(\xi) = +\infty$, all numerical experiments of this section assume that 
$\n_\infty = \n_1$ or $\n_\infty = \n_2$ and we do not observe in them the case $p(\xi) = +\infty$.
\end{rem}

\vspace{-2ex}
\begin{table}[h]
\footnotesize
\centering
\caption{Odd- and even-mode resonators of minimal length for $\ka=\ka_1$.}
\vspace{-2ex}
\begin{center}
\label{tab1}
\begin{tabular}{|c|c|c|c|c|c|c|c|}\hline
\multicolumn{4}{|c|}{Odd-mode optimal resonator}&\multicolumn{4}{|c|}{Even-mode optimal resonator}\\
\cline{1-8}
$n$    & $\vep$ & Layer left & Layer width            & $n$    & $\vep$        & Layer left & Layer  width  \tabularnewline
       &             &edge $b_n$, $\mathrm{\mu}$m & $W(n)$,$\mathrm{\mu}$m&        &             &edge $b_n$, $\mathrm{\mu}$m& $W(n)$, $\mathrm{\mu}$m\\ \hline
-3     &$\epsilon_1$ &-0.6946797015&0.2121868748       &-5      &$\epsilon_1$ &-0.7521786176&0.1512335680 \\ \hline
-2     &$\epsilon_2$ &-0.4824928267&0.1437719651       &-4      &$\epsilon_2$ &-0.6009450496&0.1514248849 \\ \hline
-1     &$\epsilon_1$ &-0.3387208616&0.1875186766       &-3      &$\epsilon_1$ &-0.4495201647&0.1603442551 \\ \hline
0      &$\epsilon_2$ &-0.1512021850&0.3024043700       &-2      &$\epsilon_2$ &-0.2891759096&0.1568410635 \\ \hline
1      &$\epsilon_1$ &0.1512021850 &0.1875186766       &-1      &$\epsilon_1$ &-0.1323348461&0.0599824929 \\ \hline
2      &$\epsilon_2$ &0.3387208616 &0.1437719651       &0       &$\epsilon_2$ &-0.0723523532&0.1447047064 \\ \hline
3      &$\epsilon_1$ &0.4824928267 &0.2121868748       &1       &$\epsilon_1$ &0.0723523532 &0.0599824929 \\ \hline
4      &             &0.6946797015 &                   &2       &$\epsilon_2$ &0.1323348461 &0.1568410635 \\ \hline
       &             &             &                   &3       &$\epsilon_1$ &0.2891759096 &0.1603442551 \\ \hline
       &             &             &                   &4       &$\epsilon_2$ &0.4495201647 &0.1514248849 \\ \hline
       &             &             &                   &5       &$\epsilon_1$ &0.6009450496 &0.1512335680 \\ \hline
       &             &             &                   &6       &             &0.7521786176 &             \\ \hline
\multicolumn{4}{|c|}{Total length $\Lr_{\min}^{\odd} (\ka_1)$ = 1.389359403 = $\Lr_{\min}^{\sym} (\ka_1)$}&
\multicolumn{4}{|c|}{Total length $\Lr_{\min}^{\even} (\ka_1)$= 1.5043572352} \\ \hline
\end{tabular} 
\end{center}
\vspace{-1ex}
\end{table}

Finding $\wt \xi_{j_*} $  such that  
\begin{align} \label{e:minwtL}
\wt l (\wt \xi_{j_*} ) = \min_{1\le j \le J} \wt l (\wt \xi_{j}),
\end{align} 
we obtain 
the approximate value $2 \wt l (\wt \xi_{j_*})$ for the minimal length $\Lr_{\min}^\sym (\ka)$ 
of a symmetric resonator generating the desired resonance $\ka$. 
Similarly, we find the approximate values $\wt \xi_{j_*^\odd}$ and $\wt \xi_{j_*^\even}$ 
of the parameter $\xi$ that correspond to the minimizers
of odd-mode and even-mode problems (\ref{e:argminOddEven}).

\begin{rem}
In the numerical experiments of this subsection, we observe only one minimizer in the discretized 
minimization problem (\ref{e:minwtL}),
while we expect that there are exceptional values of $\ka \in \CC_4$ so that the original problem (\ref{e:argminFsym})
has more than one minimizer (see Section \ref{s:dis}).
In the experiment of the next subsection for the value $\ka = \ka_0$, 
the situation is different because the exists and interval of values of $\xi$ where  $l (\xi)$ achieve its minimal value
(this follows from Remark \ref{r:NSinfty}). However, all these values of $\xi$ lead to the same 
symmetric resonator of minimal length.
\end{rem}

In Physics settings, equation (\ref{e:ep}) has the form 
$y''(s)=-\frac{\om^2}{c^2}\vep (s) y^2(s)$, where 
$c$ 
is the speed of light in vacuum and 
$\om$ is the complex frequency.
These means that $\re \om$ is the angular frequency  and $(-\im \om)$ is the rate of decay
of the corresponding eigenoscillations $e^{-\ii \om t} y(s)$ of the electric field in the cavity.
So $\ka = \frac{\om}{c}$ is generalized complex wave number 
(in the settings of previous sections it was assumed that the system of units is such that $c=1$, we do not keep
keep this convention here).

As before $\vep (s)$ ($\ep_\infty$, $\ep_1$, and $\ep_2$) is the dielectric permittivity 
of the cavity at the layer containing the point $s$ (resp., 
of the homogeneous outer medium and of the two extreme permittivities 
allowed in the optimization process). The numbers $\n_j = \ep_j^{1/2}$, $j=1,2,\infty$, are refractive indices in the corresponding
materials. In this section we  take that of air $ \n_1 = 1$ and of silicon $ \n_2 = 3.46$
(the magnetic permeability of all involved materials is assumed to be equal to 1). 
In this subsection we assume that the outer homogeneous 
medium consists of silicon, i.e., $\ep_\infty = \ep_2$ and $\n_\infty = \n_2$. 

\begin{rem} \label{r:pi/2pi}
Since $\ep_\infty = \ep_2$, the values of $\xi \in (0,\pi/2]$ cannot correspond to 
a resonator of minimal length (otherwise the first layer has the same permittivity as the outer medium).
The computations in this subsection can be restricted to the interval $\xi \in (\pi/2,\pi]$.
\end{rem}

First, we give the results of a numerical experiment for 
the case of relatively small quality factor value $Q = 50$ because this leads to a smaller 
number of layers in the resonator and it is easier to explain for this example 
our notation and the way of presentation of the numerical results.
The symmetric resonator of minimal length is calculated for the value 
\begin{gather} \label{e:ka1}
\kappa_1 = 3.653014713476505\cdot 10^{6} - \ii \; 3.653014713476505\cdot 10^{4},
\end{gather}
which corresponds to $\omega_1 = A_1 - \ii B_1$, where  
$A_1 = 1.095146260063287\cdot 10^{15}$
is the angular frequency and 
$B_1 =  A_1 \cdot 10^{-2}$ 
is the rate of decay, and  
to the wavelength in vacuum $\lambda_1 = 1720$ nm
(which is in the infrared range).

\begin{figure}[h]
\begin{minipage}[t]{0.49\linewidth}
\centering \includegraphics[width=\textwidth]{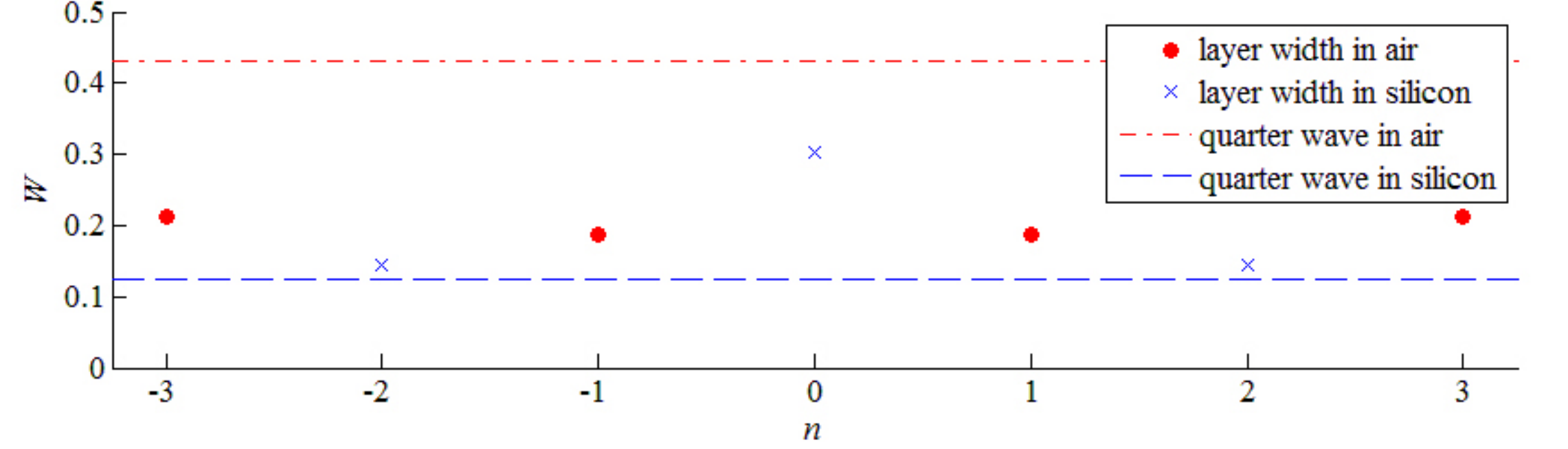}\\
\footnotesize \textbf{a)}
\end{minipage}
\begin{minipage}[t]{0.49\linewidth}
\centering
\includegraphics[width=\textwidth]{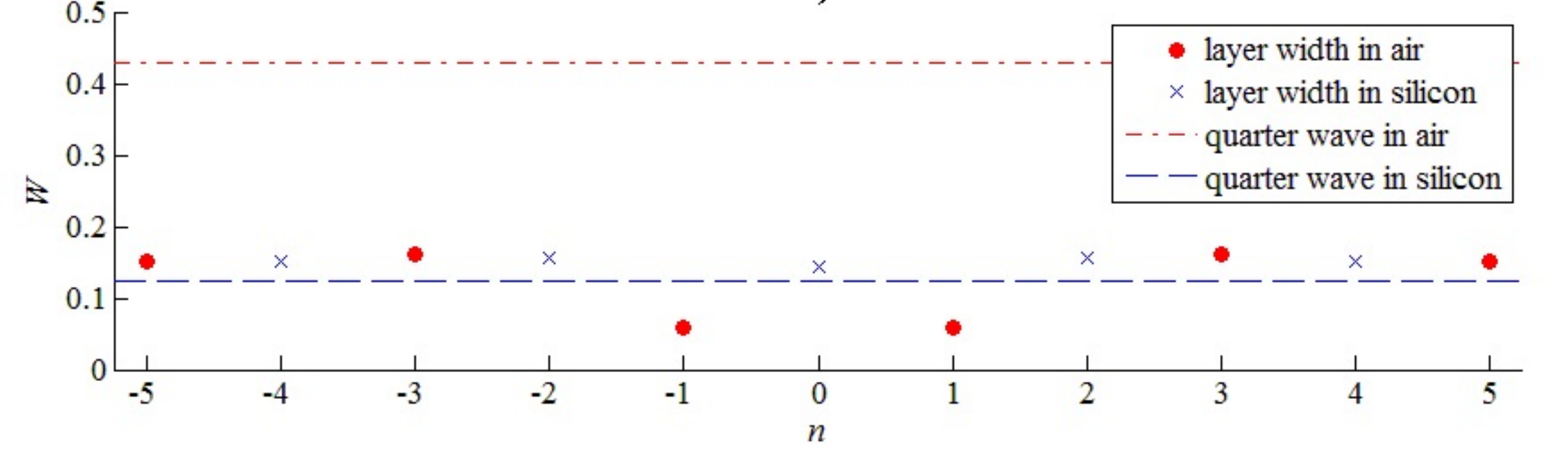}
\footnotesize \textbf{b)}
\end{minipage}
\vspace{-2ex}
\caption{\footnotesize Layers' widths for (a) the odd-mode and (b) the even-mode  optimal resonators for $\ka_1$.\label{fig2}}
\end{figure}

The graphs of computed approximated values of the functions $L (\xi)= p (\xi) - s_-$, $\xi \in [\pi/2,\pi]$,
and $d (\xi) = \overline{\Th_\xi (p (\xi))} \pa_s \Th_\xi (p (\xi))$
are given in Fig. \ref{fig1} (a) and (b), respectively.
We are interested only in the approximate values of $\xi$, where $d (\xi)  = 0$
since only these values correspond to resonators symmetric w.r.t the turning point $p (\xi)$.
Among such values the two values of $\xi$ corresponding to the two shortest resonators are marked.
It occurs that one of them is odd-mode resonator and the other is even-mode resonator.
So these two resonators give the numerical solution to the two corresponding odd- and even-mode 
minimum length problems in (\ref{e:argminOddEven}).

The odd-mode minimal length is reached at $\xi_*^{\odd} \approx 2.861410149502555$, the even-mode minimal length at $\xi_*^{\even} \approx 2.709972790535748$.
The corresponding to $\xi_*^{\odd}$ resonator $\vep^{\odd}_* (\cdot)$ has 
the length $\Lr_{\min}^{\odd} (\ka_1) = 2 l (\xi_*^{\odd}) \approx 1.389359403$ $\mu$m
and consists of $N = 7$ layers (excluding the outer medium $(-\infty,s_-) \cup (s_+,+\infty)$, where $s_+ = 2 p (\xi) - s_-$).
The shortest even-mode resonator $\vep^{\even}_* (\cdot)$ has 
$N = 11$ layers and the length $\Lr_{\min}^{\even} (\ka_1) \approx 1.5043572352$ $\mu$m.

The structure of the function $\vep^{\odd}_* (\cdot)$ (the function $ \vep^{\even}_* (\cdot)$)
is shown in Table \ref{tab1} by the approximate values of its switch points $b_n$.
Note that the resonator 
is shifted on the $s$-line such that $(b_0, b_1)$ is the the central layer $I^{\cent}$ 
defined in Section \ref{s:Corol}, $s=0$ corresponds to 
the center $p_0$ of the resonator, $b_{-3} = s_-$ (resp., $b_{-5} = - s_-$) 
is the initial shooting point,  and $b_4 = s_+ = -s_-$ (resp., $b_6 = s_+ = - b_{-5}$) 
is the reflection of $s_-$ w.r.t. $0$ (the first and the last of the points $b_n$ 
correspond to the edges $s_\mp$  of the resonator, where $x$ takes the desired values $\mp \n_\infty$,
and so are not necessarily switch points of the corresponding extended $\vep$-extremal 
in the sense of Section \ref{s:Syn}).

The 2nd (resp., the 6s) column gives the values of $\vep(s)$ 
for $s$ in the layer $(b_n,b_{n+1})$, the 4th (resp., the 8s) column the layer widths $W(n) = b_{n+1} - b_n$.
A visualization of the layer widths $W(n)$ convenient from the point of view of their 
comparison with quarter wave lengths and each other is plotted in Fig. \ref{fig2}.
The optimal resonator for $\ka_1$ is of the odd-mode type with length $\Lr_{\min}^{\odd} = 1.389359403$ $\mu$m. 
There are no layers with width equal to the quarter wave in the corresponding 
material. 

The next two numerical experiments are performed in a similar way 
for much higher values $Q$ of the quality-factor,  $Q_2 = 10^6$ and $Q_3 = 1.1 \cdot 10^6$
with the real part of $\ka$ the same as in the previous experiment $\re \ka_2 = \re \ka_3 = \re \ka_1$.
The result for case $Q = 10^6$ (and so $\ka_2 = A_1 - \ii \frac{A_1}{2 \cdot 10^6} $)
is plotted in Fig. \ref{fig3} (a), which shows the layer widths $W(n)$ for the right half of the computed 
optimal symmetrical resonator and the width $W(0)  $ of the central layer $(b_0,b_1)$. 
The optimal symmetric resonator for $\ka_2$ 
is of even-mode type.

However, for slightly different quality factor $Q_3$ and the resonance 
$\ka = \ka_3 = A_1 - \ii \frac{A_1 }{2.2 \cdot 10^6} $ the optimal symmetric resonator is of the odd-mode type.
The corresponding widths of layers are shown in Fig. \ref{fig3} (b).

\begin{figure}[h]
\begin{minipage}[t]{0.495\linewidth}
\centering \includegraphics[width=\textwidth]{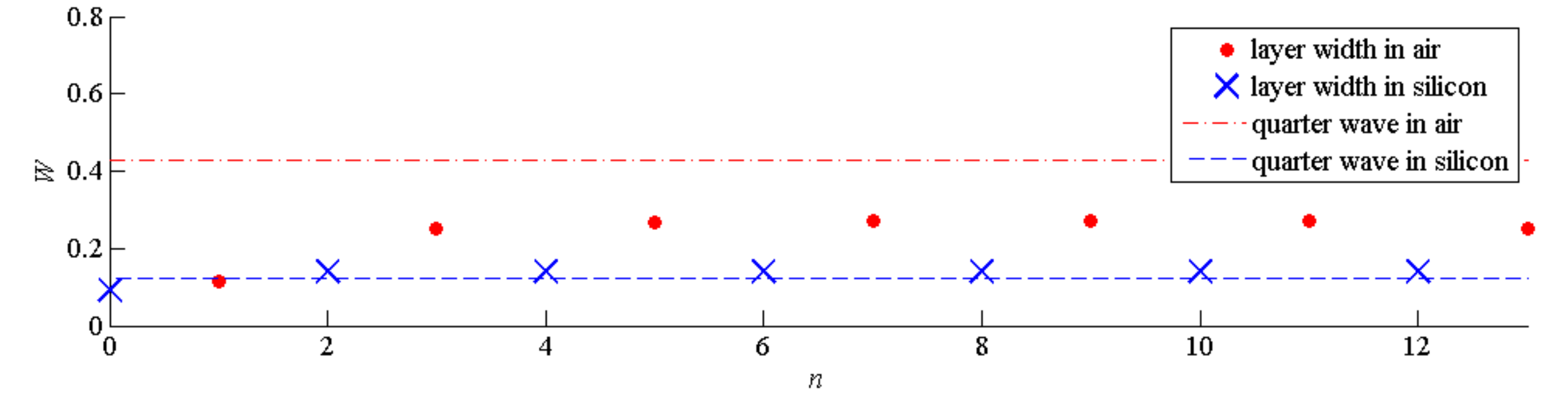}\\
\footnotesize \textbf{a)}
\end{minipage}
\begin{minipage}[t]{0.495\linewidth}
\centering
\includegraphics[width=\textwidth]{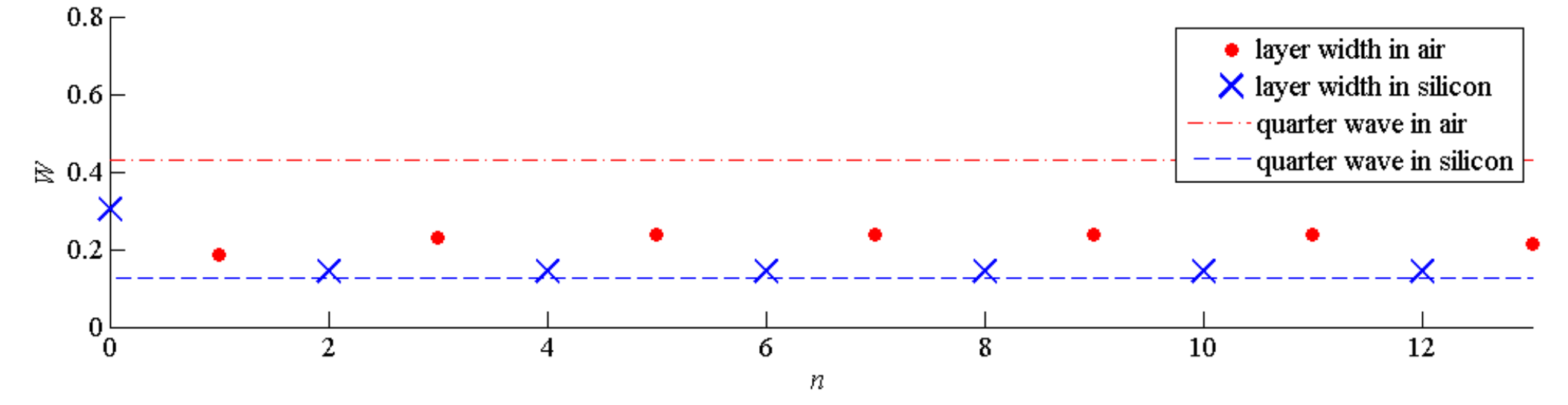}
\footnotesize \textbf{b)}
\end{minipage}
\vspace{-2ex}
\caption{\footnotesize Widths of layers for symmetric resonators of minimum length with high $Q$-factor: 
\textbf{(a)} $Q = 10^6$, $\ka=\ka_2$, the optimizer has an even mode; \ \textbf{(b)} $Q = 1.1 \cdot 10^6$, $\ka = \ka_3$, the optimizer has 
an odd mode.\label{fig3}}
\end{figure} 

\subsection{Computation of Pareto frontiers}
\label{ss:Pareto}

Varying $\Arg \ka \in (-\pi/2,0)$, finding the approximate value of 
$\Lr_{\min}^{\odd (\even)} (\ka)$ 
for each $\ka$ with the above shooting procedure, and using the rescaling (\ref{e:Tmin=rho}) to a fixed symmetric interval $[-\ell,\ell]$,
it is possible to compute parts of the even-mode  $\Pa^\even_{\md} := \{ \ee^{\ii \ga}\rho_{\min} (\ga, 0, \n_\infty)  :  \ga \in \dom \rho_{\min} (\cdot, 0, \n_\infty) \}$
and odd-mode  $\Pa^\odd_{\md} :=\{\ee^{\ii \ga} \rho_{\min} (\ga, \infty, \n_\infty)  :  \ga \in \dom \rho_{\min} (\cdot, \infty, \n_\infty) \}$
Pareto frontiers of minimal modulus (here $[s_-,s_+] = [0,\ell]$ is the right half of the symmetric resonator).

Taking $\n_1 = 1$, $\n_2 = 3.46$ as before, but $\n_\infty = \n_1$ (so that 
the outer medium is vacuum), we perform this computation near the jump of the odd-mode frontier described by Theorem \ref{t:ExOpt2} for the value $\ell = 0.1243$ $\mathrm{\mu}$m.
The result is plotted in Fig. \ref{f:Par}. 
Then, using Theorem \ref{t:PF}, its is possible to find 
the corresponding Pareto frontiers $\Pa_{\Dr}^{\even (\odd)}$  of minimal decay 
($\Pa_{\Dr}^{\odd}$ is drawn on Fig \ref{f:Par} (b)). 


\section{Conclusions and discussion}
\label{s:dis}

Analytic and numerical methods available in earlier studies have not given a clear answer about the structure of optimal resonators. 
On one side, the size modulated 1-D stack design, which was suggested in \cite{NKT08} 
on the base of simulations  with several modulation parameters and further studied in \cite{BPSZ11}, 
introduces defects to 27 silicon layers in the alternating periodic structure. 
The widths of these defects are given by 
a quadratic polynomial function of the layer's number counting from the center of the cavity. The width of each of 
these defects is small in comparison with the period of the original structure without defects. On the other side, resonances of 
periodic structures with a defect in the center were  considered in \cite{AASN03,KS08,HBKW08,LS13,OW13}. 
This defect is not necessarily small in comparison with the period.
The numerical experiments of \cite{KLV17} suggested that Pareto optimal resonators 
may involve combination of the both types of defects mentioned above. 

Consider now this question from the point of view of analytical and numerical results of the present paper.  
There are the following common features for the three experiments of Section \ref{ss:shooting}:
\begin{itemize}
\item[(i)] The central layer has the permittivity $\ep_2$ (silicon).
Combining this with the analytic result of Remark \ref{r:1/4stack}, we see that corresponding extremals are normal, and so, by Corollary 8.3, do not contain quarter-wave stacks
(this is confirmed also by the observation (iv) below).

\item[(ii)] For optimal odd-mode resonators, the central layer's width $|I^\cent| = W(0)$ is  approximately twice wider 
than the width $W (2 n)$, $1 \le  n \le (N-3)/4$, of ordinary silicon layers in the right half  of the resonator (and the widths of the layers symmetric to them w.r.t. $s^\cent$).
For the even-mode problem, we observe that $W(0) < W (2n)$, $1 \le  n \le (N-3)/4$.

\item[(iii)] For the even-mode problem, $W(1)$ is essentially narrower than the other regular air layers $W (2n-1)$, $2 \le  n \le (N-3)/4$.
\item[(iv)] Except  the case of $W(1)$ in (iii), widths of the other ordinary layers in each of materials vary very slightly, but are not equal to each other and are not 
equal to the quarter wave.
\end{itemize}
We have observed effects (i)-(iii) for the other computations of optimal odd- and even-mode resonators performed by the method described above.

It is known that for $\al = 0$ there exist values of $\n_1$, $\n_2$, and $\n_\infty \in (\n_1,\n_2)$ such that the problem 
\begin{gather} \label{e:argminDrSym}
\argmin\limits_{\substack{\ka \in \Si (\vep) , \  \re \ka = \al \\ 
\vep \in \F_{s_-,s_+}}} \Dr (\ka,\vep)  \quad \text{(where $\Dr (\ka, \vep) = - \im \ka$)}
\end{gather}
has two different optimizers $\vep_j (s) \equiv \n_j$, $j=1,2$ \cite{KLV17_MFAT}. However, the result of \cite{KLV17_MFAT} is  very exceptional because it requires a special relation  between the parameters $\n_1$, $\n_2$, $\n_\infty$, and because, in the case $\ka \in \ii \RR$, the equation (\ref{e:ep}) and boundary conditions (\ref{e:BCpm}) are real-valued and the trajectory $x(\cdot)$ of control system (\ref{e:R}), (\ref{e:R2}) that satisfies the initial condition $x(s^-) = - \n_\infty$ stays on $\wh \RR$ for all $s$. 
It also follows from \cite{KLV17_MFAT} that if $\al = 0$ and $\n_\infty \in \{\n_1,\n_2\}$ then (\ref{e:argminDrSym}) has a unique solution.

Up to know there are no results on the uniqueness or symmetry of minimizers for (\ref{e:argminDrSym}) if $\al \in \RR_+$ (see also the discussion in \cite{AK17}).
The uniqueness and the symmetry of optimizers $\vep$ are obviously connected, e.g., by the next statement.

\begin{prop} Let $\ell = s_+ = -s_-$.
Assume that $\ka \in \Pa_{\Dr}^\even \cap \Pa_{\Dr}^\odd$ and $\al = \re \ka$.
Then there exists at least two different (in $L^\infty (s_-,s_+)$-sense) minimizers for 
(\ref{e:argminDrSym}).
\end{prop}
\begin{proof}
Assume that $\ka$ is the resonance of minimal decay for $\alpha$ for problem (\ref{e:argminDrSym}). 
Since $\ka \in \Pa_{\Dr}^\even \cap \Pa_{\Dr}^\odd$, there exist 
optimizers $\vep_1 (\cdot)$ and $\vep_2 (\cdot)$ for the problems \linebreak
$
\argmin_{\substack{\ka \in \Si^\even (\vep) \\  \re \ka = \al,  \  
\ep \in \F_\ell^\sym}} \Dr (\ka,\vep)$ 
and 
$ \argmin_{\substack{\ka \in \Si^\odd (\vep) \\ \re \ka = \al, \ 
\vep \in \F_\ell^\sym}} \Dr (\ka,\vep)$, respectively.
Then $\vep_1$ and $\vep_2$ are optimizers for 
(\ref{e:argminDrSym}). Since $\Si^\odd (\vep) \cap \Si^\even (\vep) = \varnothing$ (see Remark \ref{r:sym}), we have $\vep_1 \neq \vep_2 $. Hence, we obtained the desired statement.

Assume that $(-\im \ka) > \beta_{\min} (\al)$ (i.e., $\ka$ is not a resonance of minimal decay for problem (\ref{e:argminDrSym})). Then an optimizer $\vep (\cdot)$ for (\ref{e:argminDrSym}) is not an even function. So $\wt \vep (s) := \vep (-s)$ is different from $\vep$. However, $\wt \vep $ is also an optimizer for (\ref{e:argminDrSym}). 
\end{proof}

In the computed figure Fig. \ref{f:Par} (a) we observe that $\Pa_{\Dr}^\even$  and $\Pa_{\Dr}^\odd$ intersect each other at certain points $\ka$ with $\re \ka >0$.
This provides a numerical evidence that,  in the case $\n_1 =\n_\infty =1$, $\n_2 =  3.46$, Pareto resonators of minimal decay corresponding to (\ref{e:argminDrSym}) are not necessarily unique for some of frequencies $\al>0$.


\quad\\
\noindent
\emph{Acknowledgments.} 
During various parts of this research, IK and IV 
were supported by the EU-financed projects
AMMODIT (grant agreement MSCA-RISE-2014-645672-AMMODIT),
IK was supported by the Alexander von Humboldt Foundation 
and by the VolkswagenStiftung project “Modeling, Analysis, and Approximation 
Theory toward applications in tomography and inverse problems”.
HK has been supported by the DFG through the Hausdorff Center for Mathematics.

\let\oldthebibliography\thebibliography
\let\endoldthebibliography\endthebibliography
\renewenvironment{thebibliography}[1]{
  \begin{oldthebibliography}{#1}
    \setlength{\itemsep}{0ex}
    \setlength{\parskip}{1ex}
}
{
  \end{oldthebibliography}
}

\end{document}